\pdfoutput = 1
\documentclass[11pt]{article}
\usepackage{enumerate}
\usepackage[OT1]{fontenc}
\usepackage[usenames]{color}
\usepackage{smile}
\usepackage[colorlinks,
            linkcolor=red,
            anchorcolor=blue,
            citecolor=blue
            ]{hyperref}
\usepackage{mathrsfs}
\usepackage{fullpage}
\usepackage{hyperref}
\usepackage[protrusion=true, expansion=true]{microtype}
\usepackage{float}
\usepackage{subfigure}
\usepackage{amsfonts,amsmath,amssymb,amsthm,url,xspace}
\usepackage{tikz}
\usepackage{verbatim}
\usetikzlibrary{arrows,shapes}
\usepackage{mathtools}
\usepackage{authblk}
\usepackage[bottom]{footmisc}
\usepackage{enumitem}
\usepackage{mathtools}
\usepackage{lscape}
\usepackage{caption}

\usepackage{algorithm}
\usepackage{algorithmicx,algpseudocode}
\usepackage[title]{appendix}

\ifx\counterwithout\undefined\usepackage{chngcntr}\fi
\counterwithout{equation}{section}

\allowdisplaybreaks[1]

\newcommand{\mR}{\mathbb{R}}

\newcommand{\mE}{\mathbb{E}}

\newcommand{\mL}{\mathcal{L}}

\newcommand{\mF}{\mathcal{F}}

\def\K{\mathcal{K}}
\def\P{{\mathcal P}}

\def\L{{\mathcal L}}

\def\0{{\boldsymbol 0}}

\def\N{{\mathcal N}}
\def\F{{\mathcal F}}
\def\A{{\mathcal A}}
\def\B{{\mathcal B}}
\def\C{{\mathcal C}}

\newcommand{\1}{\boldsymbol{1}}

\def\bv{{\boldsymbol{v}}}
\def\bs{{\boldsymbol{s}}}
\def\b1{{\boldsymbol{1}}}
\def\bx{{\boldsymbol{x}}}
\def\bu{{\boldsymbol{u}}}
\def\bd{{\boldsymbol{d}}}

\def\blambda{{\boldsymbol{\lambda}}}

\def\barH{\bar{H}}
\def\bartau{\bar{\tau}}

\def\bs{{\boldsymbol{s}}}

\def \b0{{\boldsymbol{0}}}
\def\barg{{\bar{g}}}
\def\barf{{\bar{f}}}
\def\barH{{\bar{H}}}
\def\bnabla{{\bar{\nabla}}}
\def\bargamma{{\bar{\gamma}}}

\def\barblambda{{\bar{\boldsymbol{\lambda}}}}
\def\barmu{{\bar{\mu}}}

\def\bartau{{\bar{\tau}}}
\def\barepsilon{{\bar{\epsilon}}}

\def\barL{{\bar{\mL}}}
\def\hatmu{{\hat{\mu}}}

\def\barK{{\bar{K}}}

\def\barbzeta{\bar{\boldsymbol{\zeta}}}
\def\bzeta{\boldsymbol{\zeta}}

\mathtoolsset{showonlyrefs=true}

\usepackage{xargs}
\usepackage[colorinlistoftodos,prependcaption,textsize=tiny]{todonotes}
\newcommandx{\unsure}[2][1=]{\todo[linecolor=red,backgroundcolor=red!25,bordercolor=red,#1]{#2}}
\newcommandx{\change}[2][1=]{\todo[linecolor=blue,backgroundcolor=blue!25,bordercolor=blue,#1]{#2}}
\newcommandx{\info}[2][1=]{\todo[linecolor=OliveGreen,backgroundcolor=OliveGreen!25,bordercolor=OliveGreen,#1]{#2}}
\newcommandx{\improvement}[2][1=]{\todo[linecolor=Plum,backgroundcolor=Plum!25,bordercolor=Plum,#1]{#2}}

\usepackage{alphalph}

\allowdisplaybreaks

\algrenewcommand{\algorithmiccomment}[1]{\hfill$\blacktriangleright$ #1}
\algnewcommand{\LongComment}[1]{\hfill$\triangleright$ #1}

\begin{document}

\title{Trust-Region Sequential Quadratic Programming for Stochastic Optimization with Random Models}

\author[1]{Yuchen Fang}
\author[2]{Sen Na}
\author[3]{Michael W. Mahoney}
\author[4]{Mladen Kolar}

\affil[1]{Department of Mathematics, University of California, Berkeley}
\affil[2]{School of Industrial and Systems Engineering, Georgia Institute of Technology}
\affil[3]{Department of Statistics, University of California, Berkeley}
\affil[4]{Data Sciences and Operations Department, University of Southern California}

\date{}

\maketitle

\begin{abstract}
	
In this work, we consider solving optimization problems with a stochastic objective and deterministic equality constraints. We propose a Trust-Region Sequential Quadratic Programming method to find both \textit{first- and second-order} stationary points. Our method utilizes a random model to represent the objective function, which is constructed from stochastic observations of the objective and is designed to satisfy proper adaptive accuracy conditions with a high but fixed probability. To converge to first-order stationary points, our method computes a gradient step in each iteration defined by minimizing a quadratic approximation of the objective subject to a (relaxed) linear approximation of the problem constraints and a trust-region constraint.
To converge to second-order stationary points, our method additionally computes an eigen step to explore the negative curvature of the reduced Hessian matrix, as well as a second-order correction step to address the potential Maratos effect, which arises due to the~\mbox{nonlinearity}~of~the~problem constraints.~Such an effect may impede the method from moving away from saddle points.~Both~\mbox{gradient}~and~eigen~step~\mbox{computations}~\mbox{leverage}~a~novel~\textit{parameter-free}~decomposition of the step and the trust-region radius, accounting for the proportions among~the~feasibility residual, optimality residual, and negative curvature. We establish \textit{global almost sure} first- and second-order convergence guarantees for our method, and present computational \mbox{results}~on~CUTEst 
problems, regression problems, and saddle-point problems to demonstrate its superiority over~existing line-search-based stochastic methods.

\end{abstract}

\section{Introduction}\label{sec:1}

We consider constrained stochastic optimization problems of the form:
\begin{equation}\label{Intro_StoProb}
\min_{\bx\in\mR^d}\;f(\bx)=\mE_\P[F(\bx;\xi)],\quad\;\;\;\text{s.t.}\quad c(\bx)=\b0,
\end{equation}
where $f:\mR^d\rightarrow \mR$ is a stochastic objective, $F(\cdot;\xi):\mR^d\rightarrow \mR$ is its realization, $c:\mR^d\rightarrow\mR^m$ are deterministic equality constraints, $\xi$ is a random variable following the distribution $\P$,~and~the~expectation is taken over the randomness of $\xi$. 
Throughout the paper, we assume that the objective~value~$f(\bx)$, together with its gradient $\nabla f(\bx)$ and Hessian $\nabla^2 f(\bx)$, cannot be exactly evaluated, but can be~estimated based on the samples $\{\xi_i\}$. Constrained stochastic problems are ubiquitous in various~scientific and engineering fields, including optimal control \citep{Betts2010Practical}, reinforcement learning \citep{Achiam2017Constrained}, portfolio optimization \citep{Cakmak2005Portfolio}, supply chain network design \citep{Santoso2005stochastic}, and physics-informed neural networks \citep{Cuomo2022Scientific}.

Numerous methods have been proposed to solve constrained \textit{deterministic} optimization problems, including penalty methods, augmented Lagrangian methods, and sequential quadratic programming (SQP) methods. While each type of method exhibits promising performance under favorable~settings, SQP methods have undoubtedly been very successful for solving both small- and large-scale~problems, particularly when the problems suffer significant nonlinearity \citep{Bertsekas1982Constrained, Boggs1995Sequential, Nocedal2006Numerical}. For problems with a \textit{stochastic} objective, several Stochastic~SQP~(SSQP) methods have also been developed recently \citep{Berahas2021Sequential, Berahas2023Stochastic, Berahas2023Accelerating, Berahas2023Sequential, Curtis2023Almost, Curtis2023Stochastic, Curtis2023Sequential, Curtis2024Stochastic, Qiu2023sequential, Na2022adaptive, Na2022Statistical, Na2023Inequality, Fang2024Fully}. We defer a detailed literature review to Section \ref{sec:1.1}.~However, existing SSQP~methods~primarily focus on first-order convergence, where the KKT residual is shown to converge to zero. This indicates that the methods may converge to saddle points or local maxima, which violate the goal of \textit{minimizing} the objective and are less meaningful for many problems. For example, in the context of deep~learning, converging to first-order stationary points can result in high generalization errors \citep{Dauphin2014Identifying, Choromanska2015Loss}.$\;\quad\quad$

In this paper, we address the above concern by designing the \textit{first} SSQP method with second-order convergence guarantees. We term our method Trust-Region Sequential Quadratic Programming for STochastic Optimization with Random Models (TR-SQP-STORM), as a \mbox{generalization}~of~the~STORM method in \cite{Chen2017Stochastic} to \textit{constrained} problems with \textit{second-order} guarantees. Our method~has the following promising features.

\begin{enumerate}[label=\textbf{(\alph*)},topsep=0pt, wide]
\setlength\itemsep{0.0em}
\item TR-SQP-STORM employs random models to represent the objective $f$ and its gradient and Hessian, which are constructed from stochastic estimates of those quantities. The random models enforce the estimates to satisfy proper \textit{adaptive} accuracy conditions with a high but fixed probability in each iteration. Moreover, the random models do not presume any parametric \mbox{distribution}~for~the estimates and allow for biased estimates, thereby accommodating various problem~settings and sampling mechanisms. With random model framework, our method adaptively~updates the \mbox{trust-region}~radius based on the ratio between predicted and actual model reductions, in a manner similar to deterministic trust-region methods.~As such, our method does not input any prespecified~trust-region~radius~(or stepsize) sequences that significantly affect algorithm performance \citep[see, e.g.,][]{Berahas2021Sequential, Berahas2023Stochastic, Berahas2023Accelerating, Berahas2023Sequential, Curtis2024Stochastic, Fang2024Fully}.

\item TR-SQP-STORM performs a trial step of two types, either a gradient step or an eigen~step. The gradient step reduces the KKT residual, while the eigen step increases the negative curvature of the reduced Lagrangian Hessian --- essentially, moving away from saddle points or local maxima. Our step computation requires overcoming an \textit{infeasibility issue}, which arises from the potential contradiction between the linearized problem constraints and the trust-region~\mbox{constraint}.~To~\mbox{resolve}~this, we relax the constraint linearization with a \textit{parameter-free decomposition technique} for the step and trust-region radius, which is designed according to the proportions among the feasibility~\mbox{residual},~\mbox{optimality} residual, and negative curvature. The decomposition balances the goals of reducing the KKT residual (i.e., feasibility + optimality) and increasing the negative curvature, and enjoys a nice \textit{scale-invariant property}.

\item TR-SQP-STORM additionally computes a second-order correction (SOC) step to resolve the \textit{(second-order) Maratos effect}.~As noted in \cite{Byrd1987Trust}, the iterates for constrained problems can fail to move away from saddle points, regardless of the trial step length. This problematic issue often arises when the constraints have significant curvatures that counteract the curvature~of~the objective. Our computation of SOC steps and the criteria of their activation~are~designed to~accommodate the inherent randomness in estimation, ensuring effectiveness for stochastic problems.

\end{enumerate}

For the above method design, we establish global almost sure first- and second-order convergence guarantees. In particular, given that the merit function parameter stabilizes and the failure probability in random models is below a certain threshold, the iteration sequence will converge almost surely to first-order stationary points, with a subsequence converging to second-order \mbox{stationary}~points. This result corroborates the findings of \cite{Chen2017Stochastic} on first-order convergence and \cite{Blanchet2019Convergence} on second-order convergence of trust-region methods designed for \textit{unconstrained}~stochastic optimization. In the context of \textit{constrained} stochastic \mbox{optimization},~we~\mbox{contribute} to \mbox{existing}~\mbox{literature} in the following four aspects.
First, TR-SQP-STORM is the first stochastic method to achieve~second-order convergence. Second, the merit parameter in our analysis only requires to be stabilized, ensured by a boundedness condition. This substantially relaxes the conditions of existing SSQP methods that demand not only stabilized but also sufficiently large merit parameters. Extreme merit parameters rely on additional model assumptions. For example, \cite{Berahas2021Sequential, Berahas2023Stochastic, Berahas2023Accelerating, Curtis2024Stochastic} imposed symmetric assumptions on the noise distribution.
Third, due to~the~trust-region constraint, our SQP subproblems remain well-defined even with indefinite Lagrangian Hessian~approximations. In contrast, most existing SSQP methods are line-search-based, necessitating positive~definite Hessian approximations typically obtained with cumbersome computational costs (e.g., matrix factorization).
Fourth, compared to random models in \cite{Na2022adaptive, Na2023Inequality}, our design~is~significantly simplified, making implementation much easier (e.g., comparing \cite[(17), (22)]{Na2022adaptive} with \eqref{def:Ak}--\eqref{def:Ck}). 
We implement TR-SQP-STORM on
problems in the CUTEst set and on regression problems to demonstrate its superior performance over line-search-based methods in practice. 
We also investigate its capability to escape saddle points in a saddle-point problem, a feature not shared by other~existing methods.

\subsection{Literature review}\label{sec:1.1}

Stochastic SQP methods have been a focal point of operations research in recent years, with a series of papers reporting on algorithm designs and analyses.~Within this line of literature, two~primary~setups for estimating objective models are commonly discussed.

The first setup is the fully stochastic setup, where a single sample is accessed at each step. 
\cite{Berahas2021Sequential} designed the first SSQP method under this setup, utilizing the $\ell_1$ merit function and a prespecified sequence $\{\beta_k\}$ to determine suitable stepsizes. Subsequently, \mbox{several}~works have expanded on this method to relax various problem conditions. For example, \cite{Berahas2023Stochastic} introduced a method to handle rank-deficient constraint Jacobians; \cite{Berahas2023Accelerating} accelerated SSQP by applying the SVRG technique; \cite{Curtis2023Stochastic} proposed an interior-point method to solve bound-constrained problems; \cite{Curtis2023Sequential} incorporated deterministic inequality constraints into the algorithm design; and \cite{Curtis2024Stochastic} inexactly solved the SQP subproblems.~The methods developed above are all line-search-based,~where~the~search~\mbox{direction}~and~\mbox{stepsize}~are~\mbox{computed}~\mbox{separately}. 
As a complement, \cite{Fang2024Fully} designed the first \textit{trust-region} SSQP method to \mbox{compute}~the~search direction and stepsize jointly. That trust-region method does not rely on positive definite Hessian approximations to make subproblems well-posed, which is critical for exhibiting \mbox{promising}~\mbox{performance} when solving nonlinear problems.~In addition, (non-)asymptotic properties of SSQP methods~and~iteration complexities
have~also been established. See \cite{Curtis2023Almost, Curtis2023Worst, Na2022Statistical, Kuang2023Online, Lu2024Variance} and references therein.~Existing literature has shown~global~almost sure convergence of~SSQP iterates to first-order stationary points. 
In~line~with~this~\mbox{series}~of~works, our paper designs~a~\mbox{trust-region}~SSQP~scheme with second-order convergence guarantees. Unlike methods in the fully-stochastic setup, our method adaptively selects the batch size based on the iteration progress and updates the trust-region~radius~according to the ratio between predicted and actual model reductions, similar to deterministic~methods. This scheme does not input any sequence $\{\beta_k\}$ to prespecify the radius or stepsize, which significantly affects the efficacy of fully stochastic methods in practice.

The second setup is the random model setup, where a batch of samples is accessed at each~step.~The random models constructed from samples aim to enforce certain estimation accuracy conditions with fixed probability. \cite{Na2022adaptive} designed the first SSQP method under this setup, where random models are employed to compute an augmented Lagrangian merit function and \mbox{perform} a stochastic line search for the stepsize selection. \cite{Na2023Inequality} further introduced an \mbox{active-set}~\mbox{strategy}~to accommodate inequality constraints and \cite{Qiu2023sequential} enhanced it to a robust~SSQP design.
Moreover, \cite{Berahas2022Adaptive} introduced a norm test condition for the batch size~\mbox{selection}~that was later generalized to projection-based and augmented Lagrangian methods with complexity~analysis \citep{Beiser2023Adaptive, Bollapragada2023adaptive, Berahas2023Sequential}. Recently, \cite{Berahas2024Modified} considered a finite-sum problem and designed a modified line-search-based SQP~to~unify~the~global and local convergence guarantees as an alternative of performing a correction step.

Following the aforementioned literature, this paper designs a trust-region SSQP method within the random model framework for constrained stochastic optimization. Our development refines existing trust-region methods for unconstrained stochastic optimization \citep{Conn2009Introduction, Conn2009Global, Bandeira2012Computation, Bandeira2014Convergence, Chen2017Stochastic, Blanchet2019Convergence}. 
In particular, due to the potential contradiction between the linearized problem constraints and the trust-region constraint, we~propose~a parameter-free decomposition technique to address the infeasibility issue when computing the trial step.~We~also streamline the construction of random models based on \cite{Chen2017Stochastic, Blanchet2019Convergence}. Our models only require accuracy conditions at iterates, unlike some models in those references that require accuracy conditions over all points within the trust region, a more stringent requirement. Furthermore, we introduce a novel reliability parameter to improve an accuracy condition of objective value estimation (see \cite[Assumption 6]{Blanchet2019Convergence} and \eqref{def:reliable_est} for comparison). This parameter, without an upper limit, enhances the algorithm's adaptivity and may reduce per-iteration batch~size.

We would also like to mention the literature that studies problems where the objective function is \textit{deterministic} but evaluated with bounded noise. \cite{Sun2023trust, Lou2024Noise, Oztoprak2023Constrained} designed robust methods for these (unconstrained) problems and showed that the iterates would visit a neighborhood of (first-order) stationary points~\mbox{infinitely}~many~times.~Their~\mbox{algorithm}~design and analysis differ significantly from ours due to the distinction between deterministic~and~stochastic optimization. In their setting, the upper bound of the noise is an input of the method and affects the radius of the convergence neighborhood; that is, the upper bound is assumed to be known in advance. Our algorithm design does not require knowledge of the upper bound of the noise.

\subsection{Notation}

We use $\|\cdot\|$ to denote the $\ell_2$ norm for vectors and the operator norm for matrices. $I$ denotes the identity matrix and $\b0$ denotes the zero vector/matrix, whose dimensions are clear from the context. For the constraints $c(\bx): \mR^d\rightarrow\mR^m$, we let $G(\bx) \coloneqq\nabla c(\bx) \in\mR^{m\times d}$ denote its Jacobian matrix and let $c^i(\bx)$ denote the $i$-th constraint for $1\leq i\leq m$ (the subscript indexes the iteration).~Define $P(\bx)=I-G(\bx)^T[G(\bx)G(\bx)^T]^{-1}G(\bx)$ to be the \mbox{projection}~\mbox{matrix}~to~the~null~space~of~$G(\bx)$.~Then,~we~let~$Z(\bx)\in\mR^{d\times (d-m)}$ form the bases of $\text{ker}(G(\bx))$ such that $Z(\bx)^TZ(\bx)=I$ and $Z(\bx)Z(\bx)^T=P(\bx)$. Throughout the paper, we use an overline to denote a stochastic estimate of a quantity. For example,~$\barg(\bx)$ denotes an estimate of $g(\bx) \coloneqq \nabla f(\bx)$.

\subsection{Structure of the paper}

We introduce the computation of gradient steps, eigen steps, and SOC steps in Section \ref{sec:2}. We propose TR-SQP-STORM in Section \ref{sec:3} and establish first- and second-order convergence guarantees in Section \ref{sec:4}.~Numerical experiments are presented in Section \ref{sec:5}, and conclusions are presented~in~Section~\ref{sec:6}.$\quad$

\section{Preliminaries}\label{sec:2}

Let $\L(\bx,\blambda)=f(\bx)+\blambda^Tc(\bx)$ be the Lagrange function of Problem \eqref{Intro_StoProb} with $\blambda\in\mR^m$ representing the Lagrangian multipliers associated with the constraints $c(\bx)$.~Under certain constraint qualifications, finding a second-order stationary point of Problem \eqref{Intro_StoProb} is equivalent to finding a pair $(\bx^*,\blambda^*)$~such~that
\begin{align}\label{1st_order_point}
\nabla \L(\bx^*,\blambda^*)=\begin{pmatrix}
\nabla_\bx \L(\bx^*,\blambda^*)\\
\nabla_{\blambda} \L(\bx^*,\blambda^*)
\end{pmatrix}=
\begin{pmatrix}
\nabla f(\bx^*)+G(\bx^*)^T\blambda^*\\
c(\bx^*)
\end{pmatrix}=\begin{pmatrix}
\b0\\\b0
\end{pmatrix}\quad\quad\text{and}\quad\quad \tau(\bx^*,\blambda^*)\geq 0,
\end{align}\vskip-0.01cm
\noindent where $\tau(\bx^*,\blambda^*)$ denotes the smallest eigenvalue of the \textit{reduced Lagrangian Hessian} $Z(\bx^*)^T\nabla^2_{\bx}\L(\bx^*,\blambda^*)Z(\bx^*)$. The Lagrangian Hessian (with respect to the primal variable $\bx$) is defined~as $\nabla^2_{\bx}\L(\bx,\blambda)=\nabla^2 f(\bx)+\sum_{i=1}^{m}\blambda^i\nabla^2 c^{i}(\bx)$. A first-order stationary point $(\bx^*,\blambda^*)$ corresponds only to $\nabla \L(\bx^*,\blambda^*)=\b0$.

Throughout the paper, we call $\|\nabla_\bx \L(\bx,\blambda)\|$ the optimality residual, $\|\nabla_\blambda \L(\bx,\blambda)\|$ (i.e.,~$\|c(\bx)\|$)~the feasibility residual, and $\|\nabla \L(\bx,\blambda)\|$ the KKT residual.~Given the $k$-th iterate $(\bx_k,\blambda_k)$, we \mbox{denote}~$g_k= \nabla f(\bx_k)$, $\nabla^2 f_k= \nabla^2 f(\bx_k)$, and their estimates $\barg_k$ and $\bar{\nabla}^2 f_k$. The construction of these estimates via random models is introduced in Section \ref{sec:3}. We denote $c_k,G_k, \{\nabla^2 c_k^i\}_{i=1}^m$ similarly. The \textit{estimated} Lagrangian gradient is defined as $\bar{\nabla} \L_k = (\bar{\nabla}_{\bx}\L_k,c_k)$ with $\bar{\nabla}_{\bx}\L_k=\barg_k+G_k^T\blambda_k$, and the \textit{estimated} Lagrangian Hessian (with respect to $\bx$) is defined as $\bar{\nabla}^2_{\bx}\L_k= \bar{\nabla}^2 f_k+\sum_{i=1}^{m}\blambda_k^i\nabla^2 c^{i}_k$.

Given the iterate $\bx_k$ and the trust-region radius $\Delta_k$ in the $k$-th iteration, we compute a trial step $\Delta\bx_k$ by (approximately) solving the trust-region SQP subproblem:
\begin{align}\label{def:SQPsubproblem}
\min_{\Delta\bx\in\mR^d}&\quad \frac{1}{2}\Delta\bx^T\barH_k\Delta\bx+\barg_k^T\Delta\bx,\quad\quad\text{s.t.}\quad c_k+G_k\Delta\bx=\b0,\;\;\;\|\Delta\bx\|\leq\Delta_k,
\end{align}
where $\barH_k$ approximates the Lagrangian Hessian $\nabla^2_{\bx}\L_k$. The subproblem \eqref{def:SQPsubproblem} performs a quadratic approximation of the nonlinear objective and a linear approximation of the nonlinear constraints~in~\eqref{Intro_StoProb}, together with a trust-region constraint.
When aiming to find the first-order stationary point, we only need $\|\barH_k\|$ to be bounded; while when aiming to find the second-order stationary~point,~we~let~$\barH_k\coloneqq\bar{\nabla}^2_{\bx}\L_k$.~Compared to unconstrained problems, a subtlety is that \eqref{def:SQPsubproblem} will not have a feasible point~if$\quad$
\begin{equation*}
\{\Delta\bx\in\mR^d:c_k+G_k\Delta\bx=\b0\}\cap\{\Delta\bx\in\mR^d:\|\Delta\bx\|\leq\Delta_k\}=\emptyset.
\end{equation*}
This infeasibility issue occurs when the radius $\Delta_k$ is too short. In this work, we introduce a~\textit{parameter-free decomposition technique} for the step and trust-region radius to relax the linearized constraint and resolve the infeasibility issue. The step is decomposed into normal and tangential components, where their lengths are controlled by respective radii that are proportional to the feasibility residual and optimality residual (or negative curvature).~Our decomposition technique does not increase~the~cost~of solving the SQP subproblem.

The trial step $\Delta\bx_k$ can be either a gradient step or an eigen step. Gradient steps aim to reduce the KKT residual to achieve first-order convergence, while eigen steps aim to explore negative curvature~of the reduced Lagrangian Hessian to achieve second-order convergence. For~the~latter~purpose, we~also need to compute a second-order correction (SOC) step to overcome the Maratos effect \citep{Conn2000Trust}. We introduce the computation of gradient steps, eigen steps, and SOC steps in~Sections~\ref{subsec:2.1},~\ref{subsec:2.2}, and \ref{subsec:2.3}, respectively.

\subsection{Gradient steps}\label{subsec:2.1}

Our gradient step computation follows a similar spirit to \cite{Fang2024Fully}.~We decompose~the~trial~step $\Delta\bx_k$ into two orthogonal segments as (recall $Z_k\in \mR^{d\times (d-m)}$ forms the bases of $\text{ker}(G_k)$)
\begin{equation*}
\Delta\bx_k=\bw_k+\bt_k, \quad \quad \text{where}\quad\quad \bw_k\in\text{im}(G_k^T)\; \text{ and }\; \bt_k=Z_k\bu_k\in \ker(G_k)\; \text{ with }\; \bu_k\in\mR^{d-m}.
\end{equation*}
Here, $\bw_k$ is called the normal step and $\bt_k$ is called the tangential step. Suppose $G_k$ has full row~rank, then we define
\begin{equation}\label{eq:Sto_compute_v_k}
\bv_k \coloneqq -G_k^T[G_kG_k^T]^{-1}c_k.
\end{equation}
Without the trust-region constraint $\|\Delta\bx_k\|\leq \Delta_k$, the linearized constraint $c_k+G_k\Delta\bx_k=\b0$ would~imply $\bw_k = \bv_k$ since $G_k\bt_k=\b0$.~However, with the trust-region constraint, we relax~the~\mbox{linearized} constraint to $\bargamma_kc_k+G_k\Delta\bx_k=\b0$ for a scalar $\bargamma_k\in(0,1]$ defined later, which corresponds to~\mbox{shrinking}~$\bv_k$~by 
\begin{equation*}
\bw_k = \bargamma_k\bv_k.
\end{equation*}

To control the lengths of the normal and tangential steps, we define
\begin{equation}\label{nequ:1}
c^{RS}_k\coloneqq \frac{c_k}{\|G_k\|},\quad\quad\quad  \bar{\nabla}_{\bx}\L_k^{RS} \coloneqq \frac{\bar{\nabla}_{\bx}\L_k}{\|\barH_k\|},\quad\quad\quad \bar{\nabla}\L_k^{RS} \coloneqq (\bar{\nabla}_{\bx}\L_k^{RS},c^{RS}_k)
\end{equation}
to be the \textit{rescaled} feasibility, optimality, and KKT residual vectors, respectively; and decompose the trust-region radius $\Delta_k$ as
\begin{equation}\label{eq:breve and tilde_delta_k}
\breve{\Delta}_k=\frac{\|c_k^{RS}\|}{\|\bar{\nabla}\L_k^{RS}\|}\cdot\Delta_k\quad\quad\quad\text{ and  }\quad\quad\quad\tilde{\Delta}_k=\frac{\|\bar{\nabla}_{\bx}\L_k^{RS}\|}{\|\bar{\nabla}\L_k^{RS}\|}\cdot\Delta_k,
\end{equation}
where we implicitly assume $\|\bar{\nabla}\L_k^{RS}\|\neq 0$ (otherwise, $\barg_k$ can be re-estimated).~We use $\breve{\Delta}_k$ to~control~the length of the normal step $\bw_k$ and use $\tilde{\Delta}_k$ to control the length of the tangential step~\mbox{$\bt_k = Z_k\bu_k$}.~In~particular, we let 
\begin{equation}\label{eq:Sto_gamma_k}
\bargamma_k \coloneqq \min\{ \breve{\Delta}_k/\|\bv_k\|,\; 1 \},
\end{equation}
and solve $\bu_k$ through the following subproblem reduced from \eqref{def:SQPsubproblem}:
\begin{equation}\label{eq:Sto_tangential_step}
\min_{\bu\in\mR^{d-m}}\quad m(\bu) \coloneqq\frac{1}{2}(Z_k\bu)^T\barH_k(Z_k\bu)+(\barg_k+\barH_k\bw_k)^TZ_k\bu,\quad\quad
\text{s.t.}\quad\|\bu\|\leq\tilde{\Delta}_k.
\end{equation}
Instead of solving \eqref{eq:Sto_tangential_step} exactly, we only require $\bu_k$ to achieve a fixed fraction $\kappa_{fcd}\in(0,1]$ of the~Cauchy reduction, that is, a reduction in the objective model $m(\cdot)$ achieved by the Cauchy~point~(see \cite{Nocedal2006Numerical}, Lemma 4.3):
\begin{equation}\label{eq:cauchy1}
m(\bu_k)-m(\b0)
\leq-\frac{\kappa_{fcd}}{2}\|Z_k^T(\barg_k+\barH_k\bw_k)\|\min\left\{\tilde{\Delta}_k,\frac{\|Z_k^T(\barg_k+\barH_k\bw_k)\|}{\|Z_k^T\barH_kZ_k\|}\right\}.
\end{equation}
Many approaches can be applied to enforce \eqref{eq:cauchy1}, such as the dogleg method, the two-dimensional~subspace minimization method, and the Steihaug's algorithm. We refer to \cite[Chapter 4]{Nocedal2006Numerical} for more details.

\begin{remark}

The radius decomposition \eqref{eq:breve and tilde_delta_k} is based on the ratios of the \textit{rescaled} feasibility and optimality residuals to the \textit{rescaled} KKT residual defined in \eqref{nequ:1}. The motivation of rescaling is to achieve scale invariance. When the problem objective and/or constraints are scaled by a (positive)~scalar,~the solution $\bx^*$ would not change but the original residuals $\|c_k\|$ and $\|\bnabla_\bx\mL_k\|$~would~be~scaled by that scalar. Thus, using original residuals would make the radius decomposition and further step~computation scale-variant.~In contrast, the \mbox{decomposition}~\eqref{eq:breve and tilde_delta_k} based~on~rescaled \mbox{residuals}~is~\mbox{scale-invariant}.\;\;\;\;

\end{remark}

\begin{remark}

Compared to \cite{Fang2024Fully}, we relax the factor $\kappa_{fcd}$ in condition \eqref{eq:cauchy1} from~$\kappa_{fcd}=1$ to $\kappa_{fcd}\in(0,1]$.~This relaxation allows the model reduction achieved by our subproblem solution~$\bu_k$~to be even less than that achieved by the Cauchy point (corresponding to $\kappa_{fcd}=1$), which can~be~computed easily and efficiently.~The factor $\kappa_{fcd}$ is determined by the approach~used~to~\mbox{compute}~$\bu_k$.~Specifically, $\kappa_{fcd}=1$ if we compute the exact Cauchy point or apply the two-dimensional method~or~the~dogleg method, while $\kappa_{fcd}\leq1$ if we apply the Steihaug’s algorithm with proper termination conditions.\;\;\;\;

\end{remark}

\subsection{Eigen steps}\label{subsec:2.2}

Performing gradient steps may not lead to a second-order stationary point because gradient steps do not keep track of the eigenvalues of the reduced Lagrangian Hessian $Z_k^T\barH_kZ_k$, which should be positive semidefinite near a second-order stationary point.~In this subsection, we introduce~eigen~steps to address negative curvature (i.e., increase the most negative eigenvalue) of the reduced \mbox{Lagrangian}~Hessian. 
Let $\bartau_k$ be the smallest eigenvalue of $Z_k^T\barH_kZ_k$ and let $\bartau_k^+\coloneqq |\min\{\bartau_k,0\}|$. The eigen step~is~taken only when $\bartau_k<0$ (cf. \eqref{eq:cri_step_comput} in Section \ref{sec:3}).

Analogous to the gradient step, the eigen step $\Delta\bx_k$ is decomposed into a normal step and~a~tangential step as $\Delta\bx_k=\bw_k+\bt_k$. To control their lengths, we let $\bartau_k^{RS+}\coloneqq \bartau_k^+/\|\barH_k\|$ be the \textit{rescaled}~negative curvature, and decompose the radius based on the proportions of the (rescaled)~feasibility~residual~and negative curvature as
\begin{equation}\label{eq:breve and tilde_delta_k_2}
\breve{\Delta}_k=\frac{\|c_k^{RS}\|}{\|(c_k^{RS},\bartau_k^{RS+})\|}\cdot\Delta_k\quad\quad\quad\text{ and  }\quad\quad\quad\tilde{\Delta}_k=\frac{\bartau_k^{RS+}}{\|(c_k^{RS},\bartau_k^{RS+})\|}\cdot\Delta_k.
\end{equation}
Again, we use $\breve{\Delta}_k$ to control the length of the normal step $\bw_k$ and use $\tilde{\Delta}_k$ to control the length~of~the tangential step $\bt_k = Z_k\bu_k$.~Specifically, the normal step is computed as $\bw_k=\bargamma_k\bv_k$, where $\bv_k$~is~defined in \eqref{eq:Sto_compute_v_k} and $\bargamma_k$ is defined in \eqref{eq:Sto_gamma_k} but with \eqref{eq:breve and tilde_delta_k_2} used to compute $\breve{\Delta}_k$. The tangential step $\bt_k=Z_k\bu_k$ solves the subproblem \eqref{eq:Sto_tangential_step}, but instead of achieving the Cauchy reduction \eqref{eq:cauchy1}, we require $\bu_k$ to~satisfy
\begin{equation}\label{eq:reduction_eigen_step}
(\barg_k+\barH_k\bw_k)^TZ_k\bu_k\leq 0,\quad\quad \|\bu_k\| \leq \tilde{\Delta}_k,\quad\quad(Z_k\bu_k)^T\barH_k(Z_k\bu_k)\leq-\kappa_{fcd}\cdot\bar{\tau}_k^+\tilde{\Delta}_k^2,
\end{equation}
which implies the curvature reduction:
\begin{equation}\label{nequ:2}
m(\bu_k) - m(\b0) \leq -\frac{\kappa_{fcd}}{2}\bar{\tau}_k^+\tilde{\Delta}_k^2<0.
\end{equation}
Here, we use $\kappa_{fcd}\in(0,1]$ to denote the fraction in both gradient steps and eigen steps for simplicity.

\begin{remark}

We briefly discuss how to compute $\bu_k$ in practice. Let $\barbzeta_k$ be an \textit{approximation} of the eigenvector of $Z_k^T\barH_kZ_k$ corresponding to the eigenvalue $\bartau_k$, and let $\barbzeta_k^{RS}\coloneqq \pm \barbzeta_k\cdot \tilde{\Delta}_k/\|\barbzeta_k\|$.~Then,~$\barbzeta_k^{RS}$ satisfies the first two conditions in \eqref{eq:reduction_eigen_step}. The third condition is also satisfied with~$\kappa_{fcd}=1$ by~computing the exact eigenvector. More generally, methods such as truncated conjugate \mbox{gradient}~and~truncated Lanczos methods can be employed to solve \eqref{eq:Sto_tangential_step} and satisfy \eqref{eq:reduction_eigen_step}; see \cite[Chapter 7.5]{Conn2000Trust} for such applications.

\end{remark}

\begin{remark}

\cite{Byrd1987Trust, Conn2000Trust} proposed decomposing the radius $\Delta_k$ into $\alpha\Delta_k$ and $(1-\alpha)\Delta_k$, where $\alpha\in(0,1)$ is a user-specified parameter. In contrast to their approaches,~our~radius decomposition is \textit{parameter-free}. In particular, we define $\breve{\Delta}_k$ and $\tilde{\Delta}_k$ in~proportion to the~rescaled feasibility residual and negative curvature. This choice is motivated by observing that the normal~step correlates with reducing the feasibility residual:
\begin{equation*}
\|c_k+G_k\Delta\bx_k\|-\|c_k\|=\|c_k+G_k\bw_k\|-\|c_k\|=-\bargamma_k\|c_k\|\leq 0,	
\end{equation*}
while the tangential step correlates with reducing $\bartau_k^+$, or equivalently, increasing the most \mbox{negative}~eigenvalue, as implied by \eqref{nequ:2}.

\end{remark}

\subsection{Second-order correction steps}\label{subsec:2.3}

Second-order correction (SOC) steps are designed to address the \textit{Maratos effect}. \cite{Byrd1987Trust} observed that when $\bx_k$ is a saddle point, the (gradient or eigen) step $\Delta\bx_k$ may increase~$f(\bx)$~and~$\|c(\bx)\|$ simultaneously, resulting in a rejection of the step in the algorithm design. Furthermore, this issue cannot be resolved by recursively reducing the radius $\Delta_k$, indicating that we are trapped at the~saddle point. Such a phenomenon (called the Maratos effect) is unique to constrained optimization~problems and stems from the inaccurate \textit{linear} approximation of the nonlinear problem constraints.

To avoid the above situation and converge to a second-order stationary point, we \textit{correct} the~trial step $\Delta\bx_k$ by following the curvature of the constraints more closely and performing the step $\Delta\bx_k +\bd_k$ \textit{when necessary}. The SOC step $\bd_k$ is given by 
\begin{equation}\label{def:correctional_step}
\bd_k=-G_k^T[G_kG_k^T]^{-1}\cbr{c(\bx_k+\Delta\bx_k)-c_k-G_k\Delta\bx_k}.
\end{equation}
Our SOC step $\bd_k$ differs from the existing one that is widely used in deterministic SQP methods \citep{Byrd1987Trust}, where $\bd_k=-G_k^T[G_kG_k^T]^{-1}c(\bx_k+\Delta\bx_k)$. This difference is motivated by the distinct behavior of the trust-region radius $\Delta_k$ in deterministic and stochastic SQP \mbox{methods}.~In~particular, in deterministic SQP methods, $\Delta_k$ is locally bounded away from zero, so the trust-region constraint will eventually become inactive. This property implies that $\bargamma_k = 1$ and $c_k + G_k\Delta\bx_k = \b0$ for large enough~$k$ (see \eqref{eq:Sto_gamma_k}). However, as shown in Lemmas \ref{lemma:diff_ared_pred_wo_corr_step} -- \ref{lemma:guarantee_succ_step_eigen}, a stochastic model is a good \mbox{surrogate}~of~the~true model only when the estimates are accurate, which holds with~a~fixed~\mbox{probability}~at each \mbox{iteration}.~As finally proved in Corollary \ref{coro:radius_conv_zero}, our stochastic SQP method exhibits $\Delta_k \rightarrow 0$, implying that $\bargamma_k$ may fail to converge to 1, and we can no longer guarantee $c_k + G_k\Delta\bx_k = \b0$.
As such, we~\mbox{incorporate}~the~remainder $c_k+G_k\Delta\bx_k$ in \eqref{def:correctional_step} to ensure that $\bd_k$ accounts for a higher order term of $\Delta\bx_k$.

\section{Trust-Region SQP for Stochastic Optimization with Random Models}\label{sec:3}

We propose the TR-SQP-STORM method in this section, which is summarized in Algorithm~\ref{Alg:STORM}.~We begin by introducing the random models used to estimate the objective value, gradient, and Hessian.$\;\;$

\subsection{Random models}

The random models in this paper are estimates of the objective values, gradients, and Hessians at each iteration. These estimates are constructed from random realizations of the stochastic objective function and are required to satisfy certain \textit{adaptive} accuracy conditions with a high~but~\textit{fixed}~probability. 
We do not specify a particular approach to obtain the estimates or assume a parametric~\mbox{distribution} for them \citep[in contrast with the sub-exponential assumption in][]{Berahas2023Sequential, Cao2023First}, which allows us to flexibly cover various problem settings.~Our goal is to~show~that,~\mbox{under}~\mbox{adaptive}~accuracy conditions, the methods utilizing these estimates converge almost surely.

Let $\kappa_h,\kappa_g,\kappa_f>0$ and $p_h,p_g,p_f\in(0,1)$ be user-specified parameters, and let $\alpha\in\{0,1\}$ be an indicator that denotes whether the algorithm is finding a first-order stationary point ($\alpha=0$) or a second-order stationary point ($\alpha=1$). Recall that $\Delta_k$ is the trust-region radius, which~will~be~adaptively adjusted in each step.

\vskip0.2cm
\noindent$\bullet$ \textbf{Hessian estimate.} 
We have to estimate the Hessian only when $\alpha=1$, i.e., when we are aiming~to find a second-order stationary point. In particular, we require
\begin{equation}\label{def:Ak}
\A_k=\left\{\|\bar{\nabla}^2f_k-\nabla^2 f_k\|\leq\kappa_h\Delta_k\right\} \quad\quad \text{satisfies}\quad\quad P(\A_k\mid\bx_k)\geq 1- p_h.
\end{equation}
The above accuracy condition indicates that the estimation error of the Hessian is proportional to the radius $\Delta_k$ with probability at least $1-p_h$. This condition is not required for first-order~convergence.

\vskip0.2cm
\noindent$\bullet$ \textbf{Gradient estimate.} 
We require the gradient estimate $\barg_k$ to satisfy an accuracy condition~proportional to $\Delta_k^{\alpha+1}$ with probability at least $1-p_g$:
\begin{equation}\label{def:Bk}
\B_k=\{\|\barg_k-g_k\|\leq\kappa_g\Delta_k^{\alpha+1}\} \quad\quad \text{satisfies}\quad\quad P(\B_k\mid\bx_k)\geq 1 - p_g.
\end{equation}

\noindent$\bullet$ \textbf{Function value estimate.}~We estimate the function value at two points: the current~\mbox{iterate}~$\bx_k$~and the trial iterate $\bx_{s_k}$, where $\bx_{s_k}= \bx_k+\Delta\bx_k$ if the SOC step is not performed and $\bx_{s_k}= \bx_k+ \Delta\bx_k+ \bd_k$ if the SOC step is performed. The trial iterate may not be accepted (i.e., $\bx_{k+1}=\bx_{k}$).

We require the following accuracy conditions:
\begin{equation}\label{def:Ck}
\C_k=\left\{\max\left(|\barf_k- f_k|,|\barf_{s_k}- f_{s_k}|\right)\leq\kappa_f\Delta_k^{\alpha+2}\right\} \quad\quad \text{satisfies}\quad\quad P(\C_k\mid\bx_k,\Delta\bx_k)\geq 1- p_f,
\end{equation}
and
\begin{equation}\label{def:reliable_est}
\max\left\{\mE\left[|\barf_k - f_k|^2\mid\bx_k,\Delta\bx_k\right],\mE\left[|\barf_{s_k} -f_{s_k}|^2\mid\bx_k,\Delta\bx_k\right]\right\}\leq\barepsilon_k^2.
\end{equation}
The first condition states that the estimation errors of $\barf_k$ and $\barf_{s_k}$ are proportional to $\Delta_k^{\alpha+2}$~with~probability at least $1-p_f$, which is more restrictive than the gradient and Hessian estimation.
The~second condition indicates that the variance of the estimates is controlled by a \mbox{reliability}~\mbox{parameter}~$\barepsilon_k$.~Here, $\barepsilon_k$ is updated at each step based on how reliably the reduction achieved in the random SQP~model~can be applied to the true SQP model, which is quantitatively measured by~the~magnitude~of~the~\mbox{reduction}.

\vskip0.2cm

We note that the above accuracy conditions \eqref{def:Ak}--\eqref{def:reliable_est} enable biased estimates, as long~as~the~probability of getting a large bias is small enough.
Estimates that satisfy these conditions can~be~obtained through various approaches. For example, we can construct estimates via subsampling as follows:$\quad\quad$
\begin{align*}
\bar{\nabla}^2f_k & = \frac{1}{|\xi_h^k|}\sum_{\xi_h\in\xi_h^k}\nabla^2F (\bx_k;\xi_h),\quad \quad\quad
\barg_k = \frac{1}{|\xi_g^k|}\sum_{\xi_g\in\xi_g^k}\nabla F (\bx_k;\xi_g),\\
\barf_k & = \frac{1}{|\xi_f^k|}\sum_{\xi_f\in\xi_f^k} F (\bx_k;\xi_f),\quad \quad\quad \quad\barf_{s_k} = \frac{1}{|\xi_f^k|} \sum_{\xi_f\in\xi_f^k}F (\bx_{s_k};\xi_f),
\end{align*}
where $\xi_h^k,\xi_g^k,\xi_f^k$ denote the sample sets and $|\cdot|$ denotes the sample size. 
If each realization $\nabla^2F(\bx_k;\xi_h)$, $\nabla F(\bx_k;\xi_h)$, $F(\bx_k;\xi_h)$ has a bounded variance, then the conditions \eqref{def:Ak}--\eqref{def:reliable_est} hold provided
\begin{equation}\label{eq:batchsize}
|\xi_h^k|\geq\frac{C_h}{p_h\cbr{\kappa_h\Delta_k}^2},\quad\quad |\xi_g^k|\geq\frac{C_g}{p_g\{\kappa_g\Delta_k^{\alpha+1}\}^2 },\quad\quad |\xi_f^k|\geq\frac{C_f}{p_f\min( \{\kappa_f\Delta_k^{\alpha+2}\}^2,\barepsilon_k^2)}
\end{equation}
for some constants $C_h,C_g,C_f>0$ (by Chebyshev's inequality). 
Furthermore, if the noise has~a~sub-exponential tail assumption, the factor $1/p_h$ in \eqref{eq:batchsize} can be relaxed to $\log(1/p_h)$ (similar~for~$1/p_g, 1/p_f$), as suggested by the (matrix) Bernstein concentration inequality \citep[Theorems 6.1 and~6.2]{Tropp2011User}.

Compared to existing literature on unconstrained problems \citep{Chen2017Stochastic, Blanchet2019Convergence}, we introduce several modifications to random models. First, our method designs random~models specifically for estimates at iterates, whereas existing literature imposed accuracy~conditions~on~all points within the trust region --- a notably more stringent requirement.~Second, \cite{Blanchet2019Convergence} adopted $\Delta_k^3$ in \eqref{def:reliable_est} to regulate expected errors in objective value estimates.~In contrast,~we~introduce a reliability parameter $\barepsilon_k$ following \cite{Na2022adaptive}. This parameter provides additional~\mbox{flexibility}~to the random model, as it is not subject to an upper bound and can be updated somewhat independently of $\Delta_k$. As we will demonstrate in Section \ref{sec:4}, with $\Delta_k\rightarrow 0$~as~$k\rightarrow\infty$,~it~is~possible~that~$\barepsilon_k \geq \Delta_k^3$ for sufficiently large $k$.~Consequently, compared to \cite{Blanchet2019Convergence}, our model may require~fewer samples to meet the reliability condition expressed in \eqref{def:reliable_est}.

\subsection{Algorithm design}\label{sec:3.2}

We require the following user-specified parameters: $p_h, p_g, p_f,\eta\in(0,1)$, $\kappa_{fcd}\in(0,1]$, $r,\Delta_{\max},\kappa_h$, $\kappa_g>0$, $0<\kappa_f\leq \frac{\kappa_{fcd}\eta^3}{16\max\{1,\Delta_{\max}\}}$, and $\rho,\gamma>1$. We initialize the method with $\bx_0$, $\Delta_0\in(0, \Delta_{\max})$,~and $\barepsilon_0$, $\barmu_0>0$. Recall that we set $\alpha=0$ if we aim to find a first-order stationary point, while $\alpha=1$~if~we aim to find a second-order stationary point.

Given $(\bx_k,\Delta_k,\barepsilon_k,\barmu_k)$ in the $k$-th iteration, our method proceeds in the following four steps.$\quad\quad$

\vskip0.2cm

\noindent\textbf{\underline{Step 1:} Gradient and Hessian estimations.} We obtain the gradient estimate $\barg_k$ that satisfies \eqref{def:Bk}. Then we compute the Lagrangian multiplier $\barblambda_k=-[G_kG_k^T]^{-1}G_k\barg_k$ and the Lagrangian~gradient $\bar{\nabla}\L_k = \barg_k + G_k^T\blambda_k$. For the Hessian estimate, we consider two cases.

\vskip0.2cm

\noindent $\bullet$ \textbf{If $\alpha=0$:} we generate any matrix $\barH_k$ to approximate the Lagrangian Hessian $\nabla_{\bx}^2\L_k$ and set~$\bartau_k^+=0$.

\vskip0.1cm
\noindent $\bullet$ \textbf{If $\alpha=1$:} we obtain the Hessian estimate $\bar{\nabla}^2 f_k$ that satisfies \eqref{def:Ak}. Then, we compute $\barH_k=\bar{\nabla}^2 f_k+\\ \sum_{i=1}^{m}\barblambda_k^i\nabla^2 c_k^i$ and set $\bartau_k$ to be the smallest eigenvalue of $Z_k^T\barH_kZ_k$ and $\bartau_k^+=|\min\{\bartau_k,0\}|$.

\vskip0.3cm

\noindent\textbf{\underline{Step 2:} Trial step computation.} With the above gradient and Hessian estimates, we compute~the trial step. In particular, if the following condition \textit{does not} hold
\begin{equation}\label{def:acc1}
\max\left\{\frac{\|\barnablaL_k\|}{\max\{1,\|\barH_k\|\}},\bartau_k^+\right\}\geq \eta\Delta_k,
\end{equation}
we say the $k$-th iteration is \textit{\textbf{unsuccessful}} and let $\bx_{k+1}=\bx_k$. We also decrease the radius and the reliability parameter by $\Delta_{k+1}=\Delta_k/\gamma$ and $\barepsilon_{k+1}=\barepsilon_k/\gamma$.

Otherwise, if \eqref{def:acc1} holds, we then decide whether to perform a gradient step or an eigen step.~We check the following condition:
\begin{equation}\label{eq:cri_step_comput}
\|\bar{\nabla}\L_k\|\min\left\{\Delta_k,\frac{\|\bar{\nabla}\L_k\|}{\|\barH_k\|}\right\}\geq  \bartau_k^+\Delta_k\left(\Delta_k+\|c_k\|\right).
\end{equation}
If \eqref{eq:cri_step_comput} holds, we compute $\Delta\bx_k$  as the gradient step (cf. Section \ref{subsec:2.1}); otherwise, we compute $\Delta\bx_k$~as the eigen step (cf. Section \ref{subsec:2.2}). 

\begin{remark}

The criterion in a similar flavor to \eqref{def:acc1} is a standard practice in stochastic optimization \citep[see, e.g.,][]{Chen2017Stochastic,Blanchet2019Convergence,Jin2024Sample}.~In fact, due~to~the~presence~of estimation errors, 
there is a potential discrepancy where the trial step leads to a sufficient reduction in the stochastic (merit) function, suggesting a successful $k$-th iteration, whereas it~leads to an insufficient reduction (or even an increase) in the actual expected (merit) function.
The~\mbox{condition}~\eqref{def:acc1}~\mbox{guarantees} that the stepsize $\Delta_k$ will not exceed a certain proportion of either $\|\bar{\nabla}\L_k\|$ or $\bartau_k^+$, should an \mbox{iteration}~be deemed successful and the iterate be updated, thus mitigating the repercussions of such discrepancies.

\end{remark}

\begin{remark}

The condition \eqref{eq:cri_step_comput} compares two reductions achieved by the gradient and eigen~steps. 
The left-hand side represents the reduction made by the gradient step, while the right-hand~side~represents the reduction made by the eigen step. Instead of computing both the gradient and eigen steps in each iteration, we always perform the more aggressive step. Certainly, when finding~a~\mbox{first-order}~stationary point, we have $\bartau_k^+ = 0$ and \eqref{eq:cri_step_comput} holds; thus, we always perform the gradient step.

\end{remark}

\noindent\textbf{\underline{Step 3:} Merit function estimation.}~After we compute the trial step, we then update the iterate $\bx_k$. The update is based on the reduction of the trial step achieved on an (estimated)~$\ell_2$~merit~\mbox{function}~that balances the objective optimality and constraints violation:
\begin{equation}\label{def:det_merit_fun}
\L_{\mu}(\bx)=f(\bx)+\mu\|c(\bx)\|.
\end{equation} 
In particular, given $\barmu_{k}$, we define the predicted reduction as
\begin{equation}\label{def:Pred_k}
\text{Pred}_k=\barg_k^T\Delta\bx_k+\frac{1}{2}\Delta\bx_k^T\barH_k\Delta\bx_k+\barmu_k(\|c_k+G_k\Delta\bx_k\|-\|c_k\|),
\end{equation}
which can be viewed as the reduction of the linearized merit function. We update the merit~parameter $\barmu_k\leftarrow\rho\barmu_k$ until
\begin{equation}\label{eq:threshold_Predk}
\text{Pred}_k\leq -\frac{\kappa_{fcd}}{2} \max\left\{\|\bar{\nabla}\L_k\|\min\left\{\Delta_k,\frac{\|\bar{\nabla}\L_k\|}{\|\barH_k\|}\right\},\bartau_k^+\Delta_k\left(\Delta_k+\|c_k\|\right)\right\}.
\end{equation}
(Our analysis in Section \ref{subsec:merit_para} shows that \eqref{eq:threshold_Predk} is ensured to satisfy for large enough $\barmu_{k}$.) On the other hand, we let $\bx_{s_k}=\bx+\Delta\bx_k$ be the trial point and obtain the function value estimates~$\barf_k$~and~$\barf_{s_k}$~that satisfy \eqref{def:Ck} and \eqref{def:reliable_est}. Then, we compute the actual reduction as
\begin{equation}\label{def:Ared_k}
\text{Ared}_k=\bar{\L}_{\barmu_k}^{s_k}-\bar{\L}_{\barmu_k}^k=\barf_{s_k}-\barf_k+\barmu_k(\|c_{s_k}\|-\|c_k\|).
\end{equation}

\noindent\textbf{\underline{Step 4:} Iterate update.} Finally, we update the iterate by checking the following condition:
\begin{equation}\label{nequ:3}
\text{(a):}\;\;\; \text{Ared}_k/\text{Pred}_k\geq\eta\hskip2cm \text{and}\hskip2cm \text{(b):} \;\;\; -\text{Pred}_k\geq\barepsilon_k.
\end{equation}
In particular, we see that \eqref{nequ:3} leads to three cases.

\vskip0.2cm

\noindent $\bullet$ \textbf{Case 1: (\ref{nequ:3}a) holds.}~We say the $k$-th iteration is \textit{\textbf{successful}}. We update the iterate~and~the~trust-region radius as $\bx_{k+1}=\bx_{s_k}$ and $\Delta_{k+1}=\min\{\gamma\Delta_k,\Delta_{\max}\}$. Furthermore, if (\ref{nequ:3}b) holds, we say~the $k$-th iteration is \textit{\textbf{reliable}} and increase the reliability parameter by $\barepsilon_{k+1}=\gamma\barepsilon_k$. 
Otherwise, we~say~the $k$-th iteration is \textit{\textbf{unreliable}} and decrease the reliability parameter by $\barepsilon_{k+1}=\barepsilon_k/\gamma$.

\vskip0.2cm

\noindent $\bullet$ \textbf{Case 2: (\ref{nequ:3}a) does not hold and $\alpha = 1$.} In this case, we decide whether to perform a~SOC~step to recheck (\ref{nequ:3}a). Specifically, if $\|c_k\|\leq r$, we compute a SOC step $\bd_k$ (cf. Section~\ref{subsec:2.3})~and~set~$\bx_{s_k}=\bx_k+\Delta\bx_k+\bd_k$ as a new trial point. Then, we re-estimate $\barf_{s_k}$ to satisfy \eqref{def:Ck} and \eqref{def:reliable_est}, recompute $\text{Ared}_k$ as in \eqref{def:Ared_k}, and recheck (\ref{nequ:3}a). If (\ref{nequ:3}a) holds, we go to \textbf{Case 1} above; if  (\ref{nequ:3}a) does not~hold,~we go to \textbf{Case 3} below. On the other hand, if $\|c_k\|>r$, the SOC step is not triggered and we directly~go to \textbf{Case 3} below.

\vskip0.2cm

\noindent $\bullet$ \textbf{Case 3: (\ref{nequ:3}a) does not hold and $\alpha = 0$.} We say the $k$-th iteration is \textit{\textbf{unsuccessful}}. 
We~do~not update the current iterate by setting $\bx_{k+1}=\bx_k$, and decrease the trust-region radius~and~the~reliability parameter by setting $\Delta_{k+1}=\Delta_k/\gamma$ and $\barepsilon_{k+1}=\barepsilon_k/\gamma$.

\vskip0.2cm

The criterion (\ref{nequ:3}a) is aligned with the deterministic trust-region methods for deciding whether~the trial step is successful \citep{Powell1990trust, Byrd1987Trust, Omojokun1989Trust, Heinkenschloss2014Matrix}, ensuring that the reduction in the merit function is at least a specified fraction of the reduction predicted by the SQP model. Based on the values of $\text{Pred}_k$ and $\barepsilon_k$,~we~\mbox{further}~\mbox{classify}~the~successful step into a reliable or unreliable step. For a reliable step, we increase the reliability parameter to relax the accuracy condition for the subsequent iteration, thereby reducing the necessary sample size. Conversely, for an unreliable step, we decrease the reliability parameter for the next iteration~to secure more reliable estimates.

To end this section, we introduce some additional notation. We define $\F_{-1}\subseteq\F_0\subseteq\F_1\cdots$ as a filtration of $\sigma$-algebras, where $\mF_{k-1} = \sigma(\{\bx_i\}_{i=0}^{k})$, $\forall k\geq 0$, contains all the randomness~before performing the $k$-th iteration. In the $k$-th iteration, we first obtain $\barg_k$ and $\bnabla^2 f_k$ (if $\alpha=1$),~and~then~compute $\Delta\bx_k$, update $\barmu_{k}$, and compute $\bd_k$ (if SOC is triggered). Defining $\F_{k-0.5} = \sigma(\{\bx_i\}_{i=0}^{k}\cup \{\barg_k, \bnabla^2 f_k\})$, we find that for all $k\geq 0$, $\sigma(\bx_k,\Delta_k,\barepsilon_k)\subseteq \F_{k-1}$ and $\sigma(\Delta\bx_k,\barblambda_k,\barmu_k,\bd_k)\subseteq\F_{k-0.5}$.

\begin{algorithm}[t]
\caption{TR-SQP for Stochastic Optimization with Random Models (TR-SQP-STORM)}\label{Alg:STORM}
\begin{algorithmic}[1]
\State \textbf{Input:} Initial iterate $\bx_0$ and radius $\Delta_0\in(0,\Delta_{\max})$, and parameters $p_h,p_g,p_f,\eta\in(0,1)$,~$\kappa_{fcd}\in(0,1]$, $\barmu_0, \barepsilon_0, r,\kappa_h,\kappa_g>0$, $0<\kappa_f\leq\frac{\kappa_{fcd}\eta^3}{16\max\{1,\Delta_{\max}\}}$, $\rho,\gamma>1$. 
\State Set $\alpha=0$ for first-order stationarity and $\alpha=1$ for second-order stationarity.
\For {$k=0,1,\cdots,$}
\State Obtain $\barg_k$ and compute $\barblambda_k$ and $\bar{\nabla}\L_k$. \Comment{\textbf{Step 1}}
\State If $\alpha=1$, obtain $\bar{\nabla}^2 f_k$, compute $\barH_k$ and the smallest eigenvalue $\bartau_{k}$ of $Z_k^T\barH_kZ_k$, and~set~$\bartau_{k}^+=|\min\{\bartau_{k},0\}|$. Otherwise, let $\barH_{k}$ be certain approximation of $\nabla^2\mL_k$ and set $\bartau_{k}^+ = 0$.

\If{\eqref{def:acc1} does not hold} \Comment{\textbf{Step 2}}
\State Set $\bx_{k+1}=\bx_k,\Delta_{k+1}=\Delta_k/\gamma,\barepsilon_{k+1}=\barepsilon_k/\gamma$. \LongComment{Unsuccessful iteration}
\Else
\State Compute $\Delta\bx_k$ as a gradient step if \eqref{eq:cri_step_comput} holds; otherwise, compute $\Delta\bx_k$ as an eigen step.

\State Perform $\barmu_k \leftarrow \rho\barmu_k$ until $\text{Pred}_k$ satisfies \eqref{eq:threshold_Predk}.  \Comment{\textbf{Step 3}}
\State Set $\bx_{s_k}=\bx_k+\Delta\bx_k$, obtain $\barf_k, \barf_{s_k}$, and compute $\text{Ared}_k$ as in \eqref{def:Ared_k}.

\If {$\text{Ared}_k/\text{Pred}_k\geq\eta$}\Comment{\textbf{Step 4 (Case 1)}}
\State Set $\bx_{k+1}=\bx_{s_k}$ and $\Delta_{k+1}=\min\{\gamma\Delta_k,\Delta_{\max}\}$. \LongComment{Successful iteration}
\If {$-\text{Pred}_k\geq \barepsilon_k$}
\State Set $\barepsilon_{k+1}=\gamma\barepsilon_k$. \LongComment{Reliable iteration}
\Else
\State Set $\barepsilon_{k+1}=\barepsilon_k/\gamma$. \LongComment{Unreliable iteration}
\EndIf

\ElsIf{$\alpha=1$ and $\|c_k\|\leq r$} \Comment{\textbf{Step 4 (Case 2)}}
\State Compute SOC step $\bd_k$, set $\bx_{s_k}=\bx_k+\Delta\bx_k+\bd_k$, re-estimate $\barf_{s_k}$, and recompute~$\text{Ared}_k$. 
\State If $\text{Ared}_k/\text{Pred}_k\geq\eta$, perform Lines 13-18; otherwise, perform Line 23.

\Else \Comment{\textbf{Step 4 (Case 3)}}
\State Set $\bx_{k+1}=\bx_k$, $\Delta_{k+1}=\Delta_k/\gamma$, $\barepsilon_{k+1}=\barepsilon_k/\gamma$. \LongComment{Unsuccessful iteration}
\EndIf
\State Set $\barmu_{k+1}=\barmu_k$.
\EndIf 
\EndFor
\end{algorithmic}
\end{algorithm}

\section{Convergence Analysis} \label{sec:4}

In this section, we establish global almost sure first- and second-order convergence properties~for~TR-SQP-STORM. We begin by stating assumptions.

\begin{assumption}\label{ass:1-1}

Let $\Omega\subseteq\mR^d$ be an open convex set containing the iterates and trial points~$\{\bx_k, \bx_{s_k}\}$. The objective $f(\bx)$ is twice continuously differentiable and bounded below by $f_{\inf}$ over $\Omega$.~The~gradient $\nabla f(\bx)$ and Hessian $\nabla^2 f(\bx)$ are both Lipschitz continuous over $\Omega$, with constants $L_{\nabla f}$ and~$L_{\nabla^2 f}$, respectively. Analogously, the constraint $c(\bx)$ is twice continuously differentiable and~its~\mbox{Jacobian}~$G(\bx)$ is Lipschitz continuous over $\Omega$ with constant $L_G$. 
For $1 \leq i \leq m$, the Hessian of the $i$-th constraint, $\nabla^2 c^i(\bx)$, is Lipschitz continuous over $\Omega$ with constant $L_{\nabla^2 c}$. Furthermore, we assume that there~exist constants $\kappa_c$, $\kappa_{\nabla f}$, $\kappa_{1,G}$, $\kappa_{2,G}>0$ such that \vskip-0.3cm
\begin{equation*}
\|c_k\| \leq \kappa_{c},\quad\quad \|g_k\| \leq \kappa_{\nabla f}, \quad\quad \kappa_{1,G} \cdot I \preceq G_k G_k^T \preceq \kappa_{2,G} \cdot I, \quad\quad\forall k\geq 0.
\end{equation*}\vskip-0.04cm
\noindent For first-order stationarity, we require the Hessian approximation $\|\barH_k\| \leq \kappa_B$ for a constant~$\kappa_B \geq 1$.

\end{assumption}

Assumption \ref{ass:1-1} is standard in the SQP literature \citep{Byrd1987Trust, Powell1990trust, ElAlem1991Global, Conn2000Trust, Berahas2021Sequential, Berahas2023Stochastic, Curtis2024Stochastic, Fang2024Fully}. For first-order stationarity, it suffices to assume that $f(\bx)$ and $c(\bx)$ are continuously~differentiable,~without continuity conditions on Hessians $\nabla^2 f(\bx)$ and $\nabla^2 c^i(\bx)$. Assumption \ref{ass:1-1} implies that $G_k$ has full row rank, $\sqrt{\kappa_{1,G}}\leq\|G_k\|\leq\sqrt{\kappa_{2,G}}$, and $\|G_k^T[G_kG_k^T]^{-1}\| \leq 1/\sqrt{\kappa_{1,G}}$. Consequently, both the true~Lagrangian multiplier $\blambda_k = -[G_kG_k^T]^{-1}G_kg_k$ and the estimated counterpart $\barblambda_k = -[G_kG_k^T]^{-1}G_k\barg_k$ are well defined. Additionally, $\|\nabla^2 f(\bx)\| \leq L_{\nabla f}$ and $\|\nabla^2 c^i(\bx)\| \leq L_G$ for $1\leq i\leq m$ over $\Omega$.

The following assumption states that the merit parameter $\barmu_k$ stabilizes when $k\rightarrow\infty$.

\begin{assumption}\label{ass:4-1}
There exist an (potentially stochastic) iteration threshold $\barK<\infty$ and a~deterministic constant $\hat{\mu}$, such that $\barmu_k=\barmu_{\barK}\leq \hatmu$ for all $k\geq\barK$.
\end{assumption}

Assumption \ref{ass:4-1} is commonly imposed in advance for studying the global convergence of SSQP \citep{Berahas2021Sequential, Berahas2023Stochastic, Berahas2023Accelerating, Curtis2024Stochastic, Fang2024Fully}. Compared to existing literature, we 
do not require $\barmu_{\barK}$ to be large enough.~In Section \ref{subsec:merit_para}, we show that our merit~\mbox{parameter}~\mbox{update}~scheme for ensuring the sufficient reduction \eqref{eq:threshold_Predk} will naturally make this assumption hold, provided $\barg_k$ and $\bar{\nabla}^2 f_k$ are upper bounded and $\|\barH_k\|$ is lower bounded.

\subsection{Fundamental lemmas}

Our first result shows that, on the event $\A_k\cap\B_k$ in \eqref{def:Ak} and \eqref{def:Bk}, the Hessian estimate $\barH_k$~used~to~reach a second-order stationary point (i.e., constructed under $\alpha=1$) is upper bounded. 

\begin{lemma}\label{lemma:bounded_Hessian}

Under Assumption \ref{ass:1-1} with $\alpha=1$, there exists a positive constant $\kappa_B\geq 1$ such that $\|\barH_k\|\leq \kappa_B$ on the event $\A_k\cap\B_k$.
\end{lemma}

\begin{proof}

Recall that $\barH_k = \bar{\nabla}^2f_k+\sum_{i=1}^m\barblambda_k^i\nabla^2c^i_k$, we have
\begin{align}\label{eq:lemma1_1}
\|\barH_k\| 
& \leq \|\bar{\nabla}^2f_k-\nabla^2f_k\|+\|\nabla^2f_k\| + \|\sum_{i=1}^m(\barblambda_k^i-\blambda_k^i)\nabla^2c^i_k\| +\|\sum_{i=1}^m\blambda_k^i\nabla^2c^i_k\| \nonumber\\
& \leq \|\bar{\nabla}^2f_k-\nabla^2f_k\|+\|\nabla^2f_k\|+\|\barblambda_k-\blambda_k\| \big\{\sum_{i=1}^m\|\nabla^2c^i_k\|^2\big\}^{1/2}+\|\blambda_k\| \big\{\sum_{i=1}^m\|\nabla^2c^i_k\| \big\}^{1/2}\nonumber\\
& \leq \|\bar{\nabla}^2f_k-\nabla^2f_k\|+L_{\nabla f}+\frac{\sqrt{m} L_G}{\sqrt{\kappa_{1,G}}}\|\barg_k-g_k\|+\frac{\sqrt{m} L_G\kappa_{\nabla f}}{\sqrt{\kappa_{1,G}}},
\end{align}
where the last inequality follows from Assumption~\ref{ass:1-1} and the definitions of $\blambda_k$ and $\barblambda_k$. On the~event $\A_k\cap\B_k$, $\|\bar{\nabla}^2f_k - \nabla^2 f_k\|\leq\kappa_h\Delta_k$ and $\|\barg_k - g_k\|\leq\kappa_g\Delta_k^2$. Since $\Delta_k\leq\Delta_{\max}$, it follows that 
\begin{equation*}
\|\barH_k\| \leq 
\kappa_h\Delta_{\max}+L_{\nabla f}+\frac{\sqrt{m} L_G}{\sqrt{\kappa_{1,G}}}(\kappa_g\Delta_{\max}^2+\kappa_{\nabla f}).
\end{equation*}
We complete the proof by setting $\kappa_B=\max\{1,\kappa_h\Delta_{\max}+L_{\nabla f}+\sqrt{m} L_G/\sqrt{\kappa_{1,G}}\cdot(\kappa_g\Delta_{\max}^2+\kappa_{\nabla f})\}$.
\end{proof}

We demonstrate in the next lemma that for second-order stationarity, the difference between~the true Lagrangian Hessian $\nabla_{\bx}^2\L_k$ and its estimate $\barH_k$ is bounded by a quantity proportional to $\Delta_k$~on the event $\A_k\cap\B_k$. Furthermore, the difference between the eigenvalue $\tau_k^+\coloneqq |\min\{\tau_k,0\}|$~and~its~estimate $\bartau_k^+$ is bounded by the same quantity. This lemma ensures that when both the objective~gradient and Hessian estimates are accurate, the estimate of $\tau_k^+$ is also precise.

\begin{lemma}\label{lemma:tau_accurate} 

Under Assumption \ref{ass:1-1} with $\alpha=1$, there exists a positive constant $\kappa_H>0$ such that $\|\nabla_{\bx}^2\L_k-\barH_k\| \leq \kappa_H\Delta_k$ and $| \tau_k^+-\bartau_k^+|\leq\kappa_H\Delta_k$ on the event $\A_k\cap\B_k$.

\end{lemma}

\begin{proof}

We have
\begin{align*}
\|\nabla_{\bx}^2\L_k-\barH_k\| & = \|\nabla^2f_k-\bar{\nabla}^2f_k+\sum_{i=1}^m(\blambda_k^i-\barblambda_k^i)\nabla^2c_k^i\| \leq \|\nabla^2f_k-\bar{\nabla}^2f_k\|+\|\barblambda_k-\blambda_k\|\big\{\sum_{i=1}^m\|\nabla^2c_k^i\|^2\}^{1/2} \\
& \leq \|\nabla^2f_k-\bar{\nabla}^2f_k\|+\frac{\sqrt{m} L_G}{\sqrt{\kappa_{1,G}}}\|g_k-\barg_k\| \quad\quad (\text{Assumption~\ref{ass:1-1}})\\
& \leq \kappa_h\Delta_k+\frac{\sqrt{m}\kappa_gL_G}{\sqrt{\kappa_{1,G}}}\Delta_k^2  \leq\left(\kappa_h+\frac{\sqrt{m}\kappa_gL_G\Delta_{\max}}{\sqrt{\kappa_{1,G}}}\right)\Delta_k \eqqcolon \kappa_H\Delta_k,
\end{align*}
where the fourth inequality is due to the event $\A_k\cap\B_k$. Next, we show $|\tau_k-\bartau_k|\leq\kappa_H\Delta_k$. Let $\barbzeta_k$~be a normalized eigenvector corresponding to $\bartau_k$, then
\begin{equation*}
\tau_k-\bartau_k \leq \barbzeta_k^T\left[Z_k^T(\nabla_{\bx}^2\L_k-\barH_k)Z_k\right]\barbzeta_k \leq\|\nabla_{\bx}^2\L_k-\barH_k\|\leq\kappa_H\Delta_k.
\end{equation*}
Let $\bzeta_k$ be a normalized eigenvector corresponding to $\tau_k$, then
\begin{equation*}
\bartau_k-\tau_k
\leq \bzeta_k^T\left[Z_k^T(\barH_k-\nabla_{\bx}^2\L_k)Z_k\right]\bzeta_k
\leq\|\nabla_{\bx}^2\L_k-\barH_k\|\leq\kappa_H\Delta_k.
\end{equation*}
Combining the last two displays, we have $|\tau_k-\bartau_k|\leq\kappa_H\Delta_k$, which implies $|\tau_k^+-\bartau_k^+|\leq\kappa_H\Delta_k$.
\end{proof}

In the following lemma, we demonstrate that when the current iterate $\bx_k$ is a neither \mbox{first-order}~nor second-order stationary point (i.e., $\|\nabla\L_k\| > 0$ or $\tau_k^+ > 0$), and the estimates of both objective gradients and Hessians are accurate, then Line 6 of Algorithm \ref{Alg:STORM} will not be triggered~(i.e.,~\eqref{def:acc1}~holds)~for~sufficiently small trust-region radius. For the sake of notational consistency, we assume $\A_k$ also holds~for $\alpha=0$, although we do not use objective Hessian estimates for the design of first-order stationarity.\;\;\;\;

\begin{lemma}\label{lemma:line6_not_hold}

Under Assumption \ref{ass:1-1} and the event $\A_k\cap\B_k$, if either
\begin{equation}\label{eq:lemma3:delta}
\|\nabla\L_k\|\geq (\kappa_g\max\{1,\Delta_{\max}\}+\eta\kappa_B) \cdot \Delta_k\quad\quad\text{or}\quad\quad\tau_k^+\geq (\kappa_H+\eta)\cdot \Delta_k,
\end{equation}
then Line 6 of Algorithm \ref{Alg:STORM} will not be triggered.
\end{lemma}

\begin{proof}

On the event $\A_k\cap\B_k$, we have $\|\nabla\L_k\|-\|\bar{\nabla}\L_k\| \leq \|g_k-\barg_k\|\leq \kappa_g\max\{1,\Delta_{\max}\}\Delta_k$, $\|\barH_k\| \leq\kappa_B$ (cf. Assumption~\ref{ass:1-1} and Lemma \ref{lemma:bounded_Hessian}), and $\tau_k^+-\bartau_k^+ \leq \kappa_H\Delta_k$ (cf. Lemma \ref{lemma:tau_accurate}). Consequently, \eqref{eq:lemma3:delta} results in either
\begin{equation*}
\frac{\|\bar{\nabla}\L_k\|}{\max\{1,\|\barH_k\|\}}\geq \eta  \cdot\Delta_k\quad\quad\text{or}\quad\quad \bartau_k^+ \geq \eta \cdot \Delta_k.
\end{equation*}
Thus, Line 6 will not be triggered.
\end{proof}

Let us define $\L_{\barmu_{k}}^{s_k} \coloneqq \L_{\barmu_{k}}(\bx_{s_k})$ and $\L_{\barmu_{k}}^{k} \coloneqq \L_{\barmu_{k}}(\bx_{k})$, where $\barmu_k$ is the merit parameter selected in the $k$-th iteration. Here, $\bx_{s_k}=\bx_k+\Delta\bx_k$ if the SOC step is not performed and $\bx_{s_k}=\bx_k+\Delta\bx_k+\bd_k$~if the SOC step is performed. The following two lemmas examine the difference between the reduction~in~the merit function (i.e., $\L_{\barmu_{k}}^{s_k}-\L_{\barmu_{k}}^{k}$) and $\text{Pred}_k$ (see \eqref{def:Pred_k}). 
We first show that on the event $\A_k\cap\B_k$,~when the SOC step is not performed, the difference has an upper bound proportional to $\Delta_k^2$.

\begin{lemma}\label{lemma:diff_ared_pred_wo_corr_step}

Under Assumptions \ref{ass:1-1}, \ref{ass:4-1}, and the event $\A_k\cap\B_k$, when the SOC step is not~performed, we have $\forall k \geq \barK$,
\begin{equation*}
\big|\L_{\barmu_{\barK}}^{s_k}-\L_{\barmu_{\barK}}^{k}-\text{Pred}_k\big| \leq \Upsilon_1\Delta_k^2,
\end{equation*}
where $\Upsilon_1=\kappa_g\max\{1,\Delta_{\max}\}+\frac{1}{2}(L_{\nabla f}+\kappa_B+\hatmu L_G)$.
\end{lemma}

\begin{proof}

Since the SOC step is not performed, $\bx_{s_k}=\bx_k+\Delta\bx_k$. Combining \eqref{def:det_merit_fun} and \eqref{def:Pred_k}, we have
\begin{equation*}
\big|\L_{\barmu_{\barK}}^{s_k}-\L_{\barmu_{\barK}}^{k}-\text{Pred}_k\big| = \left|f_{s_k}+\barmu_{\barK}\|c_{s_k}\|-f_k-\barg_k^T\Delta\bx_k-\frac{1}{2}\Delta\bx_k^T\barH_k\Delta\bx_k-\barmu_{\barK}\|c_k+G_k\Delta\bx_k\|\right|.
\end{equation*} 
By the Taylor expansion of $f(\bx)$ and the Lipschitz continuity of $\nabla f(\bx)$, we have
\begin{equation*}
f_{s_k}-f_k-\barg_k^T\Delta\bx_k \leq (g_k-\barg_k)^T\Delta\bx_k+\frac{1}{2}L_{\nabla f}\|\Delta\bx_k\|^2.
\end{equation*}
Similarly, we have
\begin{equation*}
\big|\|c_{s_k}\|-\|c_k+G_k\Delta\bx_k\|\big|\leq \|c_{s_k}-c_k-G_k\Delta\bx_k\|\leq \frac{1}{2}L_G\|\Delta\bx_k\|^2.
\end{equation*}
Recall that $\barmu_{\barK}\leq\hatmu$ (cf.~Assumption~\ref{ass:4-1}) and $\|\barH_k\|\leq\kappa_B$ (cf.~Assumption~\ref{ass:1-1} and Lemma~\ref{lemma:bounded_Hessian}),~combining the above two displays leads to
\begin{equation*}
\big|\L_{\barmu_{\barK}}^{s_k}-\L_{\barmu_{\barK}}^{k}-\text{Pred}_k\big| \leq \|g_k-\bar{g}_k\|\|\Delta\bx_k\|+\frac{1}{2}(L_{\nabla f}+\kappa_B+\hatmu L_G)\|\Delta\bx_k\|^2. 
\end{equation*}
On the event $\B_k$, we have $\|g_k-\bar{g}_k\|\leq\kappa_g\max\{1,\Delta_{\max}\}\Delta_k$. Since $\|\Delta\bx_k\|\leq \Delta_k$, the result~follows~from the display above and we complete the proof.
\end{proof}

Next, we show that on the event $\A_k\cap\B_k$, when the SOC step is performed, the bound in Lemma \ref{lemma:diff_ared_pred_wo_corr_step} is strengthened to $\Delta_k^3$.

\begin{lemma}\label{lemma:diff_ared_pred_w_corr_step}

Under Assumptions \ref{ass:1-1}, \ref{ass:4-1}, and the event $\A_k\cap\B_k$, when the SOC step is performed,~we have $\forall k\geq\barK$,
\begin{equation}\label{eq:diff_ared_pred}
\big|\L_{\barmu_{\barK}}^{s_k}-\L_{\barmu_{\barK}}^{k}-\text{Pred}_k\big|\leq \Upsilon_2\Delta_k^3,
\end{equation}
where 
\begin{multline*}
\Upsilon_2 = \kappa_g+\frac{L_{\nabla^2 f}+\kappa_h}{2}
+\frac{L_G^2\Delta_{\max}(0.5L_{\nabla f} + \sqrt{m}\hat{\mu}L_G)}{\kappa_{1,G}}\\
+ \frac{0.5\sqrt{m}L_{\nabla^2 c}(L_{\nabla f}\Delta_{\max} + \kappa_{\nabla f}) + 0.5 \sqrt{m}L_G(\kappa_g\Delta_{\max} + L_{\nabla f} + 2\hat{\mu}L_G) }{\sqrt{\kappa_{1,G}}}.
\end{multline*}
\end{lemma}

\begin{proof}

We have	
\begin{align}\label{eq:lemma_diff_w_soc_1}
\big|\L_{\barmu_{\barK}}^{s_k}-\L_{\barmu_{\barK}}^{k}-\text{Pred}_k\big|
& = \big|f_{s_k}+\barmu_{\barK}\|c_{s_k}\|-f_k-\barg_k^T\Delta\bx_k-\frac{1}{2} \Delta\bx_k^T\barH_k\Delta\bx_k-\barmu_{\barK}\|c_k+G_k\Delta\bx_k\|\big|  \nonumber \\
& \leq
\big|f_{s_k}-f_k-\barg_k^T\Delta\bx_k-\frac{1}{2} \Delta\bx_k^T\barH_k\Delta\bx_k\big|+\hatmu\| c_{s_k}-c_k-G_k\Delta\bx_k\|,
\end{align}
where we have used Assumption~\ref{ass:4-1}. First, we analyze the second term in \eqref{eq:lemma_diff_w_soc_1}.~Since the SOC~step~is performed, we have $\bx_{s_k}=\bx_k+\Delta\bx_k+\bd_k$. For $1\leq i\leq m$, by the Taylor expansion, we obtain
\begin{equation*}
|c_{s_k}^i- c_k^i - (\nabla c_k^i)^T\Delta\bx_k| \stackrel{\eqref{def:correctional_step}}{=} |c_{s_k}^i - c^i(\bx_k+\Delta\bx_k) - (\nabla c_k^i)^T\bd_k| \leq L_G(\|\Delta\bx_k\|\|\bd_k\| + \|\bd_k\|^2).
\end{equation*} 
Therefore, $\| c_{s_k} -c_k-G_k\Delta\bx_k \|\leq \sqrt{m} L_G(\|\Delta\bx_k\|\|\bd_k\| + \|\bd_k\|^2)$. For $\|\bd_k\|$, we have
\begin{equation}\label{eq:length_of_correctional_step}
\|\bd_k\| \leq \|G_k^T[G_kG_k^T]^{-1}\|\|c(\bx_k+\Delta\bx_k)-c_k-G_k\Delta\bx_k\| \leq \frac{L_G}{\sqrt{\kappa_{1,G}}}\Delta_k^2.
\end{equation}
Combining the last two results and using the fact that $\Delta_k\leq \Delta_{\max}$, we have
\begin{equation}\label{eq:cubic_const}
\|c_{s_k}-c_k-G_k\Delta\bx_k\| 
\leq
m L_G\|\Delta\bx_k\|\|\bd_k\|+m L_G\|\bd_k\|^2
\leq\left(\frac{\sqrt{m}L_G^2}{\sqrt{\kappa_{1,G}}}+\frac{\sqrt{m}  L_G^3\Delta_{\max}}{\kappa_{1,G}}\right)\Delta_k^3.
\end{equation}
Next, we analyze the first term in \eqref{eq:lemma_diff_w_soc_1}. 
For some points $\phi_1$ between $[\bx_k+\Delta\bx_k,\bx_k+\Delta\bx_k+\bd_k]$~and~$\phi_2$ between $[\bx_k,\bx_k+\Delta\bx_k]$, we have
\begin{align*}
f_{s_k} & = f(\bx_k+\Delta\bx_k+\bd_k) = f(\bx_k+\Delta\bx_k)+ \nabla f(\bx_k+\Delta\bx_k)^T\bd_k+\frac{1}{2}\bd_k^T\nabla^2 f(\phi_1)\bd_k \\
& = f_k+g_k^T\Delta\bx_k+\nabla f(\bx_k+\Delta\bx_k)^T\bd_k+\frac{1}{2}\Delta\bx_k^T\nabla^2 f(\phi_2)\Delta\bx_k +\frac{1}{2}\bd_k^T\nabla^2 f(\phi_1)\bd_k.
\end{align*}
Define $\tilde{\blambda}_k=-[G_kG_k^T]^{-1}G_k\nabla f(\bx_k+\Delta\bx_k)$. By the Taylor expansion, we have
\begin{align*} 
\nabla f(\bx_k+\Delta\bx_k)^T\bd_k & \stackrel{\mathclap{\eqref{def:correctional_step}}}{=}\; \tilde{\blambda}_k^T[c(\bx_k+\Delta\bx_k)-c_k-G_k\Delta\bx_k] \\
& =\sum_{i=1}^m\tilde{\blambda}_k^i[c^i(\bx_k+\Delta\bx_k)-c_k^i-(\nabla c_k^i)^T\Delta\bx_k] = \frac{1}{2}\sum_{i=1}^m\tilde{\blambda}_k^i\Delta\bx_k^T\nabla^2c^i(\phi_3^i)\Delta\bx_k,
\end{align*}
where the points $\{\phi_3^i\}_{i=1}^m$ are between $[\bx_k,\bx_k+\Delta\bx_k]$. Recall that $\barH_k=\bar{\nabla}^2f_k+\sum_{i=1}^{m}\barblambda_k^i\nabla^2c_k^i$. We combine the above two displays and have
\begin{align}
& \big| f_{s_k} - f_k  -\barg_k^T\Delta\bx_k-\frac{1}{2}\Delta\bx_k^T\barH_k\Delta\bx_k\big|\\
& = \bigg|(g_k-\barg_k)^T\Delta\bx_k+\frac{1}{2}\Delta\bx_k^T\left(\nabla^2 f(\phi_2)-\nabla^2f_k\right)\Delta\bx_k+\frac{1}{2}\Delta\bx_k^T\left(\nabla^2 f_k-\bar{\nabla}^2f_k\right)\Delta\bx_k  \\
& \quad\quad +\frac{1}{2}\sum_{i=1}^{m}\left(\blambda_k^i-\barblambda_k^i\right)\Delta\bx_k^T\nabla^2c_k^i\Delta\bx_k+\frac{1}{2}\sum_{i=1}^m\left(\tilde{\blambda}_k^i-\blambda_k^i\right)\Delta\bx_k^T\nabla^2c_k^i\Delta\bx_k \\
& \quad\quad +\frac{1}{2}\bd_k^T\nabla^2 f\left(\phi_1\right)\bd_k+\frac{1}{2}\sum_{i=1}^m\tilde{\blambda}_k^i\Delta\bx_k^T\left(\nabla^2c^i(\phi_3^i)-\nabla^2c_k^i\right)\Delta\bx_k\bigg|\\
& \leq \|g_k-\barg_k\|\|\Delta\bx_k\|+\frac{L_{\nabla^2 f}}{2}\|\Delta\bx_k\|^3+\frac{1}{2}\|\nabla^2 f_k-\bar{\nabla}^2f_k\|\|\Delta\bx_k\|^2\\
& \quad\quad + \frac{\sqrt{m}L_G}{2}(\|\blambda_k - \barblambda_k\| + \|\tilde{\blambda}_k-\blambda_k\|) \|\Delta\bx_k\|^2 + \frac{L_{\nabla f}}{2}\|\bd_k\|^2 + \frac{\sqrt{m}L_{\nabla^2 c}}{2}\|\tilde{\blambda}_k\| \|\Delta\bx_k\|^3\\
& \leq \rbr{\kappa_g+\frac{L_{\nabla^2 f}+\kappa_h}{2}}\Delta_k^3 + \frac{\sqrt{m}L_G}{2}(\|\blambda_k - \barblambda_k\| + \|\tilde{\blambda}_k-\blambda_k\|) \Delta_k^2 + \frac{L_{\nabla f}}{2}\|\bd_k\|^2 + \frac{\sqrt{m}L_{\nabla^2 c}\|\tilde{\blambda}_k\|}{2} \Delta_k^3\\
& \stackrel{\mathclap{\eqref{eq:length_of_correctional_step}}}{\leq} \rbr{\kappa_g+\frac{L_{\nabla^2 f}+\kappa_h}{2} + \frac{L_{\nabla f}L_G^2\Delta_{\max}}{2\kappa_{1,G}} + \frac{\sqrt{m}L_{\nabla^2 c}\|\tilde{\blambda}_k\|}{2} }\Delta_k^3 + \frac{\sqrt{m}L_G}{2}(\|\blambda_k - \barblambda_k\| + \|\tilde{\blambda}_k-\blambda_k\|) \Delta_k^2,
\end{align}
where the second inequality is by Assumption \ref{ass:1-1} and the third inequality is by the event $\A_k\cap\B_k$~and the fact that $\Delta\bx_k\leq \Delta_k$ (note that the SOC step is performed only when $\alpha=1$). 
Furthermore,~on~the event $\B_k$, it follows from  Assumption \ref{ass:1-1} that 
\begin{align*}
\|\blambda_k-\barblambda_k\| & \leq \|[G_kG_k^T]^{-1}G_k\|\|g_k-\barg_k\| \leq\frac{\kappa_g}{\sqrt{\kappa_{1,G}}}\Delta_k^2 \leq\frac{\kappa_g\Delta_{\max}}{\sqrt{\kappa_{1,G}}}\Delta_k,\\
\|\blambda_k-\tilde{\blambda}_k\| & \leq \|[G_kG_k^T]^{-1}G_k\|\|g_k-\nabla f(\bx_k+\Delta\bx_k)\| \leq \frac{L_{\nabla f}}{\sqrt{\kappa_{1,G}}}\Delta_k,\\
\|\tilde{\blambda}_k\| & \leq \|[G_kG_k^T]^{-1}G_k\| \|\nabla f(\bx_k+\Delta\bx_k\| \leq 
\frac{1}{\sqrt{\kappa_{1,G}}} (L_{\nabla f}\|\Delta\bx_k\| +\|g_k\|)\leq \frac{L_{\nabla f}\Delta_{\max}+\kappa_{\nabla f}}{\sqrt{\kappa_{1,G}}}.
\end{align*}
Combining the above results, we have
\begin{multline}\label{eq:cubic_quad}
\big|f_{s_k}-f_k  -\barg_k^T\Delta\bx_k-\frac{1}{2}\Delta\bx_k^T\barH_k\Delta\bx_k\big|  \leq \left(\kappa_g+\frac{L_{\nabla^2 f}+\kappa_h}{2}+\frac{L_{\nabla f}L_G^2\Delta_{\max}}{2\kappa_{1,G}}\right)\Delta_k^3 \\
+ \rbr{\frac{\sqrt{m}L_{\nabla^2 c}(L_{\nabla f}\Delta_{\max} + \kappa_{\nabla f}) }{2\sqrt{\kappa_{1,G}}} + \frac{\sqrt{m}L_G(\kappa_g\Delta_{\max} + L_{\nabla f})}{2\sqrt{\kappa_{1,G}}}}\Delta_k^3.
\end{multline}
We complete the proof by combining \eqref{eq:lemma_diff_w_soc_1}, \eqref{eq:cubic_const} and \eqref{eq:cubic_quad}.
\end{proof}

The next lemma demonstrates that when the current iterate $\bx_k$ is not a first-order stationary~point (i.e., $\|\nabla\L_k\|>0$), the estimates of objective models are accurate, and the trust-region~radius~is~sufficiently small, then the $k$-th iteration is guaranteed to be successful without performing the~SOC~step. Furthermore, the reduction in the merit function is of the order $\cO(\|\nabla\L_k\|\Delta_k)$.

\begin{lemma}\label{lemma:guarantee_succ_step_KKT}

Under Assumptions \ref{ass:1-1}, \ref{ass:4-1}, and the event $\A_k\cap\B_k\cap\C_k$, for $k\geq\barK$, if
\begin{equation}\label{delta:lemma_guarantee_succ_step_KKT}
\|\nabla\L_k\|\geq \max\left\{ \kappa_B,\frac{4\kappa_f\max\{1,\Delta_{\max}\}+8\Upsilon_1}{\kappa_{fcd}(1-\eta)}\right\}\Delta_k+\kappa_g\max\{1,\Delta_{\max}\}\Delta_k
\end{equation}
with $\Upsilon_1$ defined in Lemma \ref{lemma:diff_ared_pred_wo_corr_step}, then the $k$-th iteration is successful without computing the SOC step. Furthermore, 
\begin{equation*}
\L_{\barmu_{\barK}}^{k+1}-\L_{\barmu_{\barK}}^k \leq -\Upsilon_3\|\nabla\L_k\|\Delta_k,
\end{equation*}
where 
\begin{equation*}
\Upsilon_3=\frac{3\kappa_{fcd}}{8}\cdot \max\left\{\frac{\kappa_B}{\kappa_g\max\{1,\Delta_{\max}\}+\kappa_B},\frac{4\kappa_f\max\{1,\Delta_{\max}\}+8\Upsilon_1}{\{(1-\eta)\kappa_{fcd}\kappa_g+4\kappa_f\}\max\{1,\Delta_{\max}\}+8\Upsilon_1}\right\}.
\end{equation*}

\end{lemma}

\begin{proof}

To prove the $k$-th iteration is successful, it suffices to show that \eqref{def:acc1} holds and $\text{Ared}_k/\text{Pred}_k\geq\eta$. Since \eqref{delta:lemma_guarantee_succ_step_KKT} implies \eqref{eq:lemma3:delta}, Lemma \ref{lemma:line6_not_hold} indicates that \eqref{def:acc1} holds. Now, we show $\text{Ared}_k/\text{Pred}_k\geq\eta$ holds without performing the SOC step (thus $\bx_{s_k}=\bx_k+\Delta\bx_k$). 
On the event $\B_k$, we have~$\|\bar{\nabla}\L_k\|\geq\|\nabla\L_k\|-\kappa_g\max\{1,\Delta_{\max}\}\Delta_k$ and \eqref{delta:lemma_guarantee_succ_step_KKT} implies
\begin{equation}\label{eq:guarantee_succ_step_estKKT}
\|\bar{\nabla}\L_k\| \geq \max\left\{ \kappa_B,\frac{4\kappa_f\max\{1,\Delta_{\max}\}+8\Upsilon_1}{\kappa_{fcd}(1-\eta)}\right\}\Delta_k.
\end{equation}
Define a local model of $\mL_{\barmu_k}^k$ along the direction $\bs\in\mR^d$ as $m_{\barmu_k}^k(\bs)=f_k+\barg_k^T\bs+\frac{1}{2}\bs^T\barH_k\bs+\barmu_k\|c_k+G_k\bs\|$, using the definitions of $\text{Ared}_k$ in \eqref{def:Ared_k} and $\text{Pred}_k$ in \eqref{def:Pred_k}, we have
\begin{align*}
\frac{\text{Ared}_k}{\text{Pred}_k} 
& =\frac{\bar{\L}_{\barmu_{\barK}}^{s_k}-\bar{\L}_{\barmu_{\barK}}^k}{\text{Pred}_k}\\
& = \frac{\bar{\L}_{\barmu_{\barK}}^{s_k}-\L_{\barmu_{\barK}}^{s_k}+\L_{\barmu_{\barK}}^{s_k}-m_{\barmu_{\barK}}^k(\Delta\bx_k)+m_{\barmu_{\barK}}^k(\Delta\bx_k)-m_{\barmu_{\barK}}^k(\b0)+m_{\barmu_{\barK}}^k(\b0)-\L_{\barmu_{\barK}}^k+\L_{\barmu_{\barK}}^k-\bar{\L}_{\barmu_{\barK}}^{k}}{\text{Pred}_k}\notag\\
& = \frac{\bar{\L}_{\barmu_{\barK}}^{s_k}-\L_{\barmu_{\barK}}^{s_k}+\L_{\barmu_{\barK}}^{s_k}-m_{\barmu_{\barK}}^k(\Delta\bx_k)+\L_{\barmu_{\barK}}^k-\bar{\L}_{\barmu_{\barK}}^{k}}{\text{Pred}_k}+1,
\end{align*}
where we have used $\text{Pred}_k= m_{\barmu_{\barK}}^k(\Delta\bx_k)-m_{\barmu_{\barK}}^k(\b0)$ and $m_{\barmu_{\barK}}^k(\b0)-\L_{\barmu_{\barK}}^k=0$. Therefore,
\begin{equation}\label{eq:Ared/Pred-1}
\left|\frac{\text{Ared}_k}{\text{Pred}_k}-1\right|\leq\frac{\big|\bar{\L}_{\barmu_{\barK}}^{s_k}-\L_{\barmu_{\barK}}^{s_k}\big|+\big|\L_{\barmu_{\barK}}^{s_k}-m_{\barmu_{\barK}}^k(\Delta\bx_k)\big|+\big|\L_{\barmu_{\barK}}^k-\bar{\L}_{\barmu_{\barK}}^{k}\big|}{\left|\text{Pred}_k\right|}.
\end{equation}
By the algorithm design and $\max\{1,\|\barH_k\|\}\leq\kappa_B$, we have 
\begin{equation}\label{eq:lemma6_pred}
\text{Pred}_k 
\stackrel{\eqref{eq:threshold_Predk}}{\leq}
-\frac{\kappa_{fcd}}{2}\|\bar{\nabla}\L_k\|\min\left\{\Delta_k,\frac{\|\bar{\nabla}\L_k\|}{\|\barH_k\|}\right\}
\stackrel{\eqref{eq:guarantee_succ_step_estKKT}}{=}
-\frac{\kappa_{fcd}}{2}\|\bar{\nabla}\L_k\|\Delta_k.
\end{equation}
Since $|\L_{\barmu_\barK}^{s_k}-\barL_{\barmu_\barK}^{s_k}|=|f_{s_k}-\barf_{s_k}|$ and $|\L_{\barmu_\barK}^k-\barL_{\barmu_\barK}^k|=|f_k-\barf_k|$, on the event $\C_k$, we have 
\begin{equation*}
\big|\L_{\barmu_\barK}^{s_k}-\barL_{\barmu_\barK}^{s_k}\big|+\big|\L_{\barmu_\barK}^k-\barL_{\barmu_\barK}^k\big|\leq 2\kappa_f\max\{1,\Delta_{\max}\}\Delta_k^2.
\end{equation*} 
Since $m_{\barmu_{\barK}}^k(\Delta\bx_k)=\L_{\barmu_\barK}^k+\text{Pred}_k$, Lemma~\ref{lemma:diff_ared_pred_wo_corr_step} gives
\begin{equation}\label{eq:lemma6_ared_pred}
\left|\L_{\barmu_{\barK}}^{s_k}-m_{\barmu_{\barK}}^k(\Delta\bx_k)\right|=\left|\L_{\barmu_{\barK}}^{s_k}-\L_{\barmu_\barK}^k-\text{Pred}_k\right| \leq \Upsilon_1\Delta_k^2.
\end{equation}
Combining the last four displays, we have
\begin{equation*}
\left|\frac{\text{Ared}_k}{\text{Pred}_k}-1\right|\leq\frac{(4\kappa_f\max\{1,\Delta_{\max}\}+2\Upsilon_1)\Delta_k}{\kappa_{fcd}\|\bar{\nabla}\L_k\|} \stackrel{\eqref{eq:guarantee_succ_step_estKKT}}{\leq} 1-\eta,
\end{equation*}
equivalently, $\text{Ared}_k/\text{Pred}_k\geq \eta$. Since the $k$-th iteration at Line 12 of Algorithm \ref{Alg:STORM} is already~successful, the SOC step will not be computed. Next, we analyze the reduction in the merit function. Combining \eqref{eq:lemma6_pred} and \eqref{eq:lemma6_ared_pred}, and noting that $\bx_{k+1}=\bx_{s_k}$, we have
\begin{equation}\label{eq:lemma6_reduction_1}
\L_{\barmu_{\barK}}^{k+1}-\L_{\barmu_{\barK}}^k \leq \text{Pred}_k+\Upsilon_1\Delta_k^2 \leq -\frac{\kappa_{fcd}}{2}\|\bar{\nabla}\L_k\|\Delta_k+\Upsilon_1\Delta_k^2 \stackrel{\eqref{eq:guarantee_succ_step_estKKT}}{\leq} -\frac{3\kappa_{fcd}}{8}\|\bar{\nabla}\L_k\|\Delta_k.
\end{equation}
Since $\|\bar{\nabla}\L_k\|\geq \|\nabla\L_k\| - \kappa_g\max\{1,\Delta_{\max}\}\Delta_k$ and \eqref{delta:lemma_guarantee_succ_step_KKT} implies that
\begin{multline}\label{eq:lemma6_KKT}
\kappa_g\max\{1,\Delta_{\max}\}\Delta_k \\
\leq  \min\left\{ \frac{\kappa_g\max\{1,\Delta_{\max}\}}{\kappa_g\max\{1,\Delta_{\max}\}+\kappa_B},\frac{(1-\eta)\kappa_{fcd}\kappa_g\max\{1,\Delta_{\max}\}}{\{(1-\eta)\kappa_{fcd}\kappa_g+4\kappa_f\}\max\{1,\Delta_{\max}\}+8\Upsilon_1}\right\}\|\nabla\L_k\|,
\end{multline}
we complete the proof by combining \eqref{eq:lemma6_reduction_1} and \eqref{eq:lemma6_KKT}.
\end{proof}

Next, we consider $\alpha=1$ and prove that when the current iterate $\bx_k$ is not a second-order stationary point (i.e., $\tau_k^+>0$), the estimates of objective models are accurate, and the \mbox{trust-region}~radius~is sufficiently small, then the $k$-th iteration is guaranteed to be successful.
Here, we use $\C_k'$ to denote the event that the accurate estimates of objective values are regenerated when computing the~SOC~step. If the SOC step is not computed, we simply assume $\C_k'$ holds for consistency.

\begin{lemma}\label{lemma:guarantee_succ_step_eigen}

Under Assumptions \ref{ass:1-1}, \ref{ass:4-1}, and the event $\A_k\cap\B_k\cap\C_k\cap\C_k'$ with $\alpha=1$, for $k\geq\barK$, if
\begin{equation}\label{delta:lemma_guarantee_succ_step_eigen}
\tau_k^+ \geq \max\left\{ \eta, \frac{4\kappa_f\max\{1,\Delta_{\max}\}+2\Upsilon_1+2\Upsilon_2}{(1-\eta)\kappa_{fcd}\min\{1,r\}}\right\}\Delta_k+\kappa_H\Delta_k,
\end{equation}
with $\Upsilon_1$ defined in Lemma \ref{lemma:diff_ared_pred_wo_corr_step} and $\Upsilon_2$ defined in Lemma \ref{lemma:diff_ared_pred_w_corr_step}, then the $k$-th iteration is successful.
\end{lemma}

\begin{proof}

We prove that Line 6 in Algorithm \ref{Alg:STORM} is not triggered and $\text{Ared}_k/\text{Pred}_k\geq\eta$. Since \eqref{delta:lemma_guarantee_succ_step_eigen}~implies \eqref{eq:lemma3:delta}, Lemma \ref{lemma:line6_not_hold} indicates that Line 6 will not be triggered. We only need to show $\text{Ared}_k/\text{Pred}_k\geq\eta$.
On the event $\A_k\cap\B_k$, we have $\bartau_k^+\geq \tau_k^+-\kappa_H\Delta_k$ (cf. Lemma \ref{lemma:tau_accurate}). Thus, \eqref{delta:lemma_guarantee_succ_step_eigen} leads to
\begin{equation}\label{eq:guarantee_succ_step_esteigen}
\bartau_k^+ \geq \max\left\{ \eta, \frac{4\kappa_f\max\{1,\Delta_{\max}\}+2\Upsilon_1+2\Upsilon_2}{(1-\eta)\kappa_{fcd}\min\{1,r\}}\right\}\Delta_k.
\end{equation}
We first consider the case when $\|c_k\|>r$. By \eqref{eq:threshold_Predk}, we have for both the gradient and eigen steps,
\begin{equation}\label{eq:lemma7_pred_1}
\text{Pred}_k\leq - \frac{\kappa_{fcd}}{2}\bartau_k^+\|c_k\|\Delta_k \leq - \frac{r\kappa_{fcd}}{2}\bartau_k^+\Delta_k.
\end{equation}
Since $\alpha=1$, we apply the event $\C_k$ and have $|\L_{\barmu_\barK}^{s_k}-\barL_{\barmu_\barK}^{s_k}|  + |\L_{\barmu_\barK}^k-\barL_{\barmu_\barK}^k|\leq 2\kappa_f\Delta_k^3 \leq 2\kappa_f\Delta_{\max}\Delta_k^2$. 
When the SOC step is not performed, Lemma \ref{lemma:diff_ared_pred_wo_corr_step} implies \eqref{eq:lemma6_ared_pred} holds. 
Combined with \eqref{eq:Ared/Pred-1}, we~have \;\;
\begin{equation*}
\left|\frac{\text{Ared}_k}{\text{Pred}_k}-1\right| \leq \frac{(4\kappa_f\Delta_{\max}+2\Upsilon_1)\Delta_k}{r\kappa_{fcd}\bartau_k^+} \stackrel{\eqref{eq:guarantee_succ_step_esteigen}}{\leq} 1-\eta,
\end{equation*}
equivalently, $\text{Ared}_k/\text{Pred}_k\geq \eta$. Next, we consider $\|c_k\|\leq r$. If $\text{Ared}_k/\text{Pred}_k\geq \eta$ holds when~the~SOC step is not performed, there is nothing to prove. Otherwise, the condition $\|c_k\|\leq r$ will trigger the SOC step and $\bx_{s_k}=\bx_k+\Delta\bx_k+\bd_k$. On the event $\C_k\cap\C_k'$, we have $ |\L_{\barmu_\barK}^{s_k}-\barL_{\barmu_\barK}^{s_k}|+|\L_{\barmu_\barK}^k-\barL_{\barmu_\barK}^k|\leq 2\kappa_f\Delta_k^3$. Meanwhile, \eqref{eq:threshold_Predk} implies $\text{Pred}_k\leq-({\kappa_{fcd}}/{2})\bartau_k^+\Delta_k^2$.
Combining with Lemma \ref{lemma:diff_ared_pred_w_corr_step} and \eqref{eq:Ared/Pred-1}, we have 
\begin{equation*}
\left|\frac{\text{Ared}_k}{\text{Pred}_k}-1\right| \leq \frac{(4\kappa_f+2\Upsilon_2)\Delta_k}{\kappa_{fcd}\bartau_k^+} \stackrel{\eqref{eq:guarantee_succ_step_esteigen}}{\leq} 1-\eta,
\end{equation*}
which completes the proof.
\end{proof}

In the following lemma, we demonstrate that for both first and second-order stationarity, if the estimates of objective values are accurate and the $k$-th iteration is successful, then the reduction~in~the merit function is proportional to $\Delta_k^3$.

\begin{lemma}\label{lemma8:reduction_of_successful_iter}

Under Assumptions \ref{ass:1-1}, \ref{ass:4-1}, and the event $\C_k\cap\C_k'$, for $k\geq \barK$, if the $k$-th iteration is successful, then
\begin{equation*}
\L_{\barmu_{\barK}}^{k+1}-\L_{\barmu_{\barK}}^k\leq-\frac{3\kappa_{fcd}}{8\max\{1,\Delta_{\max}\}}\eta^3\Delta_k^3.
\end{equation*}
\end{lemma}

\begin{proof}

If $\alpha=0$, a successful iteration implies $\text{Ared}_k/\text{Pred}_k\geq\eta$, $\|\bar{\nabla}\L_k\|/\max\{1,\|\barH_k\|\}\geq \eta\Delta_k$,~and
\begin{equation} \label{eq:pred_1st_order}
\text{Pred}_k \leq -\frac{\kappa_{fcd}}{2}\|\bar{\nabla}\L_k\|\min\left\{ \Delta_k,\frac{\|\bar{\nabla}\L_k\|}{\|\barH_k\|}\right\} \leq -\frac{\kappa_{fcd}}{2} \eta^2\Delta_k^2.
\end{equation}
On the event $\C_k$, $\left| f_{s_k}-\barf_{s_k} \right| + \left| f_k-\barf_k \right| \leq 2\kappa_f\Delta_k^2$. Thus,
\begin{align}\label{eq:lemma8_1}
\L_{\barmu_{\barK}}^{k+1}-\L_{\barmu_{\barK}}^k 
& \leq \big| \L_{\barmu_{\barK}}^{k+1}-\bar{\L}_{\barmu_{\barK}}^{k+1} \big| + \text{Ared}_k + \big| \bar{\L}_{\barmu_{\barK}}^k-\L_{\barmu_{\barK}}^k \big| \nonumber \\
& \leq \left| f_{s_k}-\barf_{s_k} \right| + \eta \cdot \text{Pred}_k + \left| f_k-\barf_k \right|
\leq 2\kappa_f\Delta_k^2 -\frac{\kappa_{fcd}}{2}\eta^3\Delta_k^2 \leq -\frac{3\kappa_{fcd}}{8}\eta^3\Delta_k^2,
\end{align} 
where we have used the definition of $\kappa_f$ in the last inequality. If $\alpha=1$, a successful iteration implies $\max\{\frac{\|\bar{\nabla}\L_k\|}{\max\{1,\|\barH_k\|\}},\bartau_k^+\}\geq \eta\Delta_k$ and \eqref{eq:threshold_Predk} implies
\begin{equation}\label{eq:pred_2nd_order}
\text{Pred}_k\leq -\frac{\kappa_{fcd}}{2\max\{1,\Delta_{\max}\}}\eta^2\Delta_k^3.
\end{equation}
On the event $\C_k\cap\C_k'$, we have $\left| f_{s_k}-\barf_{s_k} \right| + \left| f_k-\barf_k \right| \leq 2\kappa_f\Delta_k^3 $ and
\begin{equation}\label{nequ:4}
\L_{\barmu_{\barK}}^{k+1}-\L_{\barmu_{\barK}}^k \leq 2\kappa_f\Delta_k^3 -\frac{\kappa_{fcd}}{2\max\{1,\Delta_{\max}\}}\eta^3\Delta_k^3 \leq -\frac{3\kappa_{fcd}}{8\max\{1,\Delta_{\max}\}}\eta^3\Delta_k^3,
\end{equation}
where the last inequality is by the definition of $\kappa_f$. Combining \eqref{eq:lemma8_1} and \eqref{nequ:4} completes the proof.
\end{proof}

In the following few lemmas, we investigate the global convergence of Algorithm \ref{Alg:STORM} by \mbox{leveraging}~the reduction in a potential function given by
\begin{equation*}
\Phi_{\barmu_{\barK}}^k=\nu\L_{\barmu_{\barK}}^k+\frac{1-\nu}{2}\Delta_k^3+\frac{1-\nu}{2}\barepsilon_k,
\end{equation*}
where $\nu\in(0,1)$ is a constant satisfying ($\Upsilon_3$ is defined in Lemma \ref{lemma:guarantee_succ_step_KKT})
\begin{equation}\label{eq:nu_and_1-nu}
\frac{\nu}{1-\nu}\geq \max\left\{\frac{4\gamma^3\max\{1,\Delta_{\max}\}}{\min\{\eta^3\kappa_{fcd},\Upsilon_3\}},\frac{2\gamma}{\eta}\right\}.
\end{equation}
We first consider the case when the estimates of objective models~are accurate.

\begin{lemma}\label{lemma:ABCC'_hold}

Under Assumptions \ref{ass:1-1}, \ref{ass:4-1}, and the event $\A_k\cap\B_k\cap\C_k\cap\C_k'$, for $k\geq\barK$, we have
\begin{equation}\label{eq:lemma_phi_reduction_all_acc}
\Phi_{\barmu_{\barK}}^{k+1}-\Phi_{\barmu_{\barK}}^k \leq \frac{1-\nu}{2}\left(\frac{1}{\gamma^3}-1\right)\Delta_k^3+\frac{1-\nu}{2}\left(\frac{1}{\gamma}-1\right)\barepsilon_k.
\end{equation}

\end{lemma}

\begin{proof}
We separate the analysis into two cases based on the condition \eqref{delta:lemma_guarantee_succ_step_KKT}.

\vskip3pt
\noindent \textbf{Case 1: \eqref{delta:lemma_guarantee_succ_step_KKT} holds.}
In this case, Lemma \ref{lemma:guarantee_succ_step_KKT} suggests that the $k$-th iteration is successful without computing the SOC step.~We further separate a successful step into a reliable and unreliable step.

\vskip3pt
\noindent $\bullet$ \textbf{Case 1a: reliable iteration.} We have
\begin{multline*}
\L_{\barmu_{\barK}}^{k+1} - \L_{\barmu_{\barK}}^k \stackrel{\text{Lemma }\ref{lemma:diff_ared_pred_wo_corr_step}}{\leq} \text{Pred}_k+\Upsilon_1\Delta_k^2 \stackrel{(\ref{nequ:3}\text{b})}{\leq} \frac{1}{2}\text{Pred}_k -\frac{1}{2}\barepsilon_k + \Upsilon_1\Delta_k^2 \\
\stackrel{\eqref{eq:lemma6_pred} }{\leq} -\frac{\kappa_{fcd}}{4}\|\bar{\nabla}\L_k\|\Delta_k -\frac{1}{2}\barepsilon_k + \Upsilon_1\Delta_k^2 
\stackrel{\eqref{eq:guarantee_succ_step_estKKT}}{\leq} -\frac{\kappa_{fcd}}{8}\|\bar{\nabla}\L_k\|\Delta_k -\frac{1}{2}\barepsilon_k  \stackrel{\eqref{eq:lemma6_KKT}}{\leq} -\frac{1}{3}\Upsilon_3\|\nabla\L_k\|\Delta_k -\frac{1}{2}\barepsilon_k.
\end{multline*}
For a reliable iteration, $\Delta_{k+1}\leq\gamma\Delta_k$ and $\barepsilon_{k+1}=\gamma\barepsilon_k$. Since \eqref{eq:nu_and_1-nu} implies $\frac{1-\nu}{2}(\gamma-1)\barepsilon_k\leq \frac{1}{4}\nu\barepsilon_k$,~we~have
\begin{align}\label{eq:lemma9_case1_reliable}
\Phi_{\barmu_{\barK}}^{k+1}-\Phi_{\barmu_{\barK}}^k & \leq -\frac{1}{3}\nu\Upsilon_3\|\nabla\L_k\|\Delta_k -\frac{1}{2}\nu\barepsilon_k+\frac{1-\nu}{2}(\gamma^3-1)\Delta_k^3+\frac{1-\nu}{2}(\gamma-1)\barepsilon_k \notag\\
& \leq -\frac{1}{3}\nu\Upsilon_3\|\nabla\L_k\|\Delta_k -\frac{1}{4}\nu\barepsilon_k+\frac{1-\nu}{2}(\gamma^3-1)\Delta_k^3.
\end{align}
\noindent $\bullet$ \textbf{Case 1b: unreliable iteration.} Combining Lemma~\ref{lemma:guarantee_succ_step_KKT}, $\Delta_{k+1}\leq\gamma\Delta_k$, and $\barepsilon_{k+1}=\barepsilon_k/\gamma$, we have 
\begin{equation}\label{eq:lemma9_case1_unreliable}
\Phi_{\barmu_{\barK}}^{k+1}-\Phi_{\barmu_{\barK}}^k \leq -\nu\Upsilon_3\|\nabla\L_k\|\Delta_k + \frac{1-\nu}{2}(\gamma^3-1)\Delta_k^3 + \frac{1-\nu}{2}\left(\frac{1}{\gamma} - 1\right)\barepsilon_k.
\end{equation}
Combining both \textbf{Case 1a} and \textbf{Case 1b} in \eqref{eq:lemma9_case1_reliable} and \eqref{eq:lemma9_case1_unreliable}, and noting that $\frac{1}{4}\nu\barepsilon_k\geq\frac{1-\nu}{2}\big(1-\frac{1}{\gamma} \big)\barepsilon_k$ as implied by \eqref{eq:nu_and_1-nu}, we have
\begin{equation}\label{eq:lemma9_case1_1}
\Phi_{\barmu_{\barK}}^{k+1}-\Phi_{\barmu_{\barK}}^k  \leq -\frac{1}{3}\nu\Upsilon_3\|\nabla\L_k\|\Delta_k + \frac{1-\nu}{2}(\gamma^3-1)\Delta_k^3 + \frac{1-\nu}{2}\left(\frac{1}{\gamma} - 1\right)\barepsilon_k.
\end{equation}
Since $\kappa_B\geq 1$, \eqref{delta:lemma_guarantee_succ_step_KKT} implies $\|\nabla\L_k\| \geq \Delta_k$. Thus, we know
\begin{equation}\label{eq:lemma9_case1_2}
-\frac{1}{6}\nu\Upsilon_3\|\nabla\L_k\|\Delta_k+ \frac{1-\nu}{2}(\gamma^3-1)\Delta_k^3 \leq -\frac{\nu\Upsilon_3}{6\Delta_{\max}}\Delta_k^3+ \frac{1-\nu}{2}(\gamma^3-1)\Delta_k^3 \stackrel{\eqref{eq:nu_and_1-nu}}{\leq} 0.
\end{equation}
Combining \eqref{eq:lemma9_case1_1} and \eqref{eq:lemma9_case1_2}, we know for \textbf{Case 1} that
\begin{equation}\label{eq:lemma9_case1_final}
\Phi_{\barmu_{\barK}}^{k+1}-\Phi_{\barmu_{\barK}}^k  \leq -\frac{1}{6}\nu\Upsilon_3\|\nabla\L_k\|\Delta_k + \frac{1-\nu}{2}\left(\frac{1}{\gamma} - 1\right)\barepsilon_k.
\end{equation}
\noindent \textbf{Case 2: \eqref{delta:lemma_guarantee_succ_step_KKT} does not hold.} In this case, the $k$-th iteration can be successful (reliable or unreliable) or unsuccessful.

\vskip3pt
\noindent$\bullet$ \textbf{Case 2a: reliable iteration.} We have
\begin{equation*}
\L_{\barmu_{\barK}}^{k+1} - \L_{\barmu_{\barK}}^k 
\stackrel{(\ref{nequ:3}\text{b}),\eqref{eq:lemma8_1}}{\leq} \left|f_{s_k} - \barf_{s_k}\right| + \left|f_k - \barf_k\right| + \frac{1}{2}\eta  \text{Pred}_k -\frac{1}{2} \eta \barepsilon_k.
\end{equation*}
When $\alpha=0$ (i.e., first-order stationarity), we apply the event $\C_k$ and the definition of $\kappa_f$, and obtain
\begin{equation}\label{eq:lemma_ABCC'_case2_obj_value_1}
\left|f_{s_k} - \barf_{s_k}\right| + \left|f_k - \barf_k\right| \leq 2\kappa_f\Delta_k^2 \leq \frac{\kappa_{fcd}}{8\max\{1,\Delta_{\max}\}}\eta^3\Delta_k^2.
\end{equation}
Since the iteration is successful, \eqref{eq:pred_1st_order} holds. Using $\Delta_k\leq\Delta_{\max}$,
 \begin{equation*}
\L_{\barmu_{\barK}}^{k+1} - \L_{\barmu_{\barK}}^k 
\leq \frac{\kappa_{fcd}}{8\max\{1,\Delta_{\max}\}}\eta^3\Delta_k^2 - \frac{\kappa_{fcd}}{4}\eta^3 \Delta_k^2 -\frac{1}{2} \eta \barepsilon_k \leq - \frac{\kappa_{fcd}}{8\max\{1,\Delta_{\max}\}}\eta^3 \Delta_k^3 -\frac{1}{2} \eta \barepsilon_k.
\end{equation*}
When $\alpha=1$ (i.e., second-order stationarity), we apply the event $\C_k\cap\C_k'$ and have
\begin{equation}\label{eq:lemma_ABCC'_case2_obj_value_2}
\left|f_{s_k} - \barf_{s_k}\right| + \left|f_k - \barf_k\right| \leq 2\kappa_f\Delta_k^3 \leq \frac{\kappa_{fcd}}{8\max\{1,\Delta_{\max}\}}\eta^3\Delta_k^3,
\end{equation}
which together with \eqref{eq:pred_2nd_order} yields
 \begin{equation*}
\L_{\barmu_{\barK}}^{k+1} - \L_{\barmu_{\barK}}^k \leq - \frac{\kappa_{fcd}}{8\max\{1,\Delta_{\max}\}}\eta^3 \Delta_k^3 -\frac{1}{2} \eta \barepsilon_k.
\end{equation*}
Thus, $\alpha=0$ and $\alpha=1$ share the same bound for $\L_{\barmu_{\barK}}^{k+1} - \L_{\barmu_{\barK}}^k$.~For a reliable iteration,~\mbox{$\Delta_{k+1}\leq\gamma\Delta_k$}~and $\barepsilon_{k+1}=\gamma\barepsilon_k$. Since \eqref{eq:nu_and_1-nu} implies $\frac{1-\nu}{2}(\gamma-1)\barepsilon_k\leq \frac{1}{4}\nu\eta\barepsilon_k$, we have
\begin{align*}
\Phi_{\barmu_{\barK}}^{k+1}-\Phi_{\barmu_{\barK}}^k & \leq - \frac{\nu \kappa_{fcd}}{8\max\{1,\Delta_{\max}\}}\eta^3 \Delta_k^3 -\frac{1}{2} \nu \eta \barepsilon_k +\frac{1-\nu}{2}(\gamma^3-1)\Delta_k^3 + \frac{1-\nu}{2}(\gamma-1)\barepsilon_k \\
& \leq - \frac{\nu \kappa_{fcd}}{8\max\{1,\Delta_{\max}\}}\eta^3 \Delta_k^3 - \frac{1}{4}\nu\eta \barepsilon_k + \frac{1-\nu}{2}(\gamma^3-1)\Delta_k^3.
\end{align*}
\noindent$\bullet$ \textbf{Case 2b: unreliable iteration.}~Combining Lemma \ref{lemma8:reduction_of_successful_iter}, $\Delta_{k+1}\leq\gamma\Delta_k$, and $\barepsilon_{k+1}=\barepsilon_k/\gamma$,~we~have 
\begin{equation*}
\Phi_{\barmu_{\barK}}^{k+1}-\Phi_{\barmu_{\barK}}^k \leq - \frac{3\nu \kappa_{fcd}}{8\max\{1,\Delta_{\max}\}}\eta^3 \Delta_k^3 +\frac{1-\nu}{2}(\gamma^3-1)\Delta_k^3 + \frac{1-\nu}{2}\left(\frac{1}{\gamma}-1\right)\barepsilon_k.
\end{equation*}
\noindent$\bullet$ \textbf{Case 2c: unsuccessful iteration.} Here $\bx_{k+1}=\bx_k$, $\Delta_{k+1}=\Delta_k/\gamma$, and $\barepsilon_{k+1}=\barepsilon_k/\gamma$. Thus,
\begin{equation}\label{eq:Phi_unsuccessful}
\Phi_{\barmu_{\barK}}^{k+1}-\Phi_{\barmu_{\barK}}^k \leq \frac{1-\nu}{2}\left(\frac{1}{\gamma^3}-1\right)\Delta_k^3 + \frac{1-\nu}{2}\left(\frac{1}{\gamma}-1\right)\barepsilon_k.
\end{equation}
Since \eqref{eq:nu_and_1-nu} implies
\begin{equation*}
- \frac{\nu \kappa_{fcd}}{8\max\{1,\Delta_{\max}\}}\eta^3 \Delta_k^3 +\frac{1-\nu}{2}(\gamma^3-1)\Delta_k^3 \leq \frac{1-\nu}{2}\left(\frac{1}{\gamma^3}-1\right)\Delta_k^3
\end{equation*}
and 
\begin{equation}\label{eq:lemma_ABCC'_epsilon}
\frac{1}{4}\nu\eta\barepsilon_k \geq \frac{1-\nu}{2}\left(1-\frac{1}{\gamma}\right)\barepsilon_k,
\end{equation}
we combine \textbf{Cases 2a, 2b, 2c} together and know that the result \eqref{eq:Phi_unsuccessful} for \textbf{Case 2c} also~holds~for~\textbf{Cases 2a} and \textbf{2b} as well. Note that in \textbf{Case 1}, $\Delta_k\leq\|\nabla\L_k\|$ and \eqref{eq:nu_and_1-nu} together imply that
\begin{equation}\label{last_eq}
-\frac{1}{6}\nu\Upsilon_3\|\nabla\L_k\|\Delta_k \leq \frac{1-\nu}{2}\left(\frac{1}{\gamma^3}-1\right)\Delta_k^3.
\end{equation}
The proof is complete by combining \eqref{eq:lemma9_case1_final} for \textbf{Case 1} and \eqref{eq:Phi_unsuccessful} for \textbf{Case 2}.
\end{proof}

We now examine the reduction in $\Phi_{\barmu_{\barK}}^k$ when not all estimates are accurate.

\begin{lemma}\label{lemma:Phi_(ABCC')^c}

Under Assumptions \ref{ass:1-1}, \ref{ass:4-1}, and the event $(\A_k\cap\B_k\cap\C_k\cap\C_k')^c$, for $k\geq \barK$, we have
\begin{equation}\label{eq:lemma_phi_reduction_not_all_acc}
\Phi_{\barmu_{\barK}}^{k+1}-\Phi_{\barmu_{\barK}}^k \leq \nu\cbr{\left|f_{s_k}-\barf_{s_k} \right| + \left|f_k-\barf_k \right|} + \frac{1-\nu}{2}\left(\frac{1}{\gamma^3}-1\right)\Delta_k^3 + \frac{1-\nu}{2}\left(\frac{1}{\gamma}-1\right)\barepsilon_k,
\end{equation}
where $\bx_{s_k}=\bx_k+\Delta\bx_k$ if the SOC step is not performed and $\bx_{s_k}=\bx_k+\Delta\bx_k+\bd_k$ if the SOC step is performed.
\end{lemma}

\begin{proof}

We consider the following three cases.

\vskip3pt
\noindent$\bullet$ \textbf{Case 1: reliable iteration.} The proof is similar to \textbf{Case 2a} in Lemma \ref{lemma:ABCC'_hold}. Since $\C_k$ and~$\C_k'$~may not hold, \eqref{eq:lemma_ABCC'_case2_obj_value_1} and \eqref{eq:lemma_ABCC'_case2_obj_value_2} are not guaranteed. Therefore, we have
\begin{equation*}
\Phi_{\barmu_{\barK}}^{k+1}-\Phi_{\barmu_{\barK}}^k  \leq \nu\cbr{\left|f_{s_k}-\barf_{s_k} \right| + \left|f_k-\barf_k \right|} - \frac{\nu \kappa_{fcd}}{4\max\{1,\Delta_{\max}\}}\eta^3 \Delta_k^3 - \frac{1}{4}\nu\eta \barepsilon_k + \frac{1-\nu}{2}(\gamma^3-1)\Delta_k^3.
\end{equation*}
\noindent$\bullet$ \textbf{Case 2: unreliable iteration.} We follow the proof of \textbf{Case 2b} in Lemma \ref{lemma:ABCC'_hold} and Lemma~\ref{lemma8:reduction_of_successful_iter}, and have
\begin{multline*}
\Phi_{\barmu_{\barK}}^{k+1}-\Phi_{\barmu_{\barK}}^k \leq \nu\left|f_{s_k}-\barf_{s_k} \right| + \nu\left|f_k-\barf_k \right|\\
- \frac{\nu \kappa_{fcd}}{2\max\{1,\Delta_{\max}\}}\eta^3 \Delta_k^3 + \frac{1-\nu}{2}(\gamma^3-1)\Delta_k^3 + \frac{1-\nu}{2}\left(\frac{1}{\gamma}-1\right)\barepsilon_k.
\end{multline*}
\noindent$\bullet$ \textbf{Case 3: unsuccessful iteration:} In this case, \eqref{eq:Phi_unsuccessful} holds.

\noindent Combining \textbf{Cases 1, 2,} and \textbf{3}, and noting that \eqref{eq:nu_and_1-nu} implies \eqref{eq:lemma_ABCC'_epsilon} and
\begin{equation}\label{eq:eq:lemma_(ABCC')^c_1}
- \frac{\nu \kappa_{fcd}}{4\max\{1,\Delta_{\max}\}}\eta^3 \Delta_k^3 +\frac{1-\nu}{2}(\gamma^3-1)\Delta_k^3 \leq \frac{1-\nu}{2}\left(\frac{1}{\gamma^3}-1\right)\Delta_k^3,
\end{equation}
we complete the proof.
\end{proof}

Lemma~\ref{lemma:ABCC'_hold} demonstrates that if all estimates are accurate, then a decrease in $\Phi_{\barmu_{\barK}}^k$ is guaranteed, while Lemma~\ref{lemma:Phi_(ABCC')^c} reveals that if some estimates are inaccurate, then $\Phi_{\barmu_{\barK}}^k$ might increase. Next,~we show that as long as the probability of obtaining an accurate objective model exceeds a deterministic threshold, a reduction in $\Phi_{\barmu_{\barK}}^k$ is guaranteed in expectation.

\begin{lemma}\label{lemma:One_Step_Rec}

Under Assumptions \ref{ass:1-1} and \ref{ass:4-1}, for $k\geq \barK$, if
\begin{equation}\label{eq:probability_1}
p_h+p_g+2p_f \leq \frac{(1-\nu)^2}{16\nu^2}\left(1-\frac{1}{\gamma}\right)^2
\end{equation}
then 
\begin{equation}\label{eq:One_Step_Rec}
\mE[\Phi_{\barmu_{\barK}}^{k+1}\mid\F_{k-1}]-\Phi_{\barmu_{\barK}}^k \leq \frac{1-\nu}{2}\left(\frac{1}{\gamma^3}-1\right)\Delta_k^3. 
\end{equation}
\end{lemma}

\begin{proof}

By the definitions of $\A_k, \B_k, \C_k$ (and $\C_k'$) in \eqref{def:Ak}, \eqref{def:Bk}, \eqref{def:Ck}, we know that
\begin{align}\label{eq:prob_ABCC'}
P[\A_k\cap\B_k\cap\C_k\cap\C_k'\mid\F_{k-1}] & \geq 1 - \cbr{P(\A_k\mid\F_{k-1})+P(\B_k\mid\F_{k-1})+P(\C_k\mid\F_{k-1})+P(\C_k'\mid\F_{k-1})} \nonumber\\
& \geq 1 - (p_h+p_g+\mE[\mE[\1_{\C_k}+\1_{\C_k'}\mid\F_{k-0.5}] \mid \F_{k-1}] ) \nonumber\\
& \geq 1-(p_h+p_g+2p_f).
\end{align}
Then, we apply Lemmas \ref{lemma:ABCC'_hold}, \ref{lemma:Phi_(ABCC')^c}, and \eqref{eq:prob_ABCC'}, and have
\begin{align*}
& \mE[\Phi_{\barmu_{\barK}}^{k+1}\mid\F_{k-1}] -\Phi_{\barmu_{\barK}}^k  \\
& = \mE[(\Phi_{\barmu_{\barK}}^{k+1}-\Phi_{\barmu_{\barK}}^k )\boldsymbol{1}_{(\A_k\cap\B_k\cap\C_k\cap\C_k')}\mid\F_{k-1}] + \mE[(\Phi_{\barmu_{\barK}}^{k+1}-\Phi_{\barmu_{\barK}}^k)\boldsymbol{1}_{(\A_k\cap\B_k\cap\C_k\cap\C_k')^c} \mid\F_{k-1}]\\
& \leq \frac{1-\nu}{2}\left(\frac{1}{\gamma^3}-1\right)\Delta_k^3+\frac{1-\nu}{2}\left(\frac{1}{\gamma}-1\right)\barepsilon_k + \nu\cdot \mE[(|f_{s_k}-\barf_{s_k}|+|f_k -\barf_k|) \boldsymbol{1}_{(\A_k\cap\B_k\cap\C_k\cap\C_k')^c} \mid\F_{k-1}]\\
& \leq \frac{1-\nu}{2}\left(\frac{1}{\gamma^3}-1\right)\Delta_k^3+\frac{1-\nu}{2}\left(\frac{1}{\gamma}-1\right)\barepsilon_k + 2 \nu\cdot \sqrt{p_h+p_g+2p_f}\cdot \barepsilon_k,
 \end{align*}
where the last inequality uses the Hölder's inequality and the condition \eqref{def:reliable_est}. Since \eqref{eq:probability_1} implies
\begin{equation*}
\frac{1-\nu}{2}\left(\frac{1}{\gamma}-1\right)\barepsilon_k + 2 \nu\cdot \sqrt{p_h+p_g+2p_f}\cdot \barepsilon_k\leq 0,
\end{equation*}
we combine the above two displays and complete the proof.
\end{proof}

The following result follows immediately from Lemma \ref{lemma:One_Step_Rec}.

\begin{corollary}\label{coro:radius_conv_zero}
Under the conditions of Lemma \ref{lemma:One_Step_Rec}, $\lim_{k\rightarrow\infty}\Delta_k=0$ with probability $1$.
\end{corollary}

\begin{proof}
Taking the expectation conditional on $\F_{\barK-1}$ on both sides of \eqref{eq:One_Step_Rec}, we have
\begin{equation}\label{eq:phi_mono_decreasing}
\mE[\Phi^{k+1}_{\barmu_{\barK}}-\Phi_{\barmu_{\barK}}^k\mid\F_{\barK-1}] \leq \frac{1-\nu}{2}\left(\frac{1}{\gamma^3}-1\right)\mE[\Delta_k^3\mid\F_{\barK-1}].
\end{equation}
Summing over $k\geq\barK$, and noting that $\mE[\Phi^{k}_{\barmu_{\barK}}\mid\F_{\barK-1}]$ is monotonically decreasing and bounded below by $\nu \cdot f_{\inf}$ (cf. Assumption \ref{ass:1-1}), we have
\begin{equation*}
-\infty < \sum_{k=\barK}^\infty \mE[\Phi^{k+1}_{\barmu_{\barK}}-\Phi_{\barmu_{\barK}}^k\mid\F_{\barK-1}]  \leq \frac{1-\nu}{2}\left(\frac{1}{\gamma^3}-1\right) \sum_{k=\barK}^\infty \mE[\Delta_k^3\mid\F_{\barK-1}].
\end{equation*}
Since $\Delta_k\geq 0$, by Tonelli's Theorem, we have $\mE[\sum_{k=\barK}^\infty \Delta_k^3\mid\F_{\barK-1}] < \infty$, which implies $P[\sum_{k=\barK}^{\infty}\Delta_k^3<\infty\mid\F_{\barK-1}]=1$. Since the conclusion holds for an arbitrarily given $\F_{\barK-1}$, we have $P[\sum_{k=\barK}^{\infty}\Delta_k^3<\infty]=1$, which implies that $\lim_{k\rightarrow\infty}\Delta_k=0$ with probability 1.
\end{proof}

\subsection{Global almost sure convergence}

The following result shows that the limit inferiors of both the KKT residual $\|\nabla\L_k\|$ and the negative curvature of the reduced Lagrangian Hessian $\tau_k^+$ are zero almost surely.

\begin{theorem}[Global first- and second-order convergence]\label{thm:liminf_result}
Under the conditions of Lemma \ref{lemma:One_Step_Rec}, we have both $\liminf_{k\rightarrow\infty}\|\nabla\L_k\|=0$ and $\liminf_{k\rightarrow\infty}\tau_k^+=0$ almost surely.
\end{theorem}

\begin{proof}

We note that the stochastic process $\tilde{w}_k = \sum_{i=0}^{k-1}(\boldsymbol{1}_{(\A_i\cap\B_i\cap\C_i\cap\C_i')} - \mE[\boldsymbol{1}_{(\A_i\cap\B_i\cap\C_i\cap\C_i')} | \mF_{i-1}])$ is a martingale since
\begin{equation*}
\mE[\tilde{w}_{k+1}|\mF_{k-1}] = \tilde{w}_k + \mE[\boldsymbol{1}_{(\A_k\cap\B_k\cap\C_k\cap\C_k')} | \mF_{k-1}] - \mE[\boldsymbol{1}_{(\A_k\cap\B_k\cap\C_k\cap\C_k')} | \mF_{k-1}] = \tilde{w}_k.
\end{equation*}
Using the fact that $\boldsymbol{1}_{(\A_i\cap\B_i\cap\C_i\cap\C_i')}\leq 1$ and \cite[Theorem 2.19]{Hall2014Martingale}, we know $\tilde{w}_k/k\rightarrow 0$ almost surely. Let us define $w_k = \sum_{i=0}^{k-1}(2\cdot\boldsymbol{1}_{(\A_i\cap\B_i\cap\C_i\cap\C_i')}-1)$, then
\begin{equation*}
\frac{w_k}{k} = \frac{2\tilde{w}_k}{k} + \frac{1}{k}\sum_{i=0}^{k-1}(2\mE[\boldsymbol{1}_{(\A_i\cap\B_i\cap\C_i\cap\C_i')} | \mF_{i-1}]-1) \stackrel{\eqref{eq:prob_ABCC'}}{\geq}  \frac{2\tilde{w}_k}{k} + 1 - 2(p_h+p_g+2p_f).
\end{equation*}
Since $p_h+p_g+2p_f< 0.5$ (as implied by \eqref{eq:probability_1}), we know from the above display that $w_k\rightarrow \infty$ almost surely.~With this result, we now prove $\liminf_{k\rightarrow\infty}\|\nabla\L_k\|=0$ almost surely by contradiction.
Suppose there exist $\epsilon_1>0$ and $K_1 \geq \barK$ such that for all $k\geq K_1$, $\|\nabla\L_k\| \geq \epsilon_1$. Since $\Delta_k\rightarrow 0$~by~Corollary~\ref{coro:radius_conv_zero}, there exists $K_1'\geq K_1$ such that for all $k\geq K_1'$,
\begin{equation*}
\Delta_k \leq a\coloneqq \min\left\{\frac{\Delta_{\max}}{\gamma},\frac{\epsilon_1}{\varphi}\right\}, \; \text{ with } \varphi \coloneqq \max\left\{ \kappa_B,\frac{4\kappa_f\max\{1,\Delta_{\max}\}+8\Upsilon_1}{\kappa_{fcd}(1-\eta)}\right\}+\kappa_g\max\{1,\Delta_{\max}\}.
\end{equation*}
Therefore, for all $k\geq K_1'$, we have $\|\nabla\L_k\|\geq \varphi\Delta_k$, which combined with Lemma \ref{lemma:guarantee_succ_step_KKT} shows that if $\A_k\cap\B_k\cap\C_k\cap\C_k'$ holds, then the iteration must be successful. Since $\Delta_k\leq\Delta_{\max}/\gamma$, we have $\Delta_{k+1}=\gamma\Delta_k$. On the other hand, if $(\A_k\cap\B_k\cap\C_k\cap\C_k')^c$ holds, the iteration can be~successful~or~not.~In this case, we have $\Delta_{k+1} \geq \Delta_k/ \gamma$. Let $b_k=\log_\gamma\left(\frac{\Delta_k}{a}\right)$, which satisfies $b_k\leq 0$ for all $k\geq K_1'$. In addition, for $k\geq K_1'$, if $\A_k\cap\B_k\cap\C_k\cap\C_k'$ holds, then $b_{k+1}=b_k+1$; otherwise, $b_{k+1} \geq b_k-1$.~From~the~definitions of $\{w_k\}$ and $\{b_k\}$, we know $b_k-b_{K_1'} \geq w_k-w_{K_1'}$ for all $k\geq K_1'$. Thus, $b_k\rightarrow\infty$, which contradicts $b_k\leq 0$ for all $k\geq K_1'$. This contradiction~concludes~$\liminf_{k\rightarrow\infty}\|\nabla\L_k\|=0$. Now,~we prove~$\liminf_{k\rightarrow\infty}\tau_k^+=0$ almost surely in a similar way.
Suppose there exist $\epsilon_2>0$ and $K_2\geq \barK$ such that for all $k\geq K_2$, $\tau_k^+\geq\epsilon_2$. 
Since $\Delta_k\rightarrow 0$, there exists $K_2' \geq K_2$ such that for all $k \geq K_2'$,
\begin{equation*}
\Delta_k \leq a'\coloneqq \min\left\{\frac{\Delta_{\max}}{\gamma},\frac{\epsilon_2}{\varphi'} \right\},\; \text{ with } \varphi'\coloneqq\max\left\{ \eta, \frac{4\kappa_f\max\{1,\Delta_{\max}\}+2\Upsilon_1+2\Upsilon_2}{(1-\eta)\kappa_{fcd}\min\{1,r\}}\right\}+\kappa_H.
\end{equation*}
The rest of the proof combines Lemma \ref{lemma:guarantee_succ_step_eigen}, defines $b'_k\coloneqq\log_\gamma\left(\frac{\Delta_k}{a'}\right)$, and uses the relation $b'_k-b'_{K_2'} \geq w_k-w_{K_2'}$ for all $k\geq K_2'$ to arrive at a contradiction.
\end{proof}

We now demonstrate that the above first-order convergence guarantee can be strengthened~to~limit-type convergence, which asserts that the limit of $\|\nabla\L_k\|$ is zero almost surely. To establish~this~result, we require the following lemma.

\begin{lemma}\label{lemma: finite_sum}
For any $\epsilon>0$, we let $\K_\epsilon = \{k: \|\nabla\L_k\|\geq\epsilon\}$. Under the conditions of Lemma~\ref{lemma:One_Step_Rec},~we have $\sum_{k\in \mathcal{K}_{\epsilon}}\Delta_k<\infty$ with probability $1$.
\end{lemma}

\begin{proof}
Let $\varphi$ be as in the proof of Theorem \ref{thm:liminf_result}. We first consider the reduction in $\Phi_{\barmu_{\barK}}^k$~when~$\|\nabla\L_k\|\geq\varphi\Delta_k$. On the event $\A_k\cap\B_k\cap\C_k\cap\C_k'$, \eqref{eq:lemma9_case1_final} holds, while on the event $(\A_k\cap\B_k\cap\C_k\cap\C_k')^c$, \eqref{eq:lemma_phi_reduction_not_all_acc}~holds.~Thus, we have
\begin{align*}
& \mE[\Phi_{\barmu_{\barK}}^{k+1}\mid\F_{k-1}]- \Phi_{\barmu_{\barK}}^k \\
& = \mE[(\Phi_{\barmu_{\barK}}^{k+1}- \Phi_{\barmu_{\barK}}^k)\boldsymbol{1}_{(\A_k\cap\B_k\cap\C_k\cap\C_k')} \mid\F_{k-1}] +\mE[(\Phi_{\barmu_{\barK}}^{k+1}- \Phi_{\barmu_{\barK}}^k )\boldsymbol{1}_{(\A_k\cap\B_k\cap\C_k\cap\C_k')^c}\mid\F_{k-1}]\\
& \leq -\mE[\boldsymbol{1}_{(\A_k\cap\B_k\cap\C_k\cap\C_k')} \mid\F_{k-1}]\cdot \frac{\nu\Upsilon_3}{6}\|\nabla\L_k\|\Delta_k +\mE[\boldsymbol{1}_{(\A_k\cap\B_k\cap\C_k\cap\C_k')^c} \mid\F_{k-1}]\cdot\frac{1-\nu}{2}\left(\frac{1}{\gamma^3}-1\right)\Delta_k^3\\
&\quad + \frac{1-\nu}{2}\left(\frac{1}{\gamma} - 1\right)\barepsilon_k +\nu \cdot \mE[(|f_{s_k}-\barf_{s_k}| + |f_k-\barf_k|)\boldsymbol{1}_{(\A_k\cap\B_k\cap\C_k\cap\C_k')^c}\mid\F_{k-1}] \\
& = \mE[\boldsymbol{1}_{(\A_k\cap\B_k\cap\C_k\cap\C_k')} \mid\F_{k-1}]\cbr{-\frac{\nu\Upsilon_3}{6}\|\nabla\L_k\|\Delta_k -  \frac{1-\nu}{2}\left(\frac{1}{\gamma^3}-1\right)\Delta_k^3} +  \frac{1-\nu}{2}\left(\frac{1}{\gamma^3}-1\right)\Delta_k^3\\
&\quad + \frac{1-\nu}{2}\left(\frac{1}{\gamma} - 1\right)\barepsilon_k +\nu \cdot \mE[(|f_{s_k}-\barf_{s_k}| + |f_k-\barf_k|)\boldsymbol{1}_{(\A_k\cap\B_k\cap\C_k\cap\C_k')^c}\mid\F_{k-1}] \\
&\stackrel{\mathclap{\eqref{last_eq}, \eqref{eq:prob_ABCC'}}}{ \leq}\; -(1-p_h-p_g-2p_f)\frac{\nu\Upsilon_3}{6}\|\nabla\L_k\|\Delta_k   +(p_h+p_g+2p_f)\frac{1-\nu}{2}\left(\frac{1}{\gamma^3}-1\right)\Delta_k^3+ \frac{1-\nu}{2}\left(\frac{1}{\gamma} - 1\right)\barepsilon_k \\
&\quad +\nu \cdot \mE[(|f_{s_k}-\barf_{s_k}| + |f_k-\barf_k|)\boldsymbol{1}_{(\A_k\cap\B_k\cap\C_k\cap\C_k')^c}\mid\F_{k-1}].
\end{align*}
Using the H\"older's inequality and the condition \eqref{def:reliable_est}, we have
\begin{align}\label{eq:reduction_radius_less_than_KKT}
\mE[\Phi_{\barmu_{\barK}}^{k+1}\mid\F_{k-1}]- \Phi_{\barmu_{\barK}}^k & \leq -(1-p_h-p_g-2p_f)\frac{\nu\Upsilon_3}{6}\|\nabla\L_k\|\Delta_k  +(p_h+p_g+2p_f)\frac{1-\nu}{2}\left(\frac{1}{\gamma^3}-1\right)\Delta_k^3 \notag\\
&\quad + \frac{1-\nu}{2}\left(\frac{1}{\gamma} - 1\right)\barepsilon_k +2 \nu \cdot \sqrt{p_h+p_g+2p_f}\cdot\barepsilon_k \notag\\
&\stackrel{\eqref{eq:probability_1}}{ \leq} -(1-p_h-p_g-2p_f)\frac{\nu\Upsilon_3}{6}\|\nabla\L_k\|\Delta_k.
\end{align}
Since $\Delta_k\rightarrow 0$ almost surely, for each realization of Algorithm \ref{Alg:STORM}, there exists a finite $K_3\geq\barK$ such that for all $k\geq K_3$, we have $\Delta_k\leq \epsilon/\varphi$. 
Let $\tilde{\mathcal{K}}_{\epsilon} = \K_\epsilon\cap \{k: k\geq K_3\}$. For $k\in \tilde{\mathcal{K}}_{\epsilon}$, we have~$\|\nabla\L_k\|\geq\varphi\Delta_k$ so that the reduction \eqref{eq:reduction_radius_less_than_KKT} is achieved. Since $\|\nabla\L_k\|\geq\epsilon$ for all $k\in \tilde{\mathcal{K}}_{\epsilon}$,~we~further have\vskip-0.15cm
\begin{equation*}
\mE[\Phi_{\barmu_{\barK}}^{k+1}\mid\F_{k-1}] - \Phi_{\barmu_{\barK}}^k \leq  -\frac{1}{6}(1-p_h-p_g-2p_f)\nu\Upsilon_3\epsilon\cdot\Delta_k.
\end{equation*}
Taking the conditional expectation with respect to $\F_{\barK-1}$ on both sides, and recalling that $\mE[\Phi_{\barmu_{\barK}}^k\mid\F_{\barK-1}]$ is monotone decreasing in $k$ and bounded below (cf. Assumption \ref{ass:1-1}), we have $\sum_{k\in\tilde{\mathcal{K}}_{\epsilon}}\mE[\Delta_k\mid\F_{\barK-1}]<\infty$.
By Tonelli's theorem, we have $\mE[\sum_{k\in\tilde{\mathcal{K}}_{\epsilon}}\Delta_k\mid\F_{\barK-1}]<\infty$ and thus $P[\sum_{k\in\tilde{\mathcal{K}}_{\epsilon}}\Delta_k<\infty \mid\F_{\barK-1}]=1$. Since the conclusion holds for any $\mF_{\barK-1}$, we have $P[\sum_{k\in\tilde{\mathcal{K}}_{\epsilon}}\Delta_k<\infty]=1$.
Since $\mathcal{K}_{\epsilon}\subseteq \tilde{\mathcal{K}}_{\epsilon} \cup \{k\leq K_3\}$, $\Delta_k\leq\Delta_{\max}$, and $K_3$ is finite, we complete the proof.
\end{proof}

Finally, we state the limit-type first-order global convergence guarantee for Algorithm \ref{Alg:STORM}.

\begin{theorem}[Stronger first-order convergence]\label{thm:limit_result}
Under the conditions of Lemma \ref{lemma:One_Step_Rec}, we have $\lim_{k\rightarrow \infty} \|\nabla\L_k\|=0$ almost surely.
\end{theorem}

\begin{proof}

We prove by contradiction. Suppose for a realization of Algorithm \ref{Alg:STORM}, there exist an $\epsilon>0$ and an infinite index set $\K_1$ such that $\|\nabla \L_k\|\geq2\epsilon$ for all $k\in\K_1$. By Theorem \ref{thm:liminf_result}, we~know~for~the~realization considered, there exists an infinite index set $\K_2$ such that~\mbox{$\|\nabla \L_k\|< \epsilon$}~for~all~\mbox{$k\in \K_2$}.~Thus,~there are index sets $\{m_i\}$ and $\{n_i\}$ with $m_i<n_i$ such that for all $i\geq 0$,
\begin{equation*}
\|\nabla \L_{m_i}\|\geq 2\epsilon,\quad \|\nabla \L_{n_i}\|<\epsilon,\quad 
\text{and }\;\; \|\nabla \L_k\|\geq\epsilon \;\;\text{ for }\; k\in\{m_i+1,\cdots,n_i-1\}.
\end{equation*}
By the algorithm design, for all $j\geq 0$, we have\vskip-0.5cm
\begin{equation}\label{eq:stepsize}
\|\bx_{j+1}-\bx_{j}\|\leq \|\Delta\bx_j\|+\|\bd_j\| \stackrel{\eqref{eq:length_of_correctional_step}}{\leq} \|\Delta\bx_j\|+\frac{L_G}{\sqrt{\kappa_{1,G}}}\|\Delta\bx_j\|^2 
\leq \left(1+\frac{L_G\Delta_{\max}}{\sqrt{\kappa_{1,G}}}\right)\Delta_{\max},
\end{equation}
where the last inequality is due to $\Delta_j\leq\Delta_{\max}$. Thus, by Assumption~\ref{ass:1-1} and $\nabla\L_k=(P_kg_k,c_k)$,~there exist constants $L_{\nabla \L,1}, L_{\nabla \L,2}>0$ such that $\|\nabla \L_{j+1}-\nabla\L_{j}\| \leq L_{\nabla\L,1}(\|\bx_{j+1}-\bx_{j}\|+\|\bx_{j+1}-\bx_{j}\|^2) \leq L_{\nabla\L,2}\|\bx_{j+1}-\bx_{j}\|$ for $ j\geq 0$. Then, \vskip-0.5cm
\begin{align*}
\epsilon & < \left|\|\nabla\L_{n_i}\|-\|\nabla\L_{m_i}\|\right|
 \leq \sum_{j=m_i}^{n_i-1}\|\nabla\L_{j+1}-\nabla\L_{j}\|
\leq L_{\nabla\L,2}\sum_{j=m_i}^{n_i-1}\|\bx_{j+1}-\bx_j\| \\
& \stackrel{\mathclap{\eqref{eq:stepsize}}}{\leq}L_{\nabla\L,2}\left(1+\frac{L_G\Delta_{\max}}{\sqrt{\kappa_{1,G}}}\right)
\sum_{j=m_i}^{n_i-1}\Delta_j = L_{\nabla\L,2}\left(1+\frac{L_G\Delta_{\max}}{\sqrt{\kappa_{1,G}}}\right)\left(\Delta_{m_i}+\sum_{j=m_i+1}^{n_i-1}\Delta_j\right).
\end{align*}
Since $\Delta_k$ converges to zero, we have $L_{\nabla\L,2}\left(1+L_G\Delta_{\max}/\sqrt{\kappa_{1,G}}\right)\Delta_{m_i} \leq \epsilon/2$ for $i$ large enough.~Thus, we obtain $L_{\nabla\L,2}\left(1+L_G\Delta_{\max}/\sqrt{\kappa_{1,G}}\right) \sum_{j=m_i+1}^{n_i-1}\Delta_j>\epsilon/2>0$. Since $\sum_i\sum_{j=m_i+1}^{n_i-1}\Delta_j\leq \sum_{j\in\K_{\epsilon}}\Delta_j$, we have $\sum_{j\in \K_{\epsilon}}\Delta_j=\infty$, which contradicts Lemma \ref{lemma: finite_sum}. This completes the proof.
\end{proof}

We have finished the convergence analysis of Algorithm \ref{Alg:STORM}. In particular, Theorem \ref{thm:limit_result} strengthens the liminf-type to limit-type for the first-order convergence guarantee, and shows that the~\mbox{iterates} generated by Algorithm \ref{Alg:STORM} have vanishing KKT residuals almost surely.~This result matches the~\mbox{first-order} conclusion in \cite{Chen2017Stochastic} for trust-region methods in unconstrained problems.

We also mention that strengthening the liminf-type to limit-type for the second-order convergence guarantee is challenging. Technically, for an eigen step, the predicted reduction of the merit function $\text{Pred}_k$ in \eqref{def:Pred_k} involves the term $\bartau_k^+\Delta_k^2$.~Using a proof similar to Lemma~\ref{lemma: finite_sum}~would~lead~to~$\sum_{k\in \K_{\epsilon}'} \Delta_k^2 < \infty$, where $\K_{\epsilon}' = \{k:\bartau_k^+ \geq \epsilon\}$. However, this fact would not lead to any contradiction with~$\sum_{k\in \K_{\epsilon}'} \Delta_k = \infty$. That being said, Theorem \ref{thm:liminf_result} suggests that there exists a subsequence of iterates with vanishing negative curvature of the reduced Lagrangian Hessian. This result also matches the state-of-the-art second-order conclusion in \cite{Blanchet2019Convergence} for trust-region methods in unconstrained problems.

\subsection{Merit parameter behavior}
\label{subsec:merit_para}

In this subsection, we investigate the behavior of merit parameter $\barmu_k$ and demonstrate the reasonability of Assumption \ref{ass:4-1}. 
We prove that Assumption \ref{ass:4-1} holds when the gradient estimate $\barg_k$ is bounded above and the (Lagrangian) Hessian estimate $\barH_k$ is bounded both above and below.~In~particular, we introduce the following assumption.

\begin{assumption}\label{ass:bdd_err}
For all $k \geq 0$, (i) there exists $M>0$ such that $\|\barg_k-g_k\|\leq M$; (ii)~there~exists $\kappa_B>0$ such that $1/\kappa_B\leq \|\barH_k\|\leq \kappa_B$.
\end{assumption}

The above assumption is consistent with \cite[Assumption 4.12]{Fang2024Fully}.
The upper~boundedness condition of $\barg_k$ is commonly imposed in the SSQP literature (see \cite{Berahas2021Sequential, Berahas2023Stochastic,Na2022adaptive,Na2023Inequality,Curtis2024Stochastic}), and is satisfied, for example, when the objective has a finite-sum form (i.e., sampling from the empirical distribution as in many machine learning problems). The upper boundedness condition of $\barH_k$ is restated from Assumption \ref{ass:1-1} (and~Lemma \ref{lemma:bounded_Hessian}), which is equivalent to assuming the upper boundedness of the objective~Hessian~noise $\|\bar{\nabla}^2f_k-\nabla^2f_k\|$~(cf.~\eqref{eq:lemma1_1}).

In addition, the lower boundedness of $\barH_k$ is a mild regularity condition. For first-order stationarity, we do not require $\barH_k$ to be a precise estimate of the Lagrangian Hessian $\nabla_{\bx}^2\mL_k$. We can set $\barH_k$ as the identity matrix, the estimated Hessian, the averaged Hessian, or the quasi-Newton update; all these reasonable constructions are naturally bounded away from zero.
For second-order stationarity, we~let $\barH_k=\bar{\nabla}^2 f_k+\sum_{i=1}^n\barblambda_k\nabla^2 c^i_k$ be an estimate of the Lagrangian Hessian (see Step 1 in Section~\ref{sec:3.2}).~As suggested by second-order sufficient condition \cite[Chapter 12]{Nocedal2006Numerical}, it is also~very reasonable to have a non-vanishing Hessian estimate (especially for large $k$) in order to exploit the curvature information and converge to a non-trivial second-order stationary points.

We note that, in contrast to \cite{Sun2023trust},~the upper and lower bounds for the~quantities in our study are unknown and not involved in our trust-region algorithm design.

\begin{lemma}\label{lemma: mu_stabilize}
Assumptions \ref{ass:1-1} and \ref{ass:bdd_err} imply Assumption \ref{ass:4-1}.~In particular, under \mbox{Assumptions}~\ref{ass:1-1} and \ref{ass:bdd_err}, there exist an (potentially stochastic) iteration threshold $\barK<\infty$ and a deterministic~constant $\hat{\mu}$, such that~$\barmu_k=\barmu_{\barK}\leq \hatmu$ for all  $k \geq \barK$.
\end{lemma}

\begin{proof}
See Appendix \ref{append:1}.
\end{proof}

Existing line-search-based SSQP methods require the stochastic merit parameter $\barmu_k$ not only to be stabilized but also to be stabilized at a sufficiently large value to demonstrate global (first-order)~convergence \citep{Berahas2021Sequential,Berahas2023Stochastic,Berahas2023Accelerating,Na2022adaptive,Na2023Inequality,Curtis2024Stochastic}. Without a sufficiently large merit parameter, these methods cannot establish a connection~\mbox{between}~the~\mbox{stochastic}~reduction of the merit function and the true KKT residual. To achieve this requirement, \cite{Berahas2021Sequential, Berahas2023Stochastic, Berahas2023Accelerating, Curtis2024Stochastic} additionally assumed a symmetric estimation noise,~while \cite{Na2022adaptive, Na2023Inequality} imposed a stronger feasibility condition when selecting the merit parameter.~Our method eliminates this requirement. By computing a gradient step (or an eigen step), our stochastic reduction of the merit function \eqref{eq:threshold_Predk} is proportional to the estimated KKT residual (or the negative curvature).~Then, by leveraging the design of random models, we can naturally bridge the gap between the stochastic merit function~\mbox{reduction}~and~the~true~KKT~\mbox{residual}~(or~the~true~\mbox{negative}~\mbox{curvature}),~as proved in Lemmas \ref{lemma:ABCC'_hold} and \ref{lemma:Phi_(ABCC')^c} and applied in Lemma \ref{lemma:One_Step_Rec}.

\section{Numerical Experiment}\label{sec:5}

We explore the empirical performance of TR-SQP-STORM (Algorithm \ref{Alg:STORM}). We implement the method both on a subset of equality-constrained problems from the benchmark CUTEst test set \citep{Gould2014CUTEst} and on constrained logistic regression problems using synthetic datasets~and~real~datasets from the UCI repository. In addition, we implement a saddle-point problem to examine the~capability~of~our methods to escape saddle points.
For all problems, our method is implemented for both first- and 
second-order stationarity, referred to as TR-SQP-STORM and TR-SQP-STORM2, respectively.~We compare the~performance of our method with an adaptive line-search-based~SSQP~\mbox{algorithm}~(Algorithm 3 in \cite{Na2022adaptive}, referred to as AL-SSQP below),
which is developed under a~\mbox{similar}~random model setup but offers only first-order guarantees. To investigate the role of the merit function, we also replace the augmented Lagrangian merit function in that algorithm with the same $\ell_2$ merit function we used, referring to this modified algorithm as $\ell_2$-SSQP method.

\subsection{Algorithm setups}

To estimate objective values, gradients, and Hessians, batches of samples are generated in each iteration with batch sizes selected adaptively. We denote the batch sizes as $|\xi_f^k|, |\xi_g^k|, |\xi_h^k|$ for~the~corresponding estimators.~In addition, for TR-SQP-STORM2, we use $|\xi_f^{k'}|$ to denote the batch size~of~the~set $\xi_f^{k'}$, which is only generated for re-estimating the objective value at the new trial point when the SOC step is performed (cf. Case 2 of Step 4 in Section \ref{sec:3.2}). We allow $(\xi_f^k, \xi_f^{k'}, \xi_g^k, \xi_h^k)$ to be dependent and their sizes are decided following \eqref{eq:batchsize}.~For AL-SSQP method, $|\xi_f^k|$ and $|\xi_g^k|$ are~\mbox{generated}~\mbox{following} the conclusions of Lemmas 2 and 3 in \cite{Na2022adaptive}.~Analogously, we generate $|\xi_f^k|$ and $|\xi_g^k|$~for~$\ell_2$-SSQP method as
\begin{equation*}
|\xi_f^k|\geq\frac{C_{func}\log\left(\frac{4}{p_f}\right)}{\min\left\{\left[\kappa_f\bar{\alpha}_k^2\left(\barg_k^T\Delta\bx_k-\barmu_k\|c_k\|\right)\right]^2,\bar{\epsilon}_k^2,1\right\}},\quad\quad |\xi_g^k|\geq\frac{C_{grad}\log\left(\frac{8d}{p_{grad}}\right)}{\min\left\{\kappa_{grad}^2\bar{\alpha}_k^2\|\bar{\nabla}\L_k\|^2,1\right\}},
\end{equation*}
where $C_{func}, C_{grad}$ are positive constants. We require all sample sizes to not exceed $10^4$.

For both AL-SSQP and $\ell_2$-SSQP methods, we follow the notation in \cite{Na2022adaptive} and~set~$\barmu_0=\bar{\epsilon}_0=1$, $\beta=0.3$, $\rho=1.2$, $\bar{\alpha}_0=\alpha_{\max}=1.5$, $\kappa_{grad}=0.05$, $\kappa_f=0.05$, $p_{grad}=p_f=0.1$,~and~$C_{grad}=C_{func}=5$. We set the Hessian matrix $\barH_k=I$ and solve all SQP subproblems exactly.

For our method (under both first- and second-order stationarity), we set $\Delta_0 = \barmu_0= \barepsilon_0 = 1$,~$\kappa_g = \kappa_h=0.05$, $p_f = p_g=p_h=0.9$, $C_f=C_g=C_h=5$, $\Delta_{\max}=5$, $\rho=1.2$, $\gamma=1.5$,~$\eta=0.4$,~and~$r=0.01$.
We apply \texttt{IPOPT} solver \citep{Waechter2005Implementation} to solve \eqref{eq:Sto_tangential_step} with $\kappa_{fcd}=1$. Since trust-region methods allow Hessian matrices to be indefinite, same as \cite{Fang2024Fully}, we consider four~different constructions of $\barH_k$ for first-order stationarity:

\begin{enumerate}[label=(\alph*),topsep=0pt]
\setlength\itemsep{0.0em}
\item Identity matrix (Id). This choice has been used in numerous existing SSQP literature due~to~the simplicity (see \cite{Berahas2021Sequential, Berahas2023Stochastic, Na2022adaptive,Na2023Inequality} and references therein).\quad\quad

\item Symmetric rank-one (SR1) update.  We initialize $\barH_{0}=I$ and, for $k\geq 1$, $\barH_k$ is updated as
\begin{equation*}
\barH_{k}=\barH_{k-1}+\frac{(\boldsymbol{y}_{k-1}-\barH_{k-1}\Delta\bx_{k-1})(\boldsymbol{y}_{k-1}-\barH_{k-1}\Delta\bx_{k-1})^T}{(\boldsymbol{y}_{k-1}-\barH_{k-1}\Delta\bx_{k-1})^T\Delta\bx_{k-1}}.
\end{equation*}
Here, $\boldsymbol{y}_{k-1}=\bar{\nabla}_{\bx}\L_{k}-\bar{\nabla}_{\bx}\L_{k-1}$ and $\Delta\bx_{k-1}=\bx_{k}-\bx_{k-1}$.~The quasi-Newton with SR1 update~can generate indefinite Hessian approximations and may converge faster to the true Hessian than BFGS in some scenarios \citep{Khalfan1993Theoretical}.
    
\item Estimated Hessian (EstH). As in the second-order stationarity, we estimate the Hessian matrix $\bar{\nabla}_{\bx}^2\L_k$ using a single sample and set $\barH_k=\bar{\nabla}_{\bx}^2\L_k$.

\item Averaged Hessian (AveH). We estimate the Hessian matrix $\bar{\nabla}_{\bx}^2\L_k$ using a single sample and set $\barH_k=\frac{1}{B}\sum_{i=k-B+1}^{k}\bar{\nabla}_{\bx}^2\L_i$ with $B=50$. This choice is motivated by \cite{Na2022Hessian}, which shows that averaging the Hessians helps stochastic Newton methods achieve faster~convergence. 

\end{enumerate}

\subsection{CUTEst set}

We implement 47 problems from the CUTEst test set.~All problems have a non-constant \mbox{objective},~only equality constraints, and dimension $d \leq 1000$. We employ two types of initializations: (i) the~initialization provided by the CUTEst package, and (ii) random initialization, where each entry~of~$\bx_0$~is~independently drawn from a Gaussian distribution $\N(0,100)$. For random initialization, all methods~start from the same initialization to ensure a fair comparison.

For objective values, gradients, and Hessians, we generate the estimates based on the true~deterministic quantities provided by the CUTEst package. Specifically, $F(\bx_k;\xi) \sim \N(f_k, \sigma^2)$,~$\nabla F(\bx_k;\xi) \sim \N(\nabla f_k, \sigma^2(I+\boldsymbol{1}\boldsymbol{1}^T))$, and $[\nabla^2 F(\bx_k;\xi)]_{i,j} = [\nabla^2 F(\bx_k;\xi)]_{j,i}\sim \N([\nabla^2 f_k]_{i,j},\sigma^2)$. Here, $\boldsymbol{1}$ denotes the $d$-dimensional all-one vector.
We consider four different noise levels $\sigma^2 \in \{10^{-8}, 10^{-4}, 10^{-2}, 10^{-1}\}$.~For each method on each problem and each noise~level, we perform 5 independent runs and report~the average of the KKT residuals.
The stopping criteria for TR-SQP-STORM, AL-SSQP, and $\ell_2$-SSQP are set as $\|\nabla\mathcal{L}_k\| \leq 10^{-4} \text{ OR } k \geq 10^5$, while the stopping criterion for TR-SQP-STORM2 is set as $\max\{\|\nabla\mathcal{L}_k\|, \tau_k^+\} \leq 10^{-4} \text{ OR } k \geq 10^5$.

\begin{figure}[t]
\centering
\includegraphics[width=0.48\textwidth]{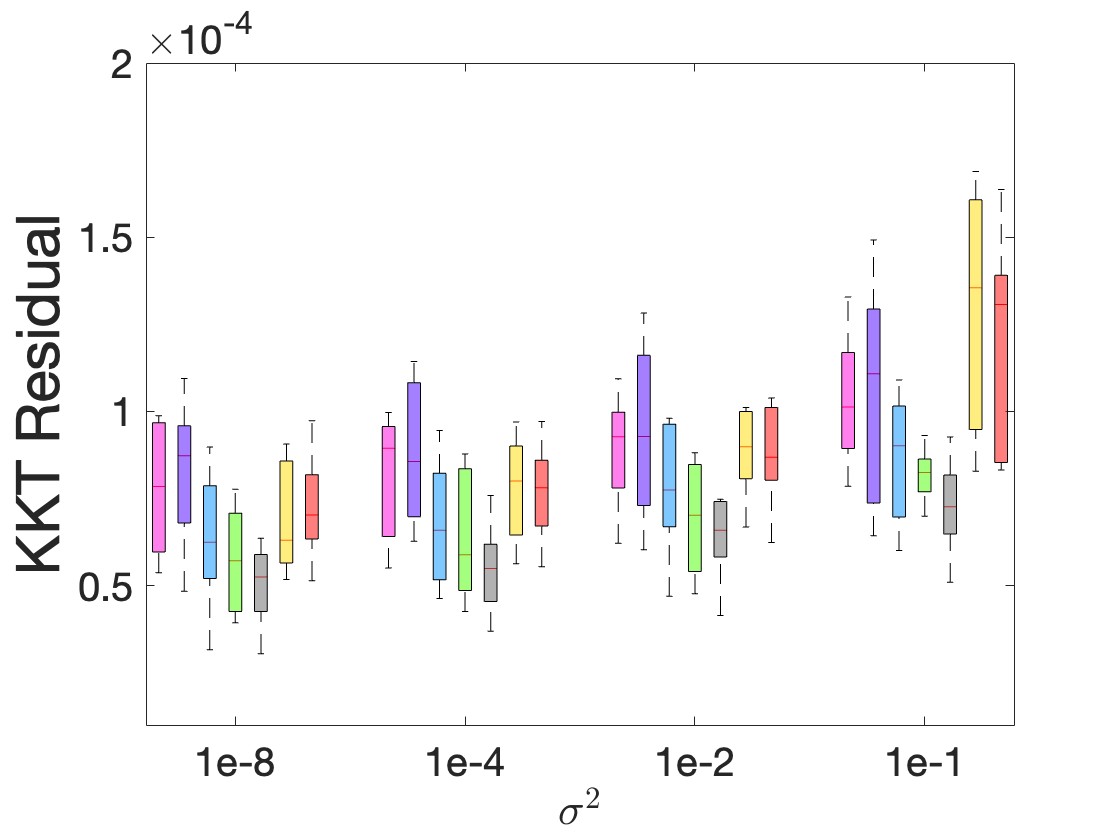}
\quad
\includegraphics[width=0.48\textwidth]{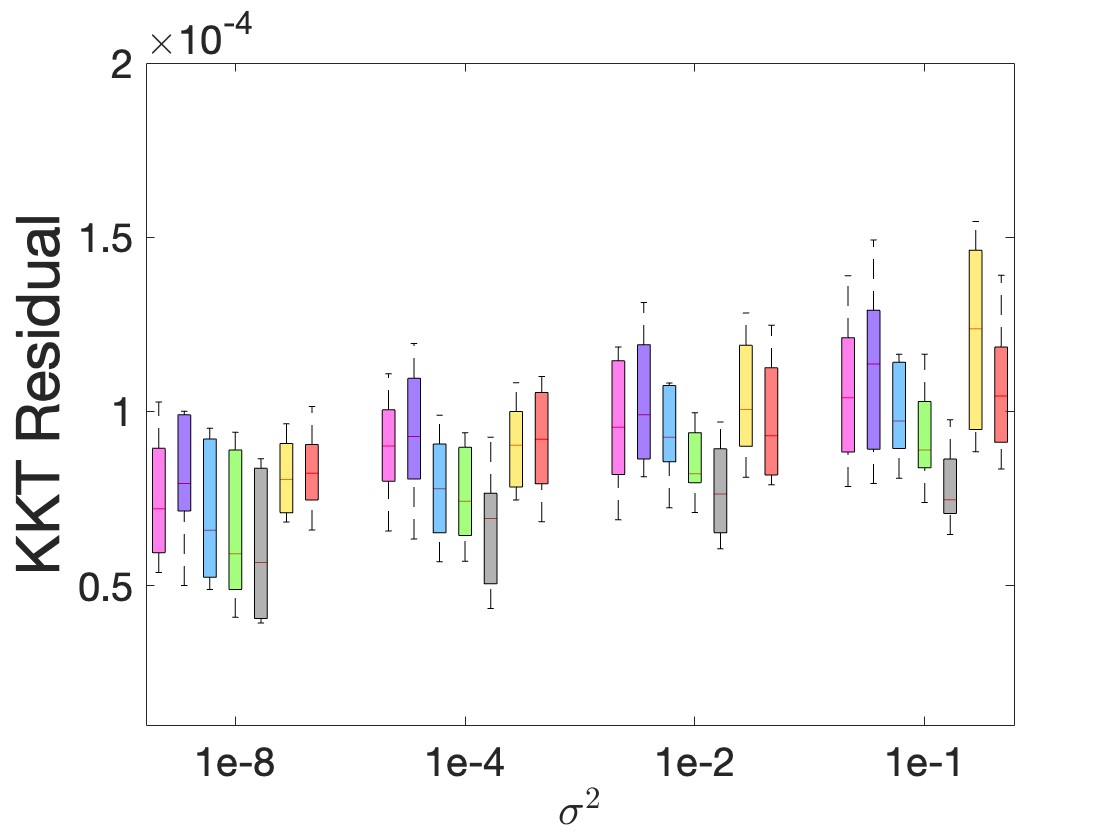}
\includegraphics[width=0.85\textwidth]{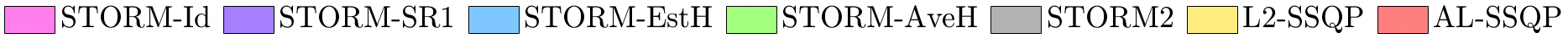}
\caption{KKT residual box plots over 47 CUTEst problems with given initialization (left) and~random initialization (right). Each panel has four different noise levels. For each noise level, the first~four boxes correspond to TR-SQP-STORM with different types of $\barH_k$; the fifth box corresponds~to~TR-SQP-STORM2; and the last two boxes correspond to $\ell_2$-SSQP and AL-SSQP, respectively.}
\label{fig:cutest}
\end{figure}

The results of the experiment are illustrated in Figure \ref{fig:cutest}.~From the figure, we observe that TR-SQP-STORM2 (i.e., our method with second-order stationarity) outperforms the other \mbox{methods} and its superior performance is robust across different noise levels and types of initializations.~This advantage is attributed to precise Hessian estimations, the ability to move along negative \mbox{curvatures},~and the computation of SOC steps. Only this method can guarantee to escape from saddle points. Furthermore, at low noise levels ($\sigma^2=10^{-8}$ or $10^{-4}$), line-search-based AL-SSQP and $\ell_2$-SSQP methods perform comparably to our trust-region methods. However, as noise levels increase, the performance of $\ell_2$-SSQP deteriorates rapidly, while AL-SSQP remains competitive though still inferior to our methods. 
In addition, among the four types of Hessians $\barH_k$ used in TR-SQP-STORM,~we~\mbox{observe}~that the SR1 update can lead to unstable performance.~It achieves~small~KKT~residuals at low~noise~\mbox{levels}~but~\mbox{performs}~well in some problems and poorly in others at high noise levels.~In contrast, the Hessians of EstH and~AveH consistently enhance the performance of TR-SQP-STORM across varying noise levels. Notably,~AveH demonstrates even better performance than EstH at high noise levels, as Hessian averaging generates more accurate Hessian estimates by aggregating samples.

\subsection{Logistic regression}
\begin{figure}[t!]
\centering
\subfigure[\texttt{covtype} dataset]{\includegraphics[width=0.48\textwidth]{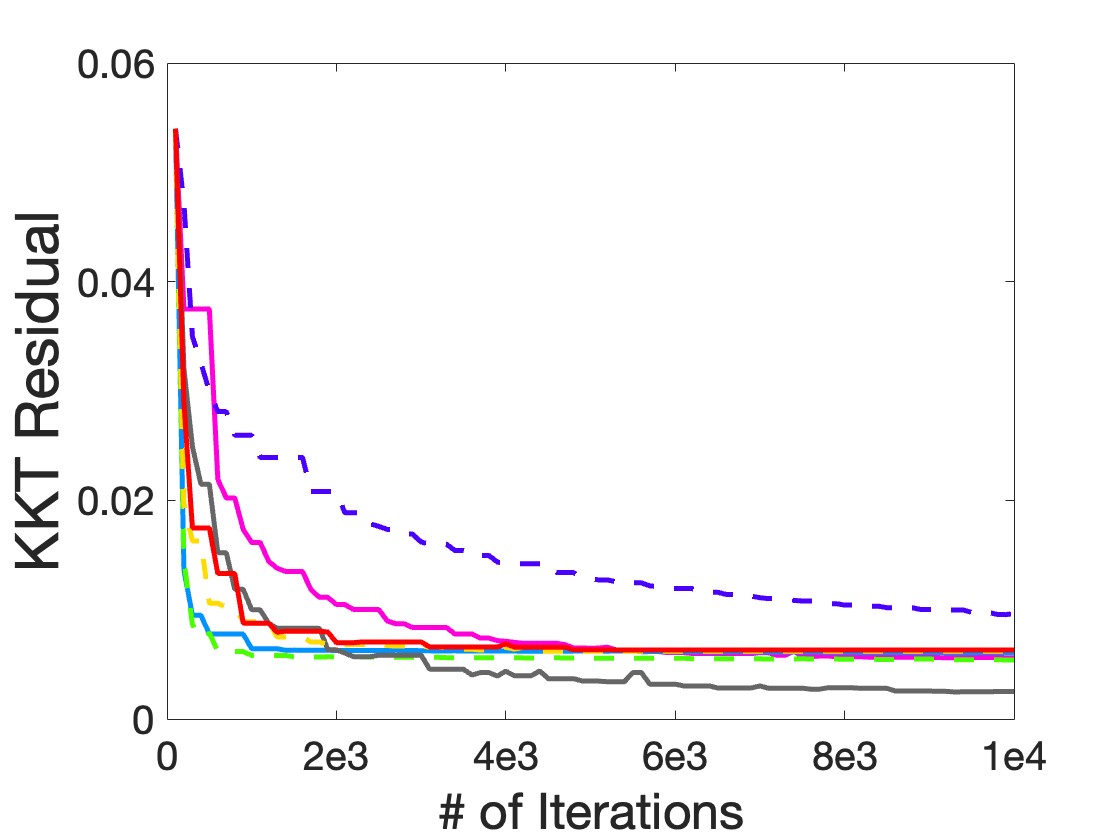}}
\quad
\subfigure[\texttt{shuttle} dataset]{\includegraphics[width=0.48\textwidth]{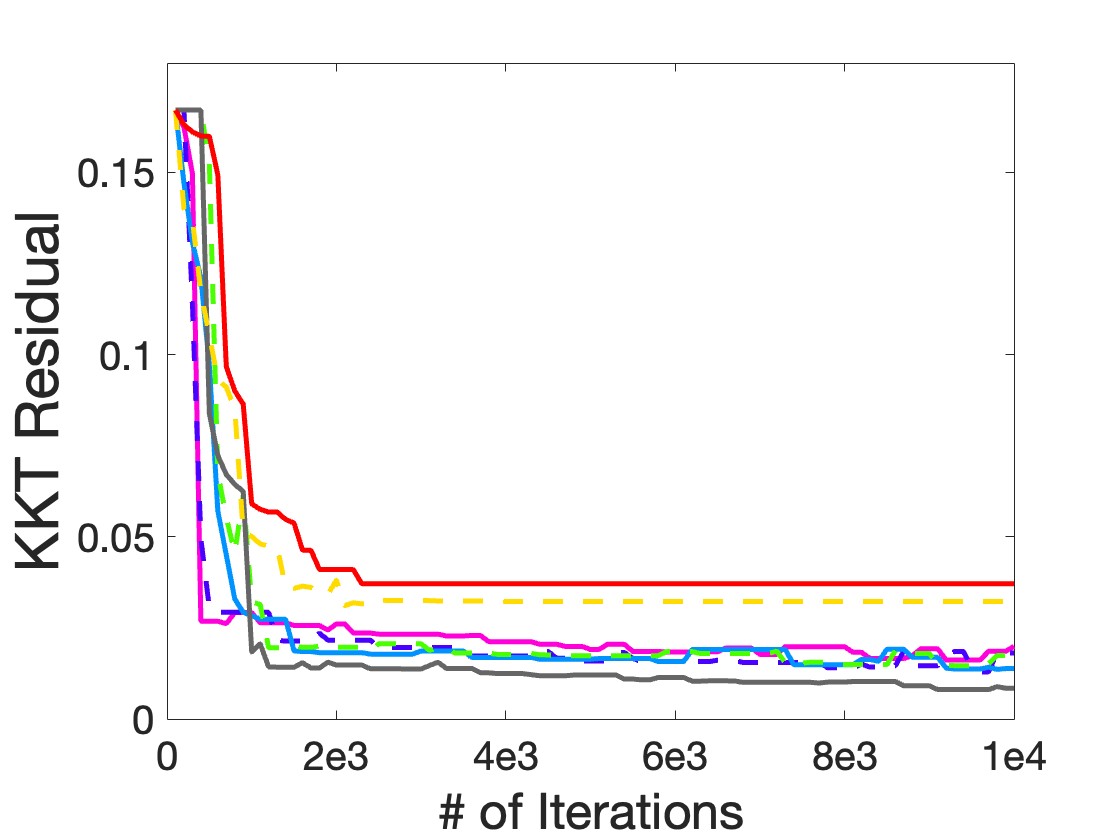}}
\quad
\subfigure[\texttt{normal} dataset]{\includegraphics[width=0.48\textwidth]{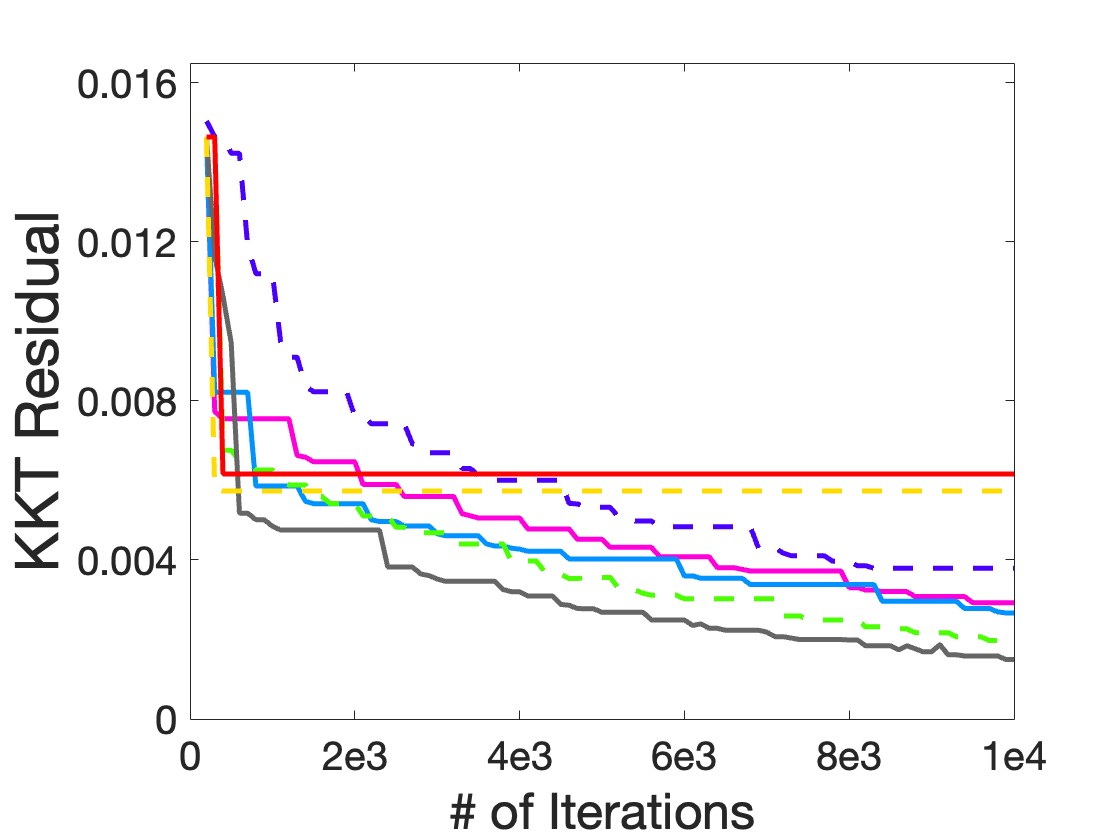}}
\quad
\subfigure[\texttt{exponential} dataset]{\includegraphics[width=0.48\textwidth]{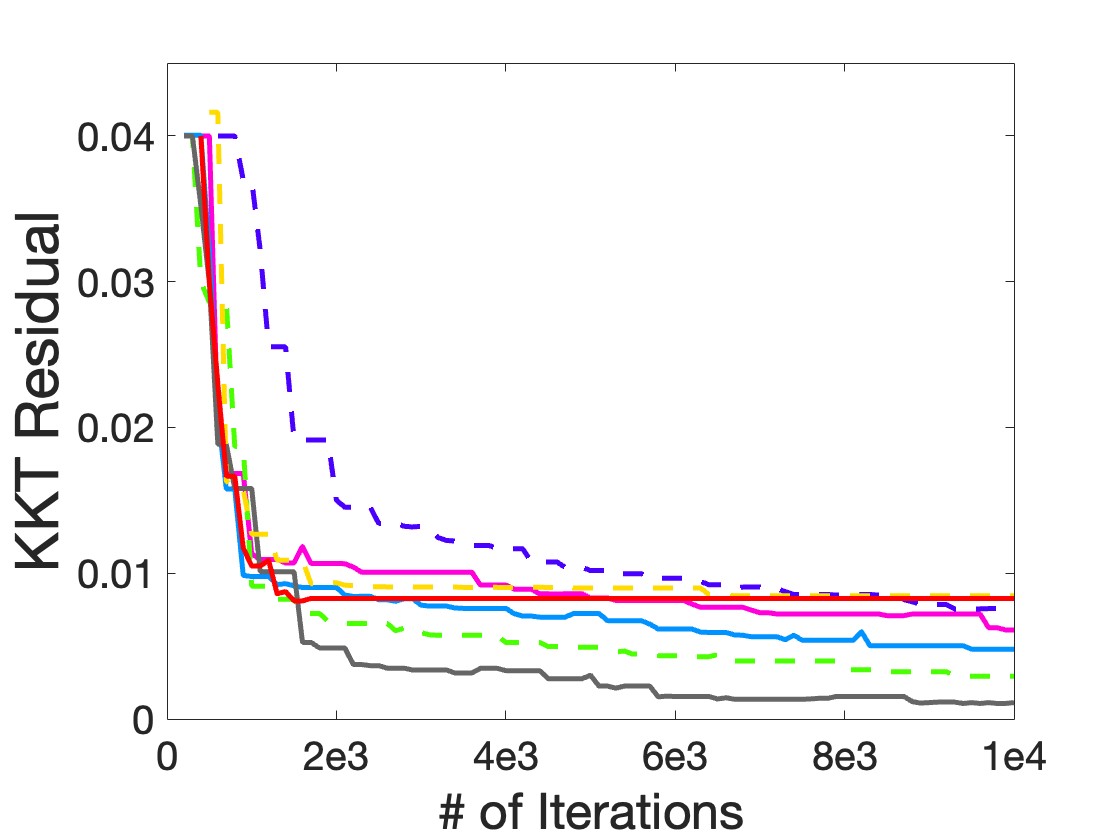}}
\quad
\includegraphics[width=0.85\textwidth]{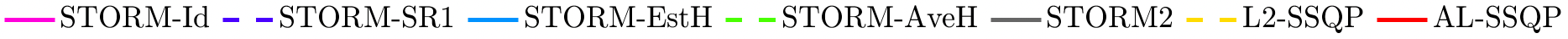}
\caption{Trajectories of KKT residuals of four datasets. Each panel corresponds to a dataset~and~includes seven lines representing the seven algorithms.}
\label{fig:logit}
\end{figure}

We consider an equality-constrained logistic regression problem of the form:
\begin{equation*}
\min_{\bx\in\mR^{d}}\;\; f(\bx)=\frac{1}{N}\sum_{i=1}^{N}\log\big(1+e^{-y_i(\boldsymbol{z}_i^T\bx)}\big),\quad\quad \text{s.t.} \quad A\bx=\boldsymbol{b},
\end{equation*}
where $\{(\boldsymbol{z}_i,y_i)\}_{i=1}^N$ are $N$ samples with features $\boldsymbol{z}_i\in\mR^d$ and labels $y_i\in\{-1,+1\}$. 
The constraint parameters $A\in\mR^{5\times d}$ and $\boldsymbol{b}\in\mR^5$ are generated for each problem with entries drawn independently~from a standard Gaussian distribution while ensuring that $A$ has full row rank.~We implement four datasets: \texttt{covtype} and \texttt{shuttle} from the UCI repository, and \texttt{normal} and \texttt{exponential} that are synthetic. 
For the \texttt{normal} and \texttt{exponential} datasets, we set $d=15$ and $N=6\times 10^4$,~equally~split~\mbox{between}~the~two classes. In the \texttt{normal} dataset, each entry of $\boldsymbol{z}_i$ is generated from $\N(0,1)$ if $y_i=1$ and $\N(5,1)$~if~$y_i=-1$. In the \texttt{exponential} dataset, each entry of $\boldsymbol{z}_i$ is generated from $\exp(1)$ if $y_i=1$ and $5+\exp(1)$ if $y_i=-1$. We set the initialization to a zero vector. 
For each algorithm and each dataset, we plot the trajectory of the average KKT residuals over five independent runs. The stopping criteria~for~TR-SQP-STORM, AL-SSQP, and $\ell_2$-SSQP are set as $ \|\nabla\mathcal{L}_k\|\leq 10^{-4}\text{ OR }k\geq 10^4$, while similarly,~the~stopping criterion for TR-SQP-STORM2 is set as $\max\{\|\nabla\mathcal{L}_k\|,\tau_k^+\}\leq 10^{-4}\text{ OR }k\geq 10^4$.

We present the results in Figure \ref{fig:logit}. From the figure, we observe that TR-SQP-STORM2~clearly~outperforms the other methods in three out of four datasets. Only in the \texttt{shuttle} dataset is its performance comparable to TR-SQP-STORM.
We also note that AL-SSQP and $\ell_2$-SSQP perform well on \texttt{covtype} but poorly on \texttt{shuttle} and \texttt{normal}. For the \texttt{exponential} dataset, the performance of AL-SSQP and $\ell_2$-SSQP is similar to that of TR-SQP-STORM with the Id and~SR1~\mbox{Hessian}~\mbox{updates},~both of which are inferior to the performance with EstH and AveH updates. 
Overall, among the four~types of Hessian matrices tested for TR-SQP-STORM, the averaged Hessian generally performs the best, followed by the estimated Hessian, while the SR1 update performs the worst. However, it~is~worth~noting that for the \texttt{shuttle} dataset, all four types of Hessians exhibit similar performance.

\subsection{Saddle-point problem}

To demonstrate the efficacy of TR-SQP-STORM2 in escaping saddle points compared to TR-SQP-STORM, AL-SSQP and $\ell_2$-SSQP methods, we consider the following saddle-point problem:
\begin{equation}\label{prob:saddle}
\min_{(x_1,x_2)} f(x_1,x_2) = 2x_1 + \frac{1}{2}x_2^2 \quad\quad \text{s.t.}\quad\quad  x_1^2 + x_2^2-1=0.
\end{equation}
We can check that Problem \eqref{prob:saddle} has two stationary points: a local minima at $(-1,0)$ and~a~\mbox{saddle}~point at $(1,0)$.
In this experiment, we initialize all methods randomly within a neighborhood of radius $0.01$ around the saddle point. Following the CUTEst experiment, we generate estimates of objective values, gradients, and Hessians based on their true deterministic quantities.~\mbox{Specifically}, we have~$F(\bx_k;\xi) \sim \N(f_k, \sigma^2)$, $\nabla F(\bx_k;\xi) \sim \N(\nabla f_k, \sigma^2(I+\boldsymbol{1}\boldsymbol{1}^T))$, and $[\nabla^2 F(\bx_k;\xi)]_{i,j} = [\nabla^2 F(\bx_k;\xi)]_{j,i}\sim \N([\nabla^2 f_k]_{i,j},\sigma^2)$. 
We consider four different noise levels $\sigma^2 \in \{10^{-8}, 10^{-4}, 10^{-2}, 10^{-1}\}$. For each method on each  noise level, we perform 5 independent runs and report the averaged trajectories of the KKT residuals and the smallest eigenvalue of the reduced Lagrangian Hessians. The stopping criteria for all methods~are set as $\max\{\|\nabla\mathcal{L}_k\|, \tau_k^+\} \leq 10^{-4} \text{ OR } k \geq 10^4$. 

\begin{figure}[t!]
\centering
\subfigure[$\sigma^2=10^{-8}$]{\includegraphics[width=0.243\textwidth]{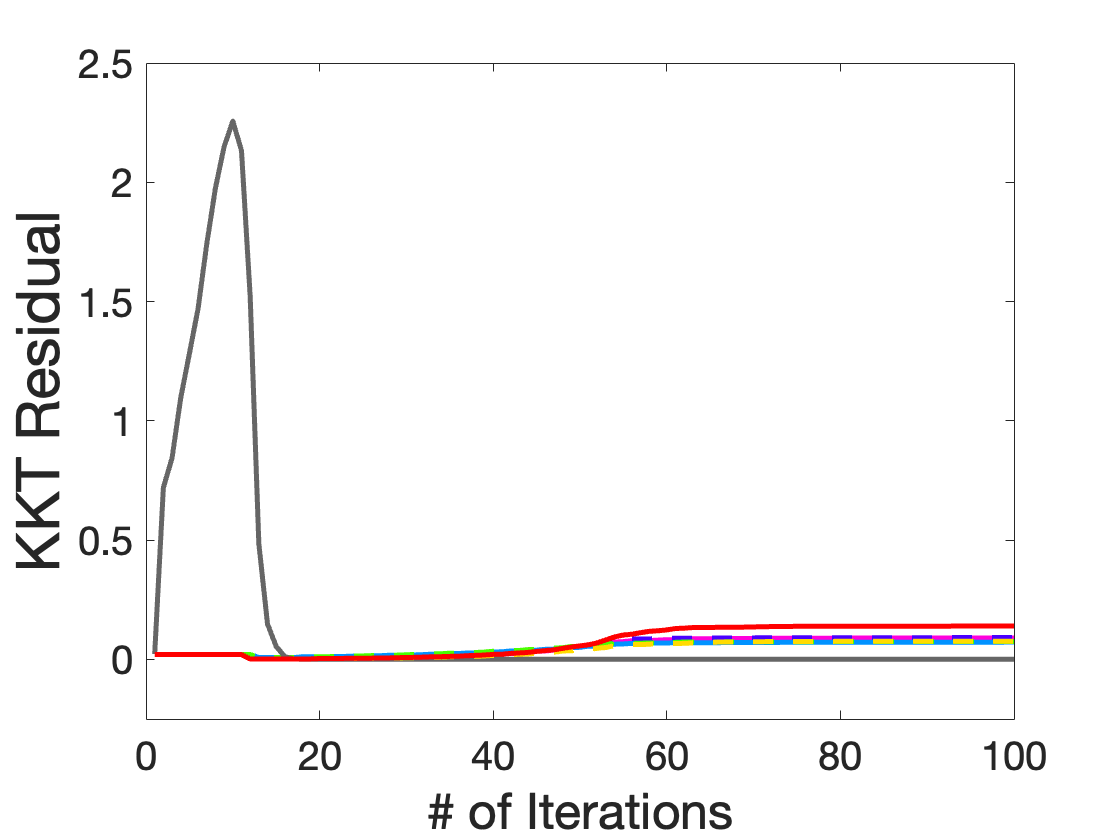}}
\subfigure[$\sigma^2=10^{-4}$]{\includegraphics[width=0.243\textwidth]{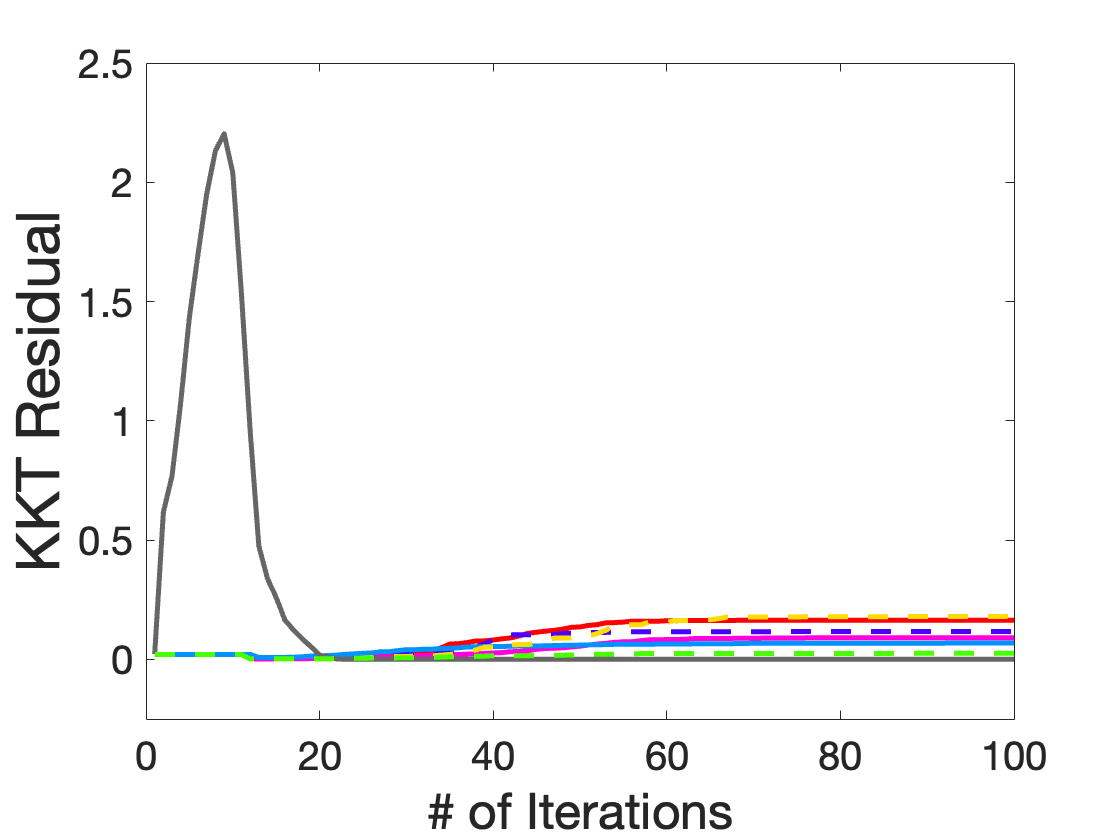}}
\subfigure[$\sigma^2=10^{-2}$]{\includegraphics[width=0.243\textwidth]{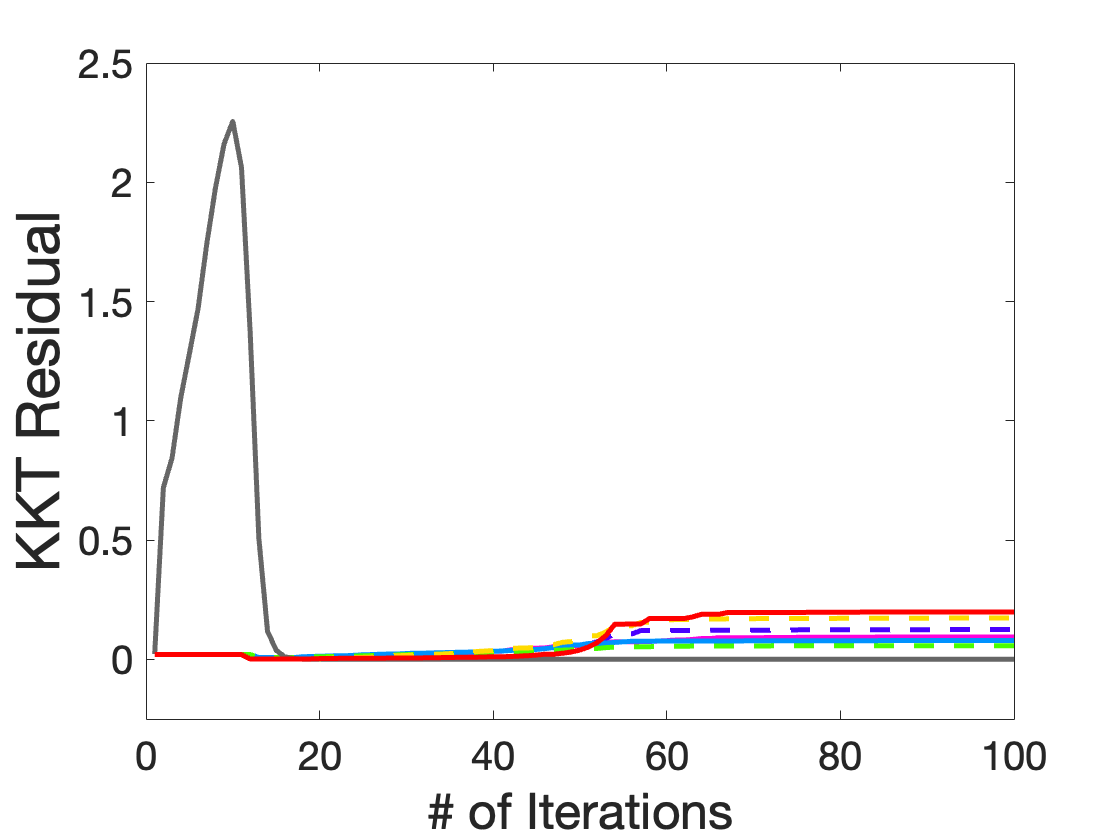}}
\subfigure[$\sigma^2=10^{-1}$]{\includegraphics[width=0.243\textwidth]{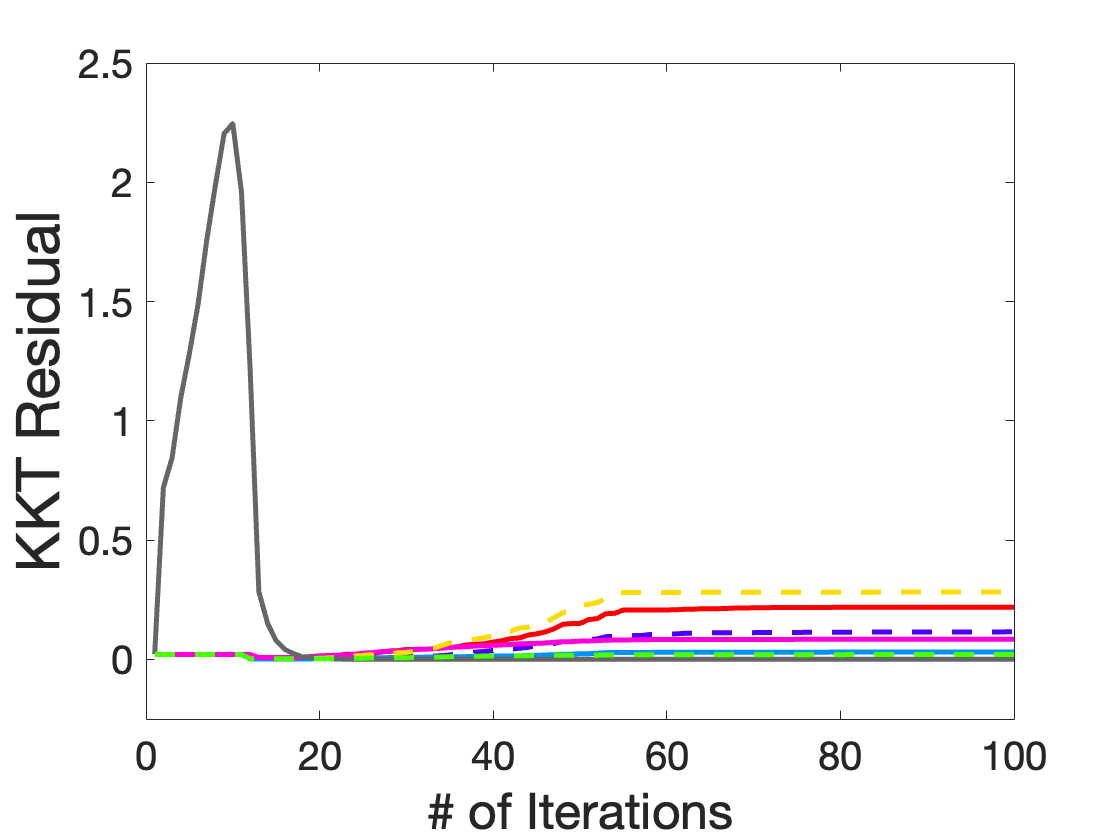}}
\subfigure[$\sigma^2=10^{-8}$]{\includegraphics[width=0.243\textwidth]{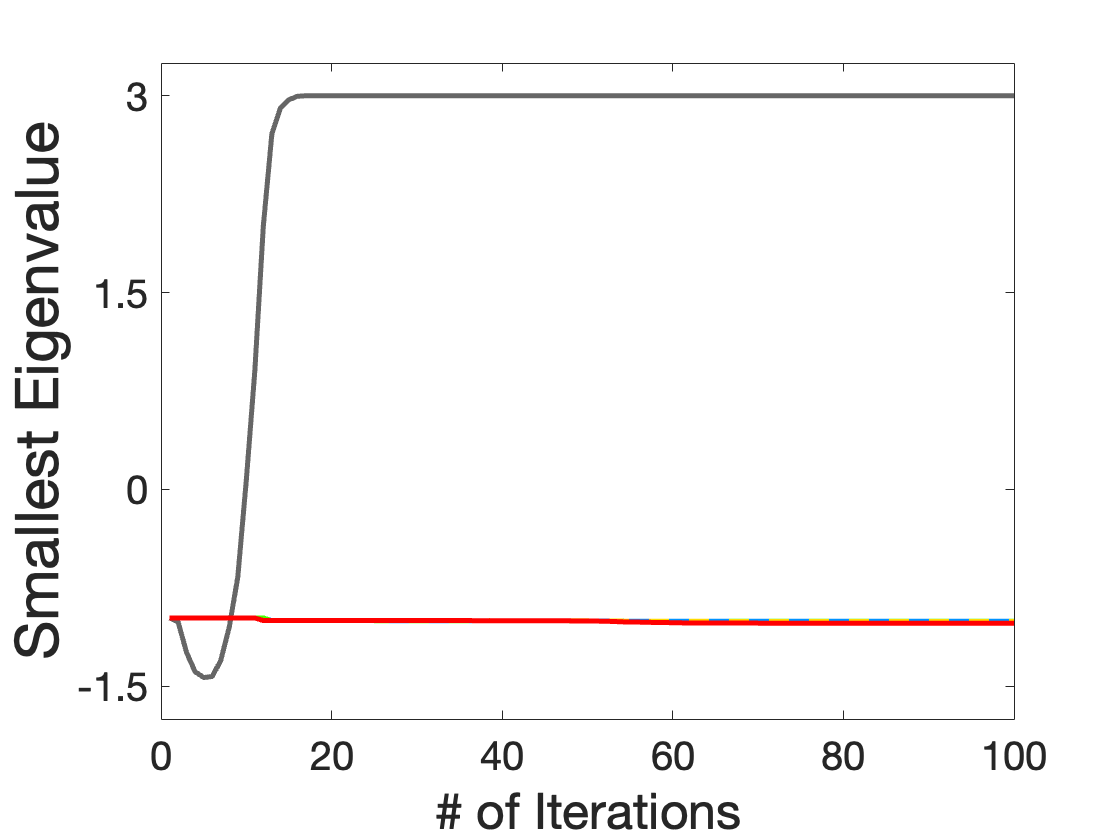}}
\subfigure[$\sigma^2=10^{-4}$]{\includegraphics[width=0.243\textwidth]{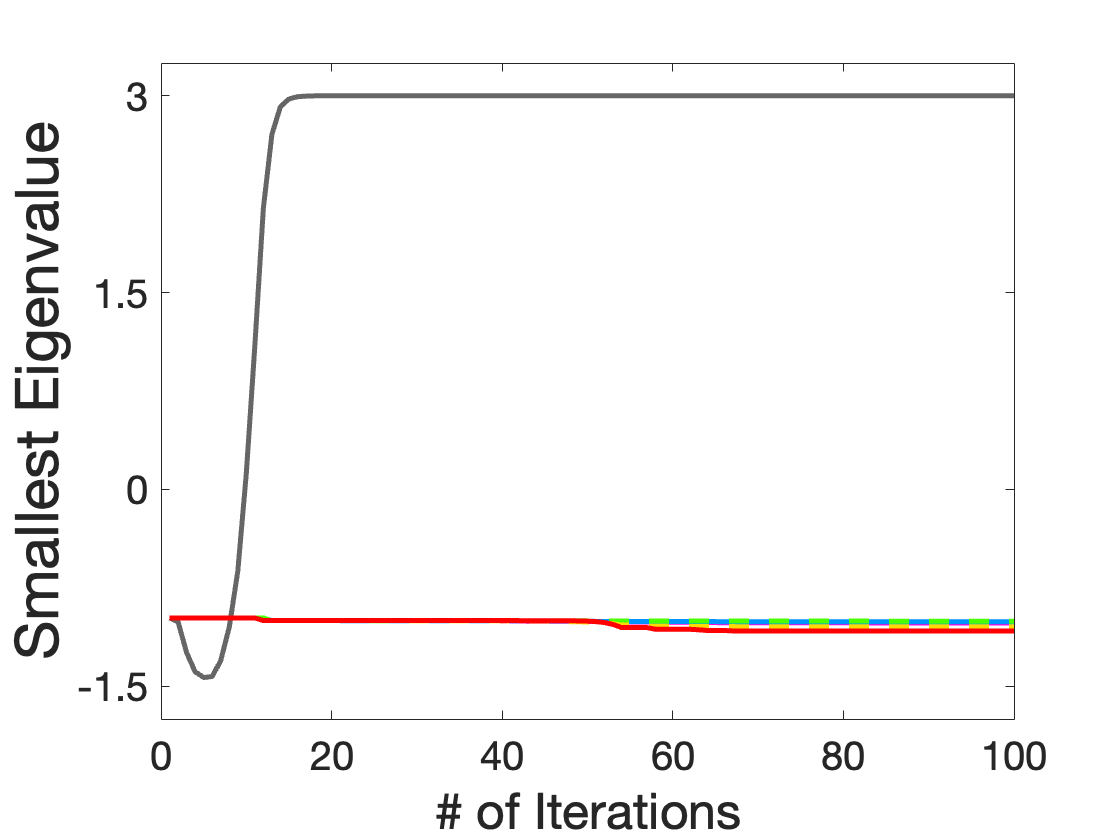}}
\subfigure[$\sigma^2=10^{-2}$]{\includegraphics[width=0.243\textwidth]{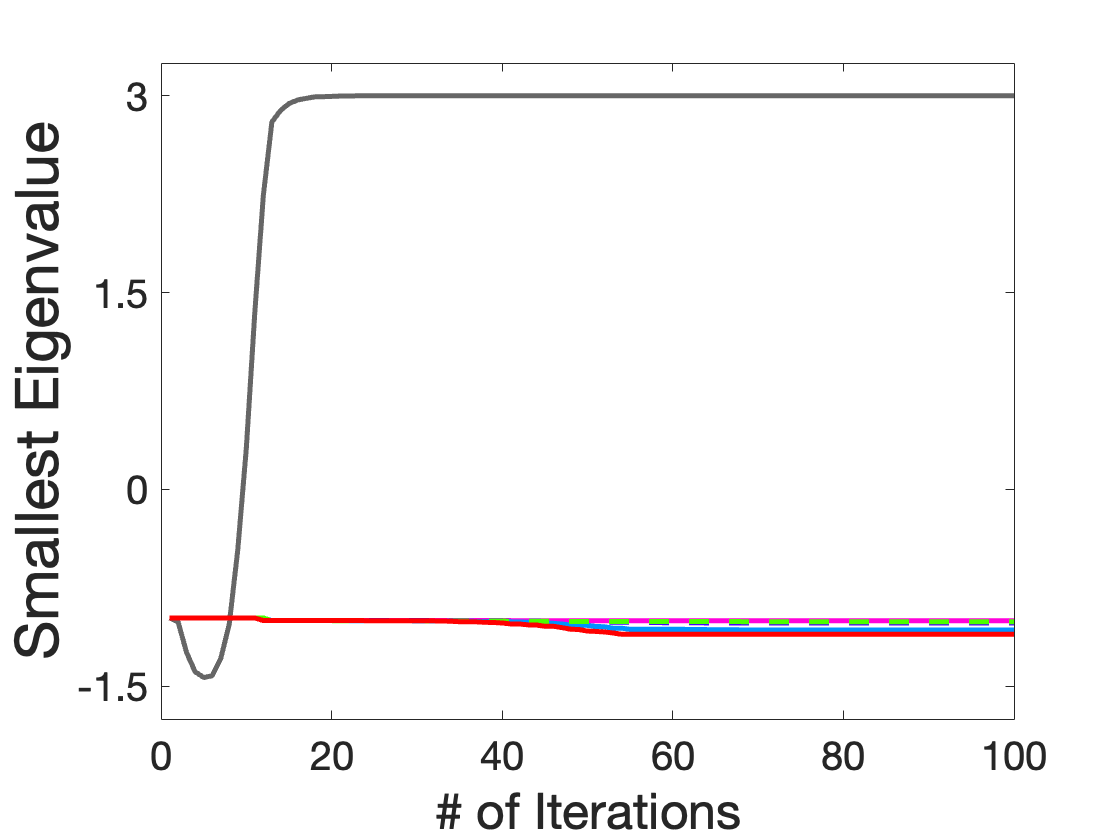}}
\subfigure[$\sigma^2=10^{-1}$]{\includegraphics[width=0.243\textwidth]{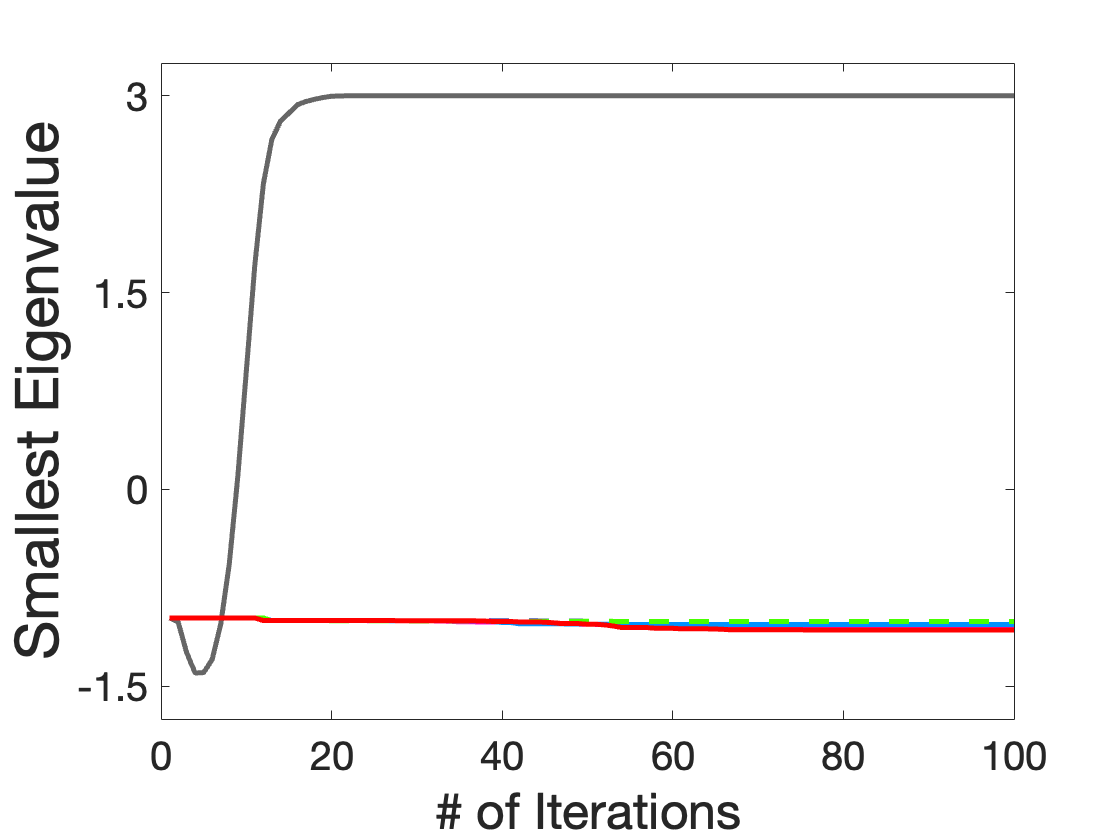}}
\includegraphics[width=0.85\textwidth]{Figures/label2.png}
\caption{Trajectories of the KKT residuals and the smallest eigenvalue of the reduced Lagrangian Hessians under four noise levels. The top four figures show the trajectories of the KKT residuals, while the bottom four figures show the trajectories of the smallest eigenvalues. Each figure corresponds~to~a noise level and includes seven lines representing the seven algorithms.}\label{fig:saddle}
\end{figure}

We present the trajectories of the KKT residuals in Figure \ref{fig:saddle}(a)-(d) and the trajectories~of~the~smallest eigenvalues in Figure \ref{fig:saddle}(e)-(h).
To better visualize the results, we plot only the first 100 iterations as both the KKT residuals and the smallest eigenvalues stabilize after this point.~From the two figures, we see that across all four noise levels, only TR-SQP-STORM2 successfully escapes the~\mbox{saddle}~point and converges to the local minima.
In contrast, for all other methods, the KKT residuals remain~relatively large and the smallest eigenvalues stay close to $-1$, which is precisely the negative curvature at the saddle point $(1,0)$. Thus, we conclude that the other methods are trapped near the saddle~point. Moreover, TR-SQP-STORM2 demonstrates a rapid escape, consistently terminating after around~20 iterations for different noise levels.

\section{Conclusion}\label{sec:6}

In this paper, we proposed a Trust-Region Sequential Quadratic Programming method called TR-SQP-STORM to find both first- and second-order stationary points for constrained stochastic problems. Our method utilizes a random model framework to represent the objective function.~At each iteration, a batch of samples is realized to estimate the objective quantities, with the batch size adaptively selected to ensure the estimators satisfy certain proper accuracy conditions with a fixed probability. We designed two types of trial steps, gradient steps and eigen steps, both of which are computed via a novel parameter-free decomposition of the step and the trust-region radius.
The gradient steps aim to reduce the KKT residuals to achieve first-order stationarity, while the eigen steps aim to explore the negative curvature of the reduced~\mbox{Lagrangian}~\mbox{Hessian}~to~achieve~\mbox{second-order}~\mbox{stationarity}. For the latter goal, we additionally computed second-order correction steps to overcome the~\mbox{potential}~Maratos effect, which occurs exculsively in constrained problems.~Under mild assumptions, we showed~global~almost sure convergence guarantees.~Numerical experiments on CUTEst benchmark problems,~constrained logistic regression problems, and saddle-point problems illustrate the promising performance of our method.

\bibliographystyle{my-plainnat}
\bibliography{ref}

\begin{thebibliography}{53}
\providecommand{\natexlab}[1]{#1}
\providecommand{\url}[1]{\texttt{#1}}
\expandafter\ifx\csname urlstyle\endcsname\relax
  \providecommand{\doi}[1]{doi: #1}\else
  \providecommand{\doi}{doi: \begingroup \urlstyle{rm}\Url}\fi

\bibitem[Achiam et~al.(2017)Achiam, Held, Tamar, and
  Abbeel]{Achiam2017Constrained}
J.~Achiam, D.~Held, A.~Tamar, and P.~Abbeel.
\newblock Constrained policy optimization.
\newblock In \emph{International conference on machine learning}, pages 22--31.
  PMLR, 2017.

\bibitem[Bandeira et~al.(2012)Bandeira, Scheinberg, and
  Vicente]{Bandeira2012Computation}
A.~S. Bandeira, K.~Scheinberg, and L.~N. Vicente.
\newblock Computation of sparse low degree interpolating polynomials and their
  application to derivative-free optimization.
\newblock \emph{Mathematical Programming}, 134\penalty0 (1):\penalty0 223--257,
  2012.

\bibitem[Bandeira et~al.(2014)Bandeira, Scheinberg, and
  Vicente]{Bandeira2014Convergence}
A.~S. Bandeira, K.~Scheinberg, and L.~N. Vicente.
\newblock Convergence of trust-region methods based on probabilistic models.
\newblock \emph{SIAM Journal on Optimization}, 24\penalty0 (3):\penalty0
  1238--1264, 2014.

\bibitem[Beiser et~al.(2023)Beiser, Keith, Urbainczyk, and
  Wohlmuth]{Beiser2023Adaptive}
F.~Beiser, B.~Keith, S.~Urbainczyk, and B.~Wohlmuth.
\newblock Adaptive sampling strategies for risk-averse stochastic optimization
  with constraints.
\newblock \emph{IMA Journal of Numerical Analysis}, 43\penalty0 (6):\penalty0
  3729--3765, 2023.

\bibitem[Berahas et~al.(2021)Berahas, Curtis, Robinson, and
  Zhou]{Berahas2021Sequential}
A.~S. Berahas, F.~E. Curtis, D.~Robinson, and B.~Zhou.
\newblock Sequential quadratic optimization for nonlinear equality constrained
  stochastic optimization.
\newblock \emph{{SIAM} Journal on Optimization}, 31\penalty0 (2):\penalty0
  1352--1379, 2021.

\bibitem[Berahas et~al.(2022)Berahas, Bollapragada, and
  Zhou]{Berahas2022Adaptive}
A.~S. Berahas, R.~Bollapragada, and B.~Zhou.
\newblock An adaptive sampling sequential quadratic programming method for
  equality constrained stochastic optimization.
\newblock \emph{arXiv preprint arXiv:2206.00712}, 2022.

\bibitem[Berahas et~al.(2023{\natexlab{a}})Berahas, Curtis, O’Neill, and
  Robinson]{Berahas2023Stochastic}
A.~S. Berahas, F.~E. Curtis, M.~J. O’Neill, and D.~P. Robinson.
\newblock A stochastic sequential quadratic optimization algorithm for
  nonlinear-equality-constrained optimization with rank-deficient jacobians.
\newblock \emph{Mathematics of Operations Research}, 2023{\natexlab{a}}.

\bibitem[Berahas et~al.(2023{\natexlab{b}})Berahas, Shi, Yi, and
  Zhou]{Berahas2023Accelerating}
A.~S. Berahas, J.~Shi, Z.~Yi, and B.~Zhou.
\newblock Accelerating stochastic sequential quadratic programming for equality
  constrained optimization using predictive variance reduction.
\newblock \emph{Computational Optimization and Applications}, 86\penalty0
  (1):\penalty0 79--116, 2023{\natexlab{b}}.

\bibitem[Berahas et~al.(2023{\natexlab{c}})Berahas, Xie, and
  Zhou]{Berahas2023Sequential}
A.~S. Berahas, M.~Xie, and B.~Zhou.
\newblock A sequential quadratic programming method with high probability
  complexity bounds for nonlinear equality constrained stochastic optimization.
\newblock \emph{arXiv preprint arXiv:2301.00477}, 2023{\natexlab{c}}.

\bibitem[Berahas et~al.(2024)Berahas, Bollapragada, and
  Shi]{Berahas2024Modified}
A.~S. Berahas, R.~Bollapragada, and J.~Shi.
\newblock Modified line search sequential quadratic methods for
  equality-constrained optimization with unified global and local convergence
  guarantees.
\newblock \emph{arXiv preprint arXiv:2406.11144}, 2024.

\bibitem[Bertsekas(1982)]{Bertsekas1982Constrained}
D.~P. Bertsekas.
\newblock \emph{Constrained Optimization and Lagrange Multiplier Methods}.
\newblock Elsevier, 1982.

\bibitem[Betts(2010)]{Betts2010Practical}
J.~T. Betts.
\newblock \emph{Practical Methods for Optimal Control and Estimation Using
  Nonlinear Programming}.
\newblock Society for Industrial and Applied Mathematics, 2010.

\bibitem[Blanchet et~al.(2019)Blanchet, Cartis, Menickelly, and
  Scheinberg]{Blanchet2019Convergence}
J.~Blanchet, C.~Cartis, M.~Menickelly, and K.~Scheinberg.
\newblock Convergence rate analysis of a stochastic trust-region method via
  supermartingales.
\newblock \emph{INFORMS Journal on Optimization}, 1\penalty0 (2):\penalty0
  92--119, 2019.

\bibitem[Boggs and Tolle(1995)]{Boggs1995Sequential}
P.~T. Boggs and J.~W. Tolle.
\newblock Sequential quadratic programming.
\newblock \emph{Acta Numerica}, 4:\penalty0 1--51, 1995.

\bibitem[Bollapragada et~al.(2023)Bollapragada, Karamanli, Keith, Lazarov,
  Petrides, and Wang]{Bollapragada2023adaptive}
R.~Bollapragada, C.~Karamanli, B.~Keith, B.~Lazarov, S.~Petrides, and J.~Wang.
\newblock An adaptive sampling augmented lagrangian method for stochastic
  optimization with deterministic constraints.
\newblock \emph{Computers \& Mathematics with Applications}, 149:\penalty0
  239--258, 2023.

\bibitem[Byrd et~al.(1987)Byrd, Schnabel, and Shultz]{Byrd1987Trust}
R.~H. Byrd, R.~B. Schnabel, and G.~A. Shultz.
\newblock A trust region algorithm for nonlinearly constrained optimization.
\newblock \emph{{SIAM} Journal on Numerical Analysis}, 24\penalty0
  (5):\penalty0 1152--1170, 1987.

\bibitem[Cao et~al.(2023)Cao, Berahas, and Scheinberg]{Cao2023First}
L.~Cao, A.~S. Berahas, and K.~Scheinberg.
\newblock First- and second-order high probability complexity bounds for
  trust-region methods with noisy oracles.
\newblock \emph{Mathematical Programming}, 207\penalty0 (1–2):\penalty0
  55--106, 2023.

\bibitem[Chen et~al.(2017)Chen, Menickelly, and Scheinberg]{Chen2017Stochastic}
R.~Chen, M.~Menickelly, and K.~Scheinberg.
\newblock Stochastic optimization using a trust-region method and random
  models.
\newblock \emph{Mathematical Programming}, 169\penalty0 (2):\penalty0 447--487,
  2017.

\bibitem[Choromanska et~al.(2015)Choromanska, Henaff, Mathieu, Arous, and
  LeCun]{Choromanska2015Loss}
A.~Choromanska, M.~Henaff, M.~Mathieu, G.~B. Arous, and Y.~LeCun.
\newblock The loss surfaces of multilayer networks.
\newblock In \emph{Artificial intelligence and statistics}, pages 192--204.
  PMLR, 2015.

\bibitem[Conn et~al.(2000)Conn, Gould, and Toint]{Conn2000Trust}
A.~R. Conn, N.~I.~M. Gould, and P.~L. Toint.
\newblock \emph{Trust Region Methods}.
\newblock Society for Industrial and Applied Mathematics, 2000.

\bibitem[Conn et~al.(2009{\natexlab{a}})Conn, Scheinberg, and
  Vicente]{Conn2009Introduction}
A.~R. Conn, K.~Scheinberg, and L.~N. Vicente.
\newblock \emph{Introduction to Derivative-Free Optimization}.
\newblock Society for Industrial and Applied Mathematics, 2009{\natexlab{a}}.

\bibitem[Conn et~al.(2009{\natexlab{b}})Conn, Scheinberg, and
  Vicente]{Conn2009Global}
A.~R. Conn, K.~Scheinberg, and L.~N. Vicente.
\newblock Global convergence of general derivative-free trust-region algorithms
  to first- and second-order critical points.
\newblock \emph{SIAM Journal on Optimization}, 20\penalty0 (1):\penalty0
  387--415, 2009{\natexlab{b}}.

\bibitem[Cuomo et~al.(2022)Cuomo, Di~Cola, Giampaolo, Rozza, Raissi, and
  Piccialli]{Cuomo2022Scientific}
S.~Cuomo, V.~S. Di~Cola, F.~Giampaolo, G.~Rozza, M.~Raissi, and F.~Piccialli.
\newblock Scientific machine learning through physics–informed neural
  networks: Where we are and what’s next.
\newblock \emph{Journal of Scientific Computing}, 92\penalty0 (3), 2022.

\bibitem[Curtis et~al.(2023{\natexlab{a}})Curtis, Jiang, and
  Wang]{Curtis2023Almost}
F.~E. Curtis, X.~Jiang, and Q.~Wang.
\newblock Almost-sure convergence of iterates and multipliers in stochastic
  sequential quadratic optimization.
\newblock \emph{arXiv preprint arXiv:2308.03687}, 2023{\natexlab{a}}.

\bibitem[Curtis et~al.(2023{\natexlab{b}})Curtis, Kungurtsev, Robinson, and
  Wang]{Curtis2023Stochastic}
F.~E. Curtis, V.~Kungurtsev, D.~P. Robinson, and Q.~Wang.
\newblock A stochastic-gradient-based interior-point algorithm for solving
  smooth bound-constrained optimization problems.
\newblock \emph{arXiv preprint arXiv:2304.14907}, 2023{\natexlab{b}}.

\bibitem[Curtis et~al.(2023{\natexlab{c}})Curtis, O’Neill, and
  Robinson]{Curtis2023Worst}
F.~E. Curtis, M.~J. O’Neill, and D.~P. Robinson.
\newblock Worst-case complexity of an sqp method for nonlinear equality
  constrained stochastic optimization.
\newblock \emph{Mathematical Programming}, 205\penalty0 (1–2):\penalty0
  431--483, 2023{\natexlab{c}}.

\bibitem[Curtis et~al.(2023{\natexlab{d}})Curtis, Robinson, and
  Zhou]{Curtis2023Sequential}
F.~E. Curtis, D.~P. Robinson, and B.~Zhou.
\newblock Sequential quadratic optimization for stochastic optimization with
  deterministic nonlinear inequality and equality constraints.
\newblock \emph{arXiv preprint arXiv:2302.14790}, 2023{\natexlab{d}}.

\bibitem[Curtis et~al.(2024)Curtis, Robinson, and Zhou]{Curtis2024Stochastic}
F.~E. Curtis, D.~P. Robinson, and B.~Zhou.
\newblock A stochastic inexact sequential quadratic optimization algorithm for
  nonlinear equality-constrained optimization.
\newblock \emph{INFORMS Journal on Optimization}, 2024.

\bibitem[Dauphin et~al.(2014)Dauphin, Pascanu, Gulcehre, Cho, Ganguli, and
  Bengio]{Dauphin2014Identifying}
Y.~N. Dauphin, R.~Pascanu, C.~Gulcehre, K.~Cho, S.~Ganguli, and Y.~Bengio.
\newblock Identifying and attacking the saddle point problem in
  high-dimensional non-convex optimization.
\newblock \emph{Advances in neural information processing systems}, 27, 2014.

\bibitem[El-Alem(1991)]{ElAlem1991Global}
M.~El-Alem.
\newblock A global convergence theory for the celis–dennis–tapia
  trust-region algorithm for constrained optimization.
\newblock \emph{SIAM Journal on Numerical Analysis}, 28\penalty0 (1):\penalty0
  266--290, 1991.

\bibitem[Fang et~al.(2024)Fang, Na, Mahoney, and Kolar]{Fang2024Fully}
Y.~Fang, S.~Na, M.~W. Mahoney, and M.~Kolar.
\newblock Fully stochastic trust-region sequential quadratic programming for
  equality-constrained optimization problems.
\newblock \emph{SIAM Journal on Optimization}, 34\penalty0 (2):\penalty0
  2007--2037, 2024.

\bibitem[Gould et~al.(2014)Gould, Orban, and Toint]{Gould2014CUTEst}
N.~I.~M. Gould, D.~Orban, and P.~L. Toint.
\newblock Cutest: a constrained and unconstrained testing environment with safe
  threads for mathematical optimization.
\newblock \emph{Computational Optimization and Applications}, 60\penalty0
  (3):\penalty0 545--557, 2014.

\bibitem[Hall and Heyde(2014)]{Hall2014Martingale}
P.~Hall and C.~C. Heyde.
\newblock \emph{Martingale Limit Theory and its Application}.
\newblock Academic press, 2014.

\bibitem[Heinkenschloss and Ridzal(2014)]{Heinkenschloss2014Matrix}
M.~Heinkenschloss and D.~Ridzal.
\newblock A matrix-free trust-region sqp method for equality constrained
  optimization.
\newblock \emph{SIAM Journal on Optimization}, 24\penalty0 (3):\penalty0
  1507--1541, 2014.

\bibitem[Jin et~al.(2024)Jin, Scheinberg, and Xie]{Jin2024Sample}
B.~Jin, K.~Scheinberg, and M.~Xie.
\newblock Sample complexity analysis for adaptive optimization algorithms with
  stochastic oracles.
\newblock \emph{Mathematical Programming}, 2024.

\bibitem[Khalfan et~al.(1993)Khalfan, Byrd, and
  Schnabel]{Khalfan1993Theoretical}
H.~F. Khalfan, R.~H. Byrd, and R.~B. Schnabel.
\newblock A theoretical and experimental study of the symmetric rank-one
  update.
\newblock \emph{SIAM Journal on Optimization}, 3\penalty0 (1):\penalty0 1--24,
  1993.

\bibitem[Kuang et~al.(2023)Kuang, Na, and Anitescu]{Kuang2023Online}
W.~Kuang, S.~Na, and M.~Anitescu.
\newblock Online covariance matrix estimation in stochastic inexact newton
  methods.
\newblock In \emph{OPT 2023: Optimization for Machine Learning}, 2023.

\bibitem[Lou et~al.(2024)Lou, Sun, and Nocedal]{Lou2024Noise}
Y.~Lou, S.~Sun, and J.~Nocedal.
\newblock Noise-tolerant optimization methods for the solution of a robust
  design problem.
\newblock \emph{arXiv preprint arXiv:2401.15007}, 2024.

\bibitem[Lu et~al.(2024)Lu, Mei, and Xiao]{Lu2024Variance}
Z.~Lu, S.~Mei, and Y.~Xiao.
\newblock Variance-reduced first-order methods for deterministically
  constrained stochastic nonconvex optimization with strong convergence
  guarantees.
\newblock \emph{arXiv preprint arXiv:2409.09906}, 2024.

\bibitem[Na and Mahoney(2022)]{Na2022Statistical}
S.~Na and M.~W. Mahoney.
\newblock Statistical inference of constrained stochastic optimization via
  sketched sequential quadratic programming.
\newblock \emph{arXiv preprint arXiv:2205.13687}, 2022.

\bibitem[Na et~al.(2022{\natexlab{a}})Na, Anitescu, and Kolar]{Na2022adaptive}
S.~Na, M.~Anitescu, and M.~Kolar.
\newblock An adaptive stochastic sequential quadratic programming with
  differentiable exact augmented lagrangians.
\newblock \emph{Mathematical Programming}, 199\penalty0 (1–2):\penalty0
  721--791, 2022{\natexlab{a}}.

\bibitem[Na et~al.(2022{\natexlab{b}})Na, Dereziński, and
  Mahoney]{Na2022Hessian}
S.~Na, M.~Dereziński, and M.~W. Mahoney.
\newblock Hessian averaging in stochastic newton methods achieves superlinear
  convergence.
\newblock \emph{Mathematical Programming}, 201\penalty0 (1–2):\penalty0
  473--520, 2022{\natexlab{b}}.

\bibitem[Na et~al.(2023)Na, Anitescu, and Kolar]{Na2023Inequality}
S.~Na, M.~Anitescu, and M.~Kolar.
\newblock Inequality constrained stochastic nonlinear optimization via
  active-set sequential quadratic programming.
\newblock \emph{Mathematical Programming}, 202\penalty0 (1–2):\penalty0
  279--353, 2023.

\bibitem[Nocedal and Wright(2006)]{Nocedal2006Numerical}
J.~Nocedal and S.~Wright.
\newblock \emph{Numerical Optimization}.
\newblock Springer New York, 2006.

\bibitem[Omojokun(1989)]{Omojokun1989Trust}
E.~O. Omojokun.
\newblock \emph{Trust region algorithms for optimization with nonlinear
  equality and inequality constraints}.
\newblock PhD thesis, University of Colorado at Boulder, CO, 1989.

\bibitem[Oztoprak et~al.(2023)Oztoprak, Byrd, and
  Nocedal]{Oztoprak2023Constrained}
F.~Oztoprak, R.~Byrd, and J.~Nocedal.
\newblock Constrained optimization in the presence of noise.
\newblock \emph{SIAM Journal on Optimization}, 33\penalty0 (3):\penalty0
  2118--2136, 2023.

\bibitem[Powell and Yuan(1990)]{Powell1990trust}
M.~J.~D. Powell and Y.~Yuan.
\newblock A trust region algorithm for equality constrained optimization.
\newblock \emph{Mathematical Programming}, 49\penalty0 (1–3):\penalty0
  189--211, 1990.

\bibitem[Qiu and Kungurtsev(2023)]{Qiu2023sequential}
S.~Qiu and V.~Kungurtsev.
\newblock A sequential quadratic programming method for optimization with
  stochastic objective functions, deterministic inequality constraints and
  robust subproblems.
\newblock \emph{arXiv preprint arXiv:2302.07947}, 2023.

\bibitem[Santoso et~al.(2005)Santoso, Ahmed, Goetschalckx, and
  Shapiro]{Santoso2005stochastic}
T.~Santoso, S.~Ahmed, M.~Goetschalckx, and A.~Shapiro.
\newblock A stochastic programming approach for supply chain network design
  under uncertainty.
\newblock \emph{European Journal of Operational Research}, 167\penalty0
  (1):\penalty0 96--115, 2005.

\bibitem[Sun and Nocedal(2023)]{Sun2023trust}
S.~Sun and J.~Nocedal.
\newblock A trust region method for noisy unconstrained optimization.
\newblock \emph{Mathematical Programming}, 202\penalty0 (1–2):\penalty0
  445--472, 2023.

\bibitem[Tropp(2011)]{Tropp2011User}
J.~A. Tropp.
\newblock User-friendly tail bounds for sums of random matrices.
\newblock \emph{Foundations of Computational Mathematics}, 12\penalty0
  (4):\penalty0 389--434, 2011.

\bibitem[Wächter and Biegler(2005)]{Waechter2005Implementation}
A.~Wächter and L.~T. Biegler.
\newblock On the implementation of an interior-point filter line-search
  algorithm for large-scale nonlinear programming.
\newblock \emph{Mathematical Programming}, 106\penalty0 (1):\penalty0 25--57,
  2005.

\bibitem[Çakmak and Özekici(2005)]{Cakmak2005Portfolio}
U.~Çakmak and S.~Özekici.
\newblock Portfolio optimization in stochastic markets.
\newblock \emph{Mathematical Methods of Operations Research}, 63\penalty0
  (1):\penalty0 151--168, 2005.

\end{thebibliography}

\appendix
%\pagebreak
\numberwithin{equation}{section}
\numberwithin{theorem}{section}
\makeatletter
\def\@seccntformat#1{\appendixname\ \csname the#1\endcsname. \;}
\makeatother

\section{Proof of Lemma \ref{lemma: mu_stabilize}}\label{append:1}

It suffices to show that \eqref{eq:threshold_Predk} is satisfied as long as $\barmu_k$ exceeds a deterministic threshold independent of $k$. We divide our analysis into two cases, depending on whether the gradient~step~or~the~eigen~step~is~taken.

\noindent $\bullet$ \textbf{Case 1: gradient step is taken.} By the algorithm design, a gradient step is taken if~and~only~if~\eqref{eq:cri_step_comput} holds. Thus, we only need to show
\begin{equation}\label{nequ:6}
\text{Pred}_k \leq -\frac{\kappa_{fcd}}{2}\|\bar{\nabla}\L_k\|\min\left\{\Delta_k,\frac{\|\bar{\nabla}\L_k\|}{\|\barH_k\|}\right\}
\end{equation}
when $\barmu_k$ is sufficiently large. Since $\|c_k+G_k\Delta\bx_k\|-\|c_k\|=-\bargamma_k\|c_k\|$, we have
\begin{align}\label{eq:pred_part1}
\text{Pred}_k & \stackrel{\mathclap{\eqref{def:Pred_k}}}{=} \barg_k^T\Delta\bx_k+\frac{1}{2}\Delta\bx_k^T\barH_k\Delta\bx_k+\barmu_k(\|c_k+G_k\Delta\bx_k\|-\|c_k\|) \notag\\
& = (\barg_k+\bargamma_k\barH_k\bv_k)^TZ_k\bu_k+\frac{1}{2}\bu_k^TZ_k^T\barH_kZ_k\bu_k+\bargamma_k\barg_k^T\bv_k +\frac{1}{2}\bargamma_k^2\bv_k^T\barH_k\bv_k-\barmu_k\bargamma_k\|c_k\| \notag\\
& \leq -\frac{\kappa_{fcd}}{2}\|Z_k^T(\barg_k+\bargamma_k\barH_k\bv_k)\|\min\left\{\tilde{\Delta}_k,\frac{\|Z_k^T(\barg_k+\bargamma_k\barH_k\bv_k)\|}{\|\barH_k\|}\right\}+\bargamma_k\|\barg_k\|\|\bv_k\| \notag\\
&\quad +\frac{1}{2}\bargamma_k\|\barH_k\|\|\bv_k\|^2-\barmu_k\bargamma_k\|c_k\|,
\end{align}	
where the inequality is due to \eqref{eq:cauchy1} and $\bargamma_k\leq 1$. By $\|Z_k^T(\barg_k+\bargamma_k\barH_k\bv_k)\|\geq \|Z_k^T\barg_k\|- \bargamma_k\|\barH_k\|\|\bv_k\|$,~we~have
\begin{align}\label{eq:pred_part2}
\|Z_k^T& (\barg_k +\bargamma_k\barH_k\bv_k)\|\min\left\{\tilde{\Delta}_k,\frac{\|Z_k^T(\barg_k+\bargamma_k\barH_k\bv_k)\|}{\|\barH_k\|}\right\} \notag \\
& \geq \|Z_k^T\barg_k\|\min\left\{\tilde{\Delta}_k,\frac{\|Z_k^T\barg_k\|}{\|\barH_k\|}-\bargamma_k\|\bv_k\|\right\} -\bargamma_k\|\barH_k\|\|\bv_k\|\min\left\{\tilde{\Delta}_k,\frac{\|Z_k^T\barg_k\|}{\|\barH_k\|}-\bargamma_k\|\bv_k\|\right\} \notag \\
& \geq \|Z_k^T\barg_k\|\min\left\{\tilde{\Delta}_k,\frac{\|Z_k^T\barg_k\|}{\|\barH_k\|}\right\}-\bargamma_k\|Z_k^T\barg_k\|\|\bv_k\|-\bargamma_k\|\barH_k\|\|\bv_k\|\tilde{\Delta}_k.
\end{align}
Combining \eqref{eq:pred_part1}, \eqref{eq:pred_part2}, the fact $\|\bv_k\|\leq\frac{1}{\sqrt{\kappa_{1,G}}}\|c_k\|$, and Assumption \ref{ass:1-1}, we obtain
\begin{align*}
\text{Pred}_k & \leq -\frac{\kappa_{fcd}}{2}\|Z_k^T\barg_k\|\min\left\{\tilde{\Delta}_k,\frac{\|Z_k^T\barg_k\|}{\|\barH_k\|}\right\}+\frac{\kappa_{fcd}}{2\sqrt{\kappa_{1,G}}}\bargamma_k\|Z_k^T\barg_k\|\|c_k\|+\frac{\kappa_{fcd}}{2\sqrt{\kappa_{1,G}}}\bargamma_k\|\barH_k\|\|c_k\|\tilde{\Delta}_k\\
&\quad +\frac{1}{\sqrt{\kappa_{1,G}}}\bargamma_k\|\barg_k\|\|c_k\| +\frac{\kappa_c}{2\kappa_{1,G}}\bargamma_k\|\barH_k\|\|c_k\|-\barmu_k\bargamma_k\|c_k\|\\
& \leq -\frac{\kappa_{fcd}}{2}\|Z_k^T\barg_k\|\min\left\{\tilde{\Delta}_k,\frac{\|Z_k^T\barg_k\|}{\|\barH_k\|}\right\}+\left(\frac{\Delta_{\max}}{2\sqrt{\kappa_{1,G}}}+\frac{\kappa_c}{2\kappa_{1,G}}\right)\bargamma_k\|\barH_k\|\|c_k\|\\
&\quad +\frac{1.5}{\sqrt{\kappa_{1,G}}}\bargamma_k\|\barg_k\|\|c_k\| -\barmu_k\bargamma_k\|c_k\|,
\end{align*}
where the second inequality uses $\|Z_k^T\barg_k\|\leq\|\barg_k\|$, $\tilde{\Delta}_k\leq\Delta_{\max}$, and $\kappa_{fcd}\leq 1$. 
By Assumptions \ref{ass:1-1} and \ref{ass:bdd_err}, we know $\|\barg_k\| \leq \|\barg_k- g_k\| + \|g_k\| \leq M + \kappa_{\nabla f}$. Noting that $\|Z_k^T\barg_k\|=\|\bar{\nabla}_{\bx}\L_k\|$~and~$\|\barH_k\|\leq \kappa_B$ (cf. Assumption \ref{ass:bdd_err}), we further have
\begin{equation}\label{nequ:5}
\text{Pred}_k \leq -\frac{\kappa_{fcd}}{2}\|\bar{\nabla}_{\bx}\L_k\|\min\left\{\tilde{\Delta}_k,\frac{\|\bar{\nabla}_{\bx}\L_k\|}{\|\barH_k\|}\right\} + 
\cbr{\frac{\Delta_{\max}\kappa_B}{2\sqrt{\kappa_{1,G}}}+\frac{\kappa_c\kappa_B}{2\kappa_{1,G}} + \frac{1.5(M+\kappa_{\nabla f})}{\sqrt{\kappa_{1,G}}}  - \barmu_k
}\bargamma_k\|c_k\|.
\end{equation}
\noindent$\bullet$ \textbf{Case 1a: $\tilde{\Delta}_k\leq \|\bar{\nabla}_{\bx}\L_k\|/\|\barH_k\|$.} We note that
\begin{multline*}
-\frac{\kappa_{fcd}}{2}\|\bar{\nabla}_{\bx}\L_k\|\tilde{\Delta}_k -\frac{\kappa_{fcd}}{2}\|\bar{\nabla}_{\bx}\L_k\|\breve{\Delta}_k-\frac{\kappa_{fcd}}{2}\|c_k\|\Delta_k 
\leq -\frac{\kappa_{fcd}}{2}\|\bar{\nabla}_{\bx}\L_k\|\Delta_k-\frac{\kappa_{fcd}}{2}\|c_k\|\Delta_k \\
\leq -\frac{\kappa_{fcd}}{2}\|\bar{\nabla}\L_k\|\Delta_k \leq -\frac{\kappa_{fcd}}{2}\|\bar{\nabla}\L_k\|\min\left\{\Delta_k,\frac{\|\bar{\nabla}\L_k\|}{\|\barH_k\|}\right\}.
\end{multline*}
Therefore, we know from \eqref{nequ:5} and the above display that \eqref{nequ:6} holds as long as
\begin{equation}\label{mu:eq1}
\barmu_k\bargamma_k\|c_k\| \geq \frac{\kappa_{fcd}}{2}\|\bar{\nabla}_{\bx}\L_k\|\breve{\Delta}_k+\frac{\kappa_{fcd}}{2}\|c_k\|\Delta_k + \cbr{\frac{\Delta_{\max}\kappa_B}{2\sqrt{\kappa_{1,G}}}+\frac{\kappa_c\kappa_B}{2\kappa_{1,G}} + \frac{1.5(M+\kappa_{\nabla f})}{\sqrt{\kappa_{1,G}}} }\bargamma_k\|c_k\|.
\end{equation}
When $\bargamma_k=1$, \eqref{mu:eq1} holds as long as
\begin{equation}\label{mu:eq2}
\barmu_k \geq \frac{\kappa_{fcd}\|\bar{\nabla}_{\bx}\L_k\|}{2\|c_k\|}\breve{\Delta}_k+\frac{\kappa_{fcd}}{2}\Delta_k 
+ \frac{\Delta_{\max}\kappa_B}{2\sqrt{\kappa_{1,G}}}+\frac{\kappa_c\kappa_B}{2\kappa_{1,G}} +\frac{1.5(M+\kappa_{\nabla f})}{\sqrt{\kappa_{1,G}}}.
\end{equation}
Noting that
\begin{equation*}
\breve{\Delta}_k  \stackrel{\eqref{eq:breve and tilde_delta_k}}{=} \frac{\|c_k^{RS}\|}{\|\bar{\nabla}\L_k^{RS}\|}\Delta_k \stackrel{\eqref{nequ:1}}{=} \frac{\|G_k\|^{-1}\|c_k\|}{\|(\|\barH_k\|^{-1}\bar{\nabla}_{\bx}\L_k,\|G_k\|^{-1}c_k)\|}\Delta_k\leq \frac{\|G_k\|^{-1}\|c_k\|}{\|\barH_k\|^{-1}\|\bar{\nabla}_{\bx}\L_k\|}\Delta_k,
\end{equation*}
we see \eqref{mu:eq2} is satisfied if
\begin{equation*}
\barmu_k \geq \frac{\kappa_{fcd}}{2}\frac{\|\barH_k\|}{\|G_k\|}\Delta_k+\frac{\kappa_{fcd}}{2}\Delta_k
+ \frac{\Delta_{\max}\kappa_B}{2\sqrt{\kappa_{1,G}}}+\frac{\kappa_c\kappa_B}{2\kappa_{1,G}} +\frac{1.5(M+\kappa_{\nabla f})}{\sqrt{\kappa_{1,G}}}.
\end{equation*}
Combining the above display with $\Delta_k\leq \Delta_{\max}$, $\|G_k\|\geq \sqrt{\kappa_{1,G}}$, $\|\barH_k\|\leq \kappa_B$, and $\kappa_{fcd}\leq 1$, we know \eqref{nequ:6}~holds~if
\begin{equation*}
\barmu_k \geq \rbr{\frac{\kappa_B}{\sqrt{\kappa_{1,G}}}+0.5}\Delta_{\max} + \frac{\kappa_c\kappa_B}{2\kappa_{1,G}} +\frac{1.5(M+\kappa_{\nabla f})}{\sqrt{\kappa_{1,G}}} \eqqcolon \hat{\mu}_1.
\end{equation*}
On the other hand, when $\bargamma_k=\breve{\Delta}_k/\|\bv_k\|$, \eqref{mu:eq1} holds provided that
\begin{align}\label{mu:eq4}
\barmu_k &\geq \frac{\kappa_{fcd}}{2}\frac{\|\bar{\nabla}_{\bx}\L_k\|\|\bv_k\|}{\|c_k\|}+\frac{\kappa_{fcd}}{2}\frac{\Delta_k\|\bv_k\|}{\breve{\Delta}_k}
+ \frac{\Delta_{\max}\kappa_B}{2\sqrt{\kappa_{1,G}}}+\frac{\kappa_c\kappa_B}{2\kappa_{1,G}} + \frac{1.5(M+\kappa_{\nabla f})}{\sqrt{\kappa_{1,G}}} \nonumber \\
& = \frac{\kappa_{fcd}}{2}\frac{\|\bar{\nabla}_{\bx}\L_k\|\|\bv_k\|}{\|c_k\|}+\frac{\kappa_{fcd}}{2}\frac{\|G_k\|\|\bv_k\|\|\bar{\nabla}\L_k^{RS}\|}{\|c_k\| }
+ \frac{\Delta_{\max}\kappa_B}{2\sqrt{\kappa_{1,G}}}+\frac{\kappa_c\kappa_B}{2\kappa_{1,G}} + \frac{1.5(M+\kappa_{\nabla f})}{\sqrt{\kappa_{1,G}}}.
\end{align}
Since $\kappa_{fcd}\leq 1$, $\|\bv_k\|\leq \|c_k\|/\sqrt{\kappa_{1,G}}$, $ \|\bar{\nabla}_{\bx}\L_k\| \leq \|\barg_k\| \leq M+\kappa_{\nabla f}$, and
\begin{align*}
\|\bar{\nabla}\L_k^{RS}\|\leq \max\{\|\barH_k\|^{-1},\|G_k\|^{-1}\}\|\bar{\nabla}\L_k\| & \leq \max\left\{\frac{1}{\kappa_B},\frac{1}{\sqrt{\kappa_{1,G}}}\right\}(\|\bar{\nabla}_{\bx}\L_k\|+\|c_k\|)\\
&\leq \max\left\{\frac{1}{\kappa_B},\frac{1}{\sqrt{\kappa_{1,G}}}\right\}(M + \kappa_{\nabla f} + \kappa_c ),
\end{align*}
we know \eqref{mu:eq4} is implied by 
\begin{equation*}
\barmu_k \geq  \frac{2(M+\kappa_{\nabla f})}{\sqrt{\kappa_{1,G}}} + \frac{\sqrt{\kappa_{2,G}}}{2\sqrt{\kappa_{1,G}} }\max\left\{\frac{1}{\kappa_B},\frac{1}{\sqrt{\kappa_{1,G}}}\right\}(M+\kappa_{\nabla f}+\kappa_c)
+ \frac{\Delta_{\max}\kappa_B}{2\sqrt{\kappa_{1,G}}}+\frac{\kappa_c\kappa_B}{2\kappa_{1,G}} \eqqcolon \hat{\mu}_2.
\end{equation*}
\noindent$\bullet$ \textbf{Case 1b: $\tilde{\Delta}_k>\|\bar{\nabla}_{\bx}\L_k\|/\|\barH_k\|$.} We note that
\begin{equation*}
-\frac{\kappa_{fcd}}{2}\frac{\|\bar{\nabla}_{\bx}\L_k\|^2}{\|\barH_k\|}-\frac{\kappa_{fcd}}{2}\frac{\|c_k\|^2}{\|\barH_k\|}  = -\frac{\kappa_{fcd}}{2}\frac{\|\bar{\nabla}\L_k\|^2}{\|\barH_k\|}  \leq -\frac{\kappa_{fcd}}{2}\|\bar{\nabla}\L_k\|\min\left\{\Delta_k,\frac{\|\bar{\nabla}\L_k\|}{\|\barH_k\|}\right\}.
\end{equation*}
Thus, \eqref{nequ:5} suggests that \eqref{nequ:6} holds as long as
\begin{align}\label{nequ:7}
\barmu_k\bargamma_k\|c_k\| & \geq \frac{\kappa_{fcd}}{2}\frac{\|c_k\|^2}{\|\barH_k\|}
+ \cbr{ \frac{\Delta_{\max}\kappa_B}{2\sqrt{\kappa_{1,G}}}+\frac{\kappa_c\kappa_B}{2\kappa_{1,G}} + \frac{1.5(M+\kappa_{\nabla f})}{\sqrt{\kappa_{1,G}}} }\bargamma_k\|c_k\| \nonumber\\
\Longleftrightarrow \barmu_k & \geq \frac{\kappa_{fcd}}{2}\frac{\|c_k\|}{\bargamma_k\|\barH_k\|}
+ \frac{\Delta_{\max}\kappa_B}{2\sqrt{\kappa_{1,G}}}+\frac{\kappa_c\kappa_B}{2\kappa_{1,G}} + \frac{1.5(M+\kappa_{\nabla f})}{\sqrt{\kappa_{1,G}}}.
\end{align}
By \eqref{eq:breve and tilde_delta_k}, $\tilde{\Delta}_k>\|\bar{\nabla}_{\bx}\L_k\|/\|\barH_k\| = \|\bar{\nabla}_{\bx}\L_k^{RS}\|$ implies $\Delta_k>\|\bar{\nabla}\L_k^{RS}\|$. Using $\breve{\Delta}_k=\|G_k\|^{-1}\|c_k\|/\|\bar{\nabla}\L_k^{RS}\|\cdot\Delta_k$ and $\|\bv_k\|\leq \|c_k\|/\sqrt{\kappa_{1,G}}$, we have
\begin{multline*}
\bargamma_k=\min\left\{\frac{\breve{\Delta}_k}{\|\bv_k\|},1\right\}=\min\left\{\frac{\|G_k\|^{-1}\|c_k\|\Delta_k}{\|\bar{\nabla}\L_k^{RS}\|\|\bv_k\|},1\right\} \\
\geq \min\left\{\frac{\|c_k\|}{\|G_k\|\|\bv_k\|},1\right\}\geq \min\left\{\sqrt{\frac{\kappa_{1,G}}{\kappa_{2,G}}},1\right\}=\sqrt{\frac{\kappa_{1,G}}{\kappa_{2,G}}}.
\end{multline*}
Thus, with $\kappa_{fcd}\leq 1$, \eqref{nequ:7} (and hence \eqref{nequ:6}) is implied by
\begin{equation}\label{mu:eq9}
\barmu_k \geq \frac{\kappa_c}{2\kappa_B}\sqrt{\frac{\kappa_{2,G}}{\kappa_{1,G}}}
+ \frac{\Delta_{\max}\kappa_B}{2\sqrt{\kappa_{1,G}}}+\frac{\kappa_c\kappa_B}{2\kappa_{1,G}} + \frac{1.5(M+\kappa_{\nabla f})}{\sqrt{\kappa_{1,G}}} \eqqcolon \hat{\mu}_3.
\end{equation}
\noindent$\bullet$ \textbf{Case 2: eigen step is taken.} We only need to show
\begin{equation}\label{nequ:8}
\text{Pred}_k \leq -\frac{\kappa_{fcd}}{2}\bartau_k^+\Delta_k^2-\frac{\kappa_{fcd}}{2}\bartau_k^+\|c_k\|\Delta_k
\end{equation}
when $\barmu_k$ is sufficiently large. By the definition of the eigen step, we have
\begin{align*}	
\text{Pred}_k & \stackrel{\mathclap{\eqref{def:Pred_k}}}{=} \barg_k^T\Delta\bx_k+\frac{1}{2}\Delta\bx_k^T\barH_k\Delta\bx_k+\barmu_k(\|c_k+G_k\Delta\bx_k\|-\|c_k\|)\\
& = (\barg_k+\bargamma_k\barH_k\bv_k)^TZ_k\bu_k+\frac{1}{2}\bu_k^TZ_k^T\barH_kZ_k\bu_k+\bargamma_k\barg_k^T\bv_k +\frac{1}{2}\bargamma_k^2\bv_k^T\barH_k\bv_k-\barmu_k\bargamma_k\|c_k\|\\		
& \stackrel{\mathclap{\eqref{eq:reduction_eigen_step}}}{\leq} -\frac{\kappa_{fcd}}{2}\bartau_k^+\tilde{\Delta}_k^2+\bargamma_k\|\barg_k\|\|\bv_k\| +\frac{1}{2}\bargamma_k\|\barH_k\|\|\bv_k\|^2-\barmu_k\bargamma_k\|c_k\|.
\end{align*}
By Assumptions \ref{ass:1-1}, \ref{ass:bdd_err} and $\|\bv_k\|\leq \|c_k\|/\sqrt{\kappa_{1,G}}$, we further have
\begin{align*}
\text{Pred}_k & \leq -\frac{\kappa_{fcd}}{2}\bartau_k^+\tilde{\Delta}_k^2 
+\frac{M+\kappa_{\nabla f}}{\sqrt{\kappa_{1,G}}}\bargamma_k\|c_k\| + \frac{\kappa_c\kappa_B}{2\kappa_{1,G}}\bargamma_k\|c_k\| -\barmu_k\bargamma_k\|c_k\|\\
& = -\frac{\kappa_{fcd}}{2}\bartau_k^+\Delta_k^2 -\frac{\kappa_{fcd}}{2}\bartau_k^+\|c_k\|\Delta_k +\frac{\kappa_{fcd}}{2}\bartau_k^+\breve{\Delta}_k^2 +\frac{\kappa_{fcd}}{2}\bartau_k^+\|c_k\|\Delta_k \\
& \quad + \cbr{\frac{M+\kappa_{\nabla f}}{\sqrt{\kappa_{1,G}}} + \frac{\kappa_c\kappa_B}{2\kappa_{1,G}} - \barmu_k} \bargamma_k\|c_k\|. 
\end{align*}
Thus, \eqref{nequ:8} holds if
\begin{align}\label{eq:merit_4}
\barmu_k \bargamma_k & \|c_k\|  \geq \frac{\kappa_{fcd}}{2}\bartau_k^+\breve{\Delta}_k^2 + \frac{\kappa_{fcd}}{2}\bartau_k^+\|c_k\|\Delta_k + \cbr{\frac{M+\kappa_{\nabla f}}{\sqrt{\kappa_{1,G}}} + \frac{\kappa_c\kappa_B}{2\kappa_{1,G}}} \bargamma_k\|c_k\| \nonumber\\
& \stackrel{\mathclap{\eqref{eq:breve and tilde_delta_k_2}}}{=} 
\frac{\kappa_{fcd}}{2}\frac{\bartau_k^+ \|c_k^{RS}\|^2}{(\bartau_k^{RS+})^2+\|c_k^{RS}\|^2}\Delta_k^2 + \frac{\kappa_{fcd}}{2}\bartau_k^+\|c_k\|\Delta_k + \cbr{\frac{M+\kappa_{\nabla f}}{\sqrt{\kappa_{1,G}}} + \frac{\kappa_c\kappa_B}{2\kappa_{1,G}}} \bargamma_k\|c_k\|.
\end{align}
When $\bargamma_k=1$, by Young's inequality, $\bartau_k^+\leq\|\barH_k\|\leq\kappa_B$, and $\Delta_k\leq\Delta_{\max}$, we know \eqref{eq:merit_4} is~implied~by
\begin{align*}
\barmu_k & \geq 
\frac{\kappa_{fcd}}{4}\frac{\bartau_k^+\|c_k^{RS}\|}{\bartau_k^{RS+}\|c_k\|}\Delta_{\max}^2 
+ \frac{\kappa_{fcd}}{2}\kappa_B\Delta_{\max}
+ \frac{M+\kappa_{\nabla f}}{\sqrt{\kappa_{1,G}}} + \frac{\kappa_c\kappa_B}{2\kappa_{1,G}}\\
\Longleftarrow \barmu_k & \geq 
\frac{\kappa_{fcd}}{4}\frac{\kappa_B}{\sqrt{\kappa_{1,G}}}\Delta_{\max}^2 
+ \frac{\kappa_{fcd}}{2}\kappa_B\Delta_{\max}
+ \frac{M+\kappa_{\nabla f}}{\sqrt{\kappa_{1,G}}} + \frac{\kappa_c\kappa_B}{2\kappa_{1,G}} \\
\Longleftarrow \barmu_k & \geq 
\frac{\kappa_B}{4\sqrt{\kappa_{1,G}}}\Delta_{\max}^2 
+ \frac{\kappa_B\Delta_{\max}}{2}
+ \frac{M+\kappa_{\nabla f}}{\sqrt{\kappa_{1,G}}} + \frac{\kappa_c\kappa_B}{2\kappa_{1,G}} 
\eqqcolon \hat{\mu}_4.
\end{align*}
When $\bargamma_k<1$, without loss of generality, we suppose $\|c_k\|\neq 0$ (otherwise \eqref{eq:merit_4} is trivial).~From~\eqref{eq:breve and tilde_delta_k_2}~and $\bargamma_k\|\bv_k\|=\breve{\Delta}_k$, we have
\begin{equation*}
\Delta_k=\frac{\bargamma_k\|\bv_k\|\|G_k\|}{\|c_k\|}\|(\bartau_k^{RS+},c_k^{RS})\|.
\end{equation*}
Thus, \eqref{eq:merit_4} is equivalent to
\begin{equation*}
\barmu_k \bargamma_k\|c_k\| \geq \frac{\kappa_{fcd}}{2}\bartau_k^+\bargamma_k^2\|\bv_k\|^2 + \frac{\kappa_{fcd}}{2}\bartau_k^+\bargamma_k\|\bv_k\|\|G_k\|\|(\bartau_k^{RS+},c_k^{RS})\| + \cbr{\frac{M+\kappa_{\nabla f}}{\sqrt{\kappa_{1,G}}} + \frac{\kappa_c\kappa_B}{2\kappa_{1,G}}} \bargamma_k\|c_k\|.
\end{equation*}
Since $\max\{\bargamma_k,\kappa_{fcd}\}\leq 1$, we only need $\barmu_k$ to satisfy
\begin{equation*}
\barmu_k \geq\frac{\bartau_k^+\|\bv_k\|^2}{2\|c_k\|} + \frac{\bartau_k^+\|\bv_k\|}{2\|c_k\|}   \|G_k\|\|(\bartau_k^{RS+},c_k^{RS})\| + \frac{M+\kappa_{\nabla f}}{\sqrt{\kappa_{1,G}}} + \frac{\kappa_c\kappa_B}{2\kappa_{1,G}}.
\end{equation*}
Using $\bartau_k^+\leq \|\barH_k\|\leq \kappa_B$ and $\|\bv_k\|\leq \|c_k\|/\sqrt{\kappa_{1,G}}$, and applying Assumptions \ref{ass:1-1} and \ref{ass:bdd_err}, we know the above display can be further implied by 
\begin{equation*}
\barmu_k \geq \frac{\kappa_c\kappa_B}{\kappa_{1,G}} + \frac{\kappa_B\sqrt{\kappa_{2,G}}}{2\sqrt{\kappa_{1,G}}}\left(1+\frac{\kappa_c}{\sqrt{\kappa_{1,G}}}\right) +\frac{M+\kappa_{\nabla f}}{\sqrt{\kappa_{1,G}}} \eqqcolon \hat{\mu}_5.
\end{equation*}	
Combining all the results above by defining $\tilde{\mu} = \max\{\hat{\mu}_1,\ldots,\hat{\mu}_5\}$, we know \eqref{eq:threshold_Predk} is satisfied if $\barmu_k\geq\tilde{\mu}$. Since $\barmu_k$ is increased by a factor of $\rho$ for each update, this result suggests that $\barmu_k\leq \rho\tilde{\mu}\eqqcolon\hatmu$. This completes the proof.

\end{document}